\documentclass[fleqn]{article}

\usepackage[english]{babel}
\usepackage{amsmath}
\usepackage{amsfonts}
\usepackage{amsthm}

\usepackage[utf8]{inputenc}
\usepackage{mathrsfs}
\usepackage{paralist}

\usepackage{subfigure}
\usepackage{pgf}
\usepackage[numbers]{natbib}

\numberwithin{equation}{section}

\theoremstyle{plain}
\newtheorem{theorem}{Theorem}[section]
\newtheorem{proposition}[theorem]{Proposition}
\newtheorem{lemma}[theorem]{Lemma}
\newtheorem{corollary}[theorem]{Corollary}

\theoremstyle{definition}
\newtheorem{definition}[theorem]{Definition}

\clearpage{}

\newcommand{\C}{\mathbb{C}}
\newcommand{\R}{\mathbb{R}}

\newcommand{\cK}{\mathcal{K}}
\newcommand{\cL}{\mathcal{L}}

\newcommand{\cO}{\mathcal{O}}
\newcommand{\cP}{\mathcal{P}}

\newcommand{\abs}[1]{\vert #1 \vert}
\newcommand{\abslr}[1]{\left\vert #1 \right\vert}
\newcommand{\norm}[1]{\Vert #1 \Vert}

\DeclareMathOperator*{\argmin}{argmin}

\DeclareMathOperator{\interior}{int}
\DeclareMathOperator{\exterior}{ext}

\clearpage{}

\newcommand{\coloneq}{\mathrel{\mathop:}=}

\newcommand{\eop}{}

\title{Properties and examples of Faber--Walsh polynomials}

\author{Olivier S\`{e}te\footnotemark[2] \and J\"{o}rg Liesen\footnotemark[2]}

\begin{document}
\maketitle

\renewcommand{\thefootnote}{\fnsymbol{footnote}}

\footnotetext[2]{Institute of Mathematics, Technische Universit\"{a}t 
Berlin, Stra{\ss}e des 17. Juni 136, 10623 Berlin, Germany.
\texttt{\{sete,liesen\}@math.tu-berlin.de}}

\renewcommand{\thefootnote}{\arabic{footnote}}

\begin{abstract}
The Faber--Walsh polynomials are a direct generalization of the (classical)
Faber polynomials from simply connected sets to sets with several simply 
connected components.
In this paper we derive new properties of the Faber--Walsh polynomials, where 
we focus on results of interest in numerical linear algebra, and on the 
relation between the Faber--Walsh polynomials and the classical Faber and 
Chebyshev polynomials.
Moreover, we present examples of Faber--Walsh polynomials for two real 
intervals as well as for some non-real sets consisting of several simply 
connected components.

 \end{abstract}



\section{Introduction}

The (classical) Faber polynomials associated with a simply connected compact 
set $\Omega \subset \C$ have found numerous applications in numerical 
approximation~\cite{Ellacott1983,Ellacott1986}  and in particular in numerical 
linear
algebra~\cite{BeckermannReichel2009,Eiermann1989,HeuvelineSadkane1997,MoretNovati2001,MoretNovati2001a,StarkeVarga1993}.
The main idea behind their 
construction, originally due to Faber~\cite{Faber1903}, is to have a sequence 
of polynomials $F_0, F_1, F_2,\dots$, so that each analytic function on 
$\Omega$ can be expanded in a convergent series of the form $\sum_{j=0}^\infty 
\gamma_j F_j(z)$, where the polynomials depend only on the set $\Omega$.  The 
definition and practical computation of the Faber polynomials is based on the 
Riemann map from the exterior of $\Omega$ (in the extended complex plane) onto 
the exterior of the unit disk.  The Faber polynomials therefore exist for 
simply connected sets only.  For surveys of the theory of Faber polynomials we 
refer to~\cite{Curtiss1971,Suetin1998}.

In 1956, Walsh found a direct generalization of the Riemann mapping
theorem from simply connected domains to multiply connected domains, which now 
are mapped conformally onto a \emph{lemniscatic domain}~\cite{Walsh1956};
see Theorem~\ref{thm:conformal_map} below for a complete statement.  Further 
existence proofs of Walsh's \emph{lemniscatic map} were given by 
Grunsky~\cite{Grunsky1957,Grunsky1957a} 
and~\cite[Theorem~3.8.3]{Grunsky1978}, Jenkins~\cite{Jenkins1958} and
Landau~\cite{Landau1961}.
We recently studied lemniscatic maps in~\cite{SeteLiesen2015} and derived 
several explicit examples, which are among the first in the literature (see 
also the technical report~\cite{Sete2013}).

In a subsequent paper of 1958, Walsh used his lemniscatic map for
obtaining a generalization of the Faber polynomials from simply
connected sets to sets consisting of several simply connected 
components~\cite{Walsh1958}; see Theorem~\ref{thm:fw_poly} below for a complete 
statement.  While the literature on Faber polynomials is quite extensive, the
\emph{Faber--Walsh polynomials} have rarely been studied in the literature.
One notable exception is Suetin's book~\cite{Suetin1998}, which contains a 
proper subsection on the Faber--Walsh polynomials as well as a few
further references (see also the technical report~\cite{Sete2013}).

Clearly, the mathematical theory and practical applicability of the 
Faber--Walsh polynomials have not been fully explored yet.
Our goal in this paper is to contribute to a better understanding.
To this end we derive some new theoretical results on Faber--Walsh polynomials 
and give several analytic as well as numerically computed examples.  In our 
theoretical study we focus on results that are of interest in constructive 
approximation and numerical linear algebra applications, 
and on the relation between Faber--Walsh polynomials and the classical Faber as 
well as Chebyshev polynomials.  In our examples we consider sets consisting
of two real intervals, as well as non-real sets consisting of several 
components.  In particular, our numerical results demonstrate that the 
Faber--Walsh polynomials are computable for a wide range of sets
via a numerical conformal mapping technique for multiply connected domains 
introduced in~\cite{NLS2015_numconf}.

The paper is organized as follows.
In Section~\ref{sect:FW} we give a summary of Walsh's results, and recall the 
definition of Faber--Walsh polynomials.
We then derive general properties of Faber--Walsh polynomials in 
Section~\ref{sect:theory}.
In Section~\ref{sect:intervals} we consider Faber--Walsh polynomials for two 
real intervals and particularly study their relation to the classical Chebyshev 
polynomials.
In Section~\ref{sect:examples} we show numerical examples of Faber--Walsh 
polynomials for two different nonreal sets.

 
\section{The Faber--Walsh polynomials}
\label{sect:FW}

We first discuss Walsh's generalization of the Riemann mapping theorem.
For a given integer $N \geq 1$, let $a_1, \ldots, a_N \in \C$ be pairwise
distinct and let the positive real numbers $m_1, \ldots, m_N$ satisfy 
$\sum_{j=1}^N m_j = 1$.  Then for any $\mu > 0$ the set
\begin{equation}
\cL \coloneq \{ w \in \widehat{\C} : \abs{U(w)} > \mu \}, \quad \text{where}
\quad U(w) \coloneq \prod_{j=1}^N (w-a_j)^{m_j},
\label{eqn:lemniscate}
\end{equation}
is called a \emph{lemniscatic domain} in the extended complex plane 
$\widehat{\C} = \C \cup \{ \infty \}$. The following theorem of Walsh shows 
that lemniscatic domains are canonical domains for certain $N$-times connected 
domains (open and connected sets).

\begin{theorem}[{\cite[Theorems~3 and~4]{Walsh1956}}]
\label{thm:conformal_map}
Let $E \coloneq \cup_{j=1}^N E_j$, where $E_1, \ldots, E_N \subseteq \C$ are
mutually exterior simply connected compact sets (none a single point), and let 
$\cK \coloneq \widehat{\C} \backslash E$.  Then there exists a unique 
lemniscatic domain $\cL$ of the form~\eqref{eqn:lemniscate} with
$\mu>0$ equal to the logarithmic capacity of $E$, and a unique bijective 
conformal map
\begin{equation*}
\Phi : \cK \to \cL \quad \text{with} \quad \Phi(z) = z + \cO \left( \frac{1}{z} 
\right) \quad \text{for $z$ near infinity.}
\end{equation*}
We call the function $\Phi$ the \emph{lemniscatic map} of $\cK$ (or of $E$), 
and denote $\psi = \Phi^{-1}$.
\end{theorem}

For $N=1$ the set $E$ is simply connected and a lemniscatic domain is the
exterior of a disk.  Hence in this case Theorem~\ref{thm:conformal_map} 
is equivalent with the Riemann mapping theorem.
In~\cite{SeteLiesen2015} we studied the properties of lemniscatic maps and 
derived several analytic examples.  In the subsequent 
paper~\cite{NLS2015_numconf}, written jointly with Nasser, we presented a 
numerical method for computing lemniscatic maps.  Both the analytic results 
from~\cite{SeteLiesen2015} and the numerical method from~\cite{NLS2015_numconf} 
will be used in Sections~\ref{sect:intervals} and~\ref{sect:examples} below.

In~\cite{Walsh1958} Walsh used Theorem~\ref{thm:conformal_map} for proving the 
existence of a direct generalization of the (classical) Faber polynomials
to sets $E$ with several components.
The second major ingredient is the following.  For the unit disk, the monomials 
$w^k$ are fundamental for Taylor and Laurent series of analytic functions, and 
the zeros of $w^k$ are at the center of the unit disk.  For a lemniscatic 
domain $\cL$, we need a generalization of $w^k$ to a polynomial with zeros at 
the foci $a_1, \ldots, a_N$ of $\cL$, and the multiplicity of each zero 
$a_j$ must correspond to its ``importance'' for $\cL$, given by the exponent 
$m_j$.

\begin{lemma}[{\cite[Lemma~2]{Walsh1958}}] \label{lem:alpha_n}
Let $\cL$ be a lemniscatic domain as in~\eqref{eqn:lemniscate}.
\begin{enumerate}
\item There exists a sequence $(\alpha_j)_{j=1}^\infty$, chosen from the 
foci $a_1, \ldots, a_N$, such that
\begin{equation}
\abs{N_{k,j} - k m_j} \leq A \quad \text{for } j = 1, 2, \ldots, N, \quad k = 
1, 2, \ldots,
\label{eqn:alpha_frequency}
\end{equation}
where $N_{k,j}$ denotes the number of times $a_j$ appears in the 
sequence $\alpha_1, \ldots$, $\alpha_k$, and where $A > 0$ is a constant.

\item Any such sequence has the following property:
For any closed set $S \subseteq \widehat{\C}$ not containing any of the points 
$a_1, \ldots, a_N$ there exist constants $A_1, A_2 > 0$, such that
\begin{equation}
A_1 < \frac{ \abs{u_k(w)} }{ \abs{U(w)}^k } < A_2 \quad \text{for } k = 
0, 1, 2, \ldots, \text{ and any } w \in S, \label{eqn:double_bound_un}
\end{equation}
where $u_k(w) \coloneq \prod_{j=1}^k (w - \alpha_j) = \prod_{j=1}^N 
(w-a_j)^{N_{k,j}}$.
\end{enumerate}
\end{lemma}

For $N = 1$ a lemniscatic domain is the exterior of a disk, and we have 
$\alpha_j = a_1$ for all $j \geq 1$ and $u_k(w) = (w-a_1)^k$.
For $N \geq 2$, the sequence $(\alpha_j)_{j=1}^\infty$ is \emph{not unique}, 
but it can be chosen constructively from $a_1, \ldots, a_N$; 
see~\cite{Walsh1958}.
Note that a smaller constant $A$ in~\eqref{eqn:alpha_frequency} implies better 
bounds in~\eqref{eqn:double_bound_un}.
For $N = 2$ one possible choice is $\alpha_j = a_1$ if $\lfloor j m_1 \rfloor > 
\lfloor (j-1) m_1 \rfloor$, and $\alpha_j = a_2$ otherwise, where $\lfloor 
\cdot \rfloor$ denotes the integer part.
We use this choice in our examples in Sections~\ref{sect:intervals} 
and~\ref{subsec:disks} below.

In the notation of Theorem~\ref{thm:conformal_map},  the Green's functions with 
pole at infinity for $\cL$ and $\cK = \widehat{\C} \backslash E$ are
\begin{equation}
g_\cL(w) = \log \abs{U(w)} - \log(\mu) \quad \text{and} \quad
g_\cK(z) = g_\cL(\Phi(z)),
\label{eqn:greens_fct}
\end{equation}
respectively; see~\cite{SeteLiesen2015,Walsh1958}.
For $\sigma > 1$ we denote their level curves by
\begin{equation*}
\begin{split}
\Gamma_\sigma &= \{ z \in \cK : g_\cK(z) = \log(\sigma) \}, \\
\Lambda_\sigma &= \{ w \in \cL : g_\cL(w) = \log(\sigma) \} = \{ w \in \cL : 
\abs{U(w)} = \sigma \mu \}.
\end{split}
\end{equation*}
Note that $\Phi(\Gamma_\sigma) = \Lambda_\sigma$.  Further, we denote by 
$\interior$ and $\exterior$ the interior and exterior of a closed curve (or 
union of closed curves), respectively.  In particular, we have
\begin{equation*}
\interior(\Lambda_\sigma) = \{ w \in \widehat{\C} : \abs{U(w)} < \sigma \mu \},
\quad
\exterior(\Lambda_\sigma) = \{ w \in \widehat{\C} : \abs{U(w)} > \sigma \mu \}.
\end{equation*}
We can now state Walsh's main result from~\cite{Walsh1958}.

\begin{theorem}[{\cite[Theorem~3]{Walsh1958}}] \label{thm:fw_poly}
Let $E$, $\cL$ and $\psi = \Phi^{-1}$ be as in Theorem~\ref{thm:conformal_map}.
Let $(\alpha_j)_{j=1}^\infty$ and the corresponding polynomials $u_k(w)$, $k = 
0, 1, \ldots$, be as in Lemma~\ref{lem:alpha_n}.
Then the following hold:
\begin{enumerate}
\item For $z \in \Gamma_\sigma$ and $w \in \exterior(\Lambda_\sigma)$ we have
\begin{equation}
\frac{ \psi'(w) }{\psi(w)-z} = \sum_{k=0}^\infty \frac{ b_k(z) }{ u_{k+1}(w) }, 
\label{eqn:gen_function}
\end{equation}
where
\begin{equation}
b_k(z)
= \frac{1}{2 \pi i} \int_{\Lambda_\lambda} u_k(\tau) \frac{ 
\psi'(\tau) }{ \psi(\tau) - z } \, d\tau
= \frac{1}{2 \pi i} \int_{\Gamma_\lambda} \frac{u_k(\Phi(\zeta))}{\zeta-z} 
\,d\zeta
\label{eqn:def_bk}
\end{equation}
for any $\lambda>\sigma$.  The function $b_k$ is a monic polynomial of
degree $k$, which is called the \emph{$k$th Faber--Walsh polynomial}
for $E$ and $(\alpha_j)_{j=1}^\infty$.

\item Let $f$ be analytic on $E$, and let $\rho>1$ be the largest number such
that $f$ is analytic and single-valued in $\interior(\Gamma_\rho)$.
Then $f$ has a unique representation as a \emph{Faber--Walsh series}
\begin{equation*}
f(z) = \sum_{k=0}^\infty a_k b_k(z), \quad a_k = \frac{1}{2 \pi i}
\int_{\Lambda_\lambda} \frac{ f(\psi(\tau)) }{ u_{k+1}(\tau) } \, d\tau,
\quad 1 < \lambda < \rho,
\end{equation*}
which converges absolutely in $\interior(\Gamma_\rho)$ and maximally on $E$, 
i.e.,
\begin{equation*}
\limsup_{n \to \infty} \norm{ f - \sum_{k=0}^n a_k b_k }_E^{\frac{1}{n}}
= \frac{1}{\rho},
\end{equation*}
where $\norm{\cdot}_E$ denotes the maximum norm on $E$.
\end{enumerate}
\end{theorem}

Note that the assertions about the Faber--Walsh polynomials hold for any 
admissible sequence $(\alpha_j)_{j=1}^\infty$ as in Lemma~\ref{lem:alpha_n}.  
In this article, if we do not explicitly mention the sequence, the 
corresponding results hold for any such sequence.

For $N = 1$ the Faber--Walsh polynomials reduce to the monic Faber 
polynomials for the (simply connected) set $E$ as considered 
in~\cite{Ellacott1983,Markushevich1967,SmirnovLebedev}.

For an entire function $f$ we have $\rho = \infty$ and hence 
$\interior(\Gamma_\rho) = \C$ in part \textit{2.}\@ of the theorem.

In our proof of Proposition~\ref{prop:FW_asympt_opt} below we will use
that for each given $\sigma > 1$ there exists positive constants $C_1, C_2$
independent of $k$ such that
\begin{equation}
0 < C_1 \abs{u_k(\Phi(z))} \leq \abs{b_k(z)} \leq C_2 \abs{u_k(\Phi(z))} \quad 
\text{for } z \in \Gamma_\sigma.
\label{eqn:double_bound_bn}
\end{equation}
Here the upper bound on $\abs{b_k(z)}$ holds for all $k$ and the lower bound 
holds only for sufficiently large $k$; see~\cite[p.~253]{Suetin1998}.

In~\eqref{eqn:gen_function}--\eqref{eqn:def_bk} the Faber--Walsh polynomials 
are defined as the (polynomial) coefficients in the expansion of the function
$\psi'(w) /(\psi(w)-z)$.  Similar to the (classical) Faber polynomials,
the Faber--Walsh polynomials can also be defined using the coefficients
of the Laurent series of the conformal map $\psi$ in a neighborhood of
infinity.  Using this approach one can derive a recursive formula for computing
the Faber--Walsh polynomials.  In the following result we state the recursion
that we have used in our numerical computations that are described in
Section~\ref{subsec:sym_ints}.  A variant of this recursion was first
published in the technical report~\cite{Sete2013}.

\begin{proposition} \label{prop:recursion_bn}
In the notation of Theorem~\ref{thm:fw_poly}, the Laurent series at 
infinity of the conformal map $\psi = \Phi^{-1}$ has the form
\begin{equation*}
\psi(w) = w + \sum_{k=1}^\infty \frac{c_k}{w^k}.
\end{equation*}
Then the Faber--Walsh polynomials are recursively given by
\begin{equation*}
\begin{split}
b_0(z) &= 1 \\
b_k(z) &= (z-\alpha_k) b_{k-1}(z) + \beta_{k-1,1}(z), \quad k \geq 1,
\end{split}
\end{equation*}
where the $\beta_{k,\ell}(z)$ are polynomials given by $\beta_{0,1}(z)=0$ and
\begin{equation*}
\begin{split}
\beta_{1,\ell}(z) &= \alpha_1 (\ell-1) c_{\ell-1} - (\ell+1) c_\ell, \quad
\ell \geq 1, \\
\beta_{k,\ell}(z) &= -c_\ell b_{k-1}(z) - \alpha_k \beta_{k-1,\ell}(z) +
\beta_{k-1,\ell+1}(z), \quad k \geq 2, \quad \ell \geq 1,
\end{split}
\end{equation*}
with $c_0 = 0$.
\end{proposition}

 
\section{On the theory of Faber--Walsh polynomials}
\label{sect:theory}

In this section we derive several new results about Faber--Walsh polynomials.
We begin with an alternative representation, and then relate Faber polynomials
for a simply connected set $\Omega$ to the Faber--Walsh polynomials for a
polynomial pre-image of $\Omega$.  Finally, we will show that the normalized
Faber--Walsh polynomials are asymptotically optimal.

Our first result is an easy consequence of Theorem~\ref{thm:fw_poly}.

\begin{corollary} \label{cor:bn_poly_part}
In the notation of Theorem~\ref{thm:fw_poly}, the $k$th Faber--Walsh
polynomial $b_k(z)$ is the polynomial part of the Laurent series at infinity
of $u_k(\Phi(z))$.
\end{corollary}

\begin{proof}
By Theorem~\ref{thm:conformal_map}, the Laurent series at infinity of the
lemniscatic map $\Phi$ has the form
$\Phi(\zeta) = \zeta + \sum_{j=1}^\infty \frac{d_j}{\zeta^j}$, and thus
\begin{equation*}
u_k(\Phi(\zeta)) = \prod_{j=1}^k (\Phi(\zeta)-\alpha_j) = p_k(\zeta) +
\sum_{j=1}^\infty \frac{\widetilde{d}_j}{\zeta^j},
\end{equation*}
where $p_k$ is a monic polynomial of degree $k$.  For $z \in \C$ and $\lambda >
1$ sufficiently large,~\eqref{eqn:def_bk} yields
\begin{equation*}
\begin{split}
b_k(z) &= \frac{1}{2 \pi i} \int_{\Gamma_\lambda}
\frac{u_k(\Phi(\zeta))}{\zeta-z} \, d\zeta
= \frac{1}{2 \pi i} \int_{\Gamma_\lambda} \frac{p_k(\zeta)}{\zeta-z} \, d\zeta
+ \frac{1}{2 \pi i} \int_{\Gamma_\lambda} \sum_{j=1}^\infty
\frac{\widetilde{d}_j}{\zeta^j} \frac{1}{\zeta-z} \, d\zeta \\
&= p_k(z).
\end{split}
\end{equation*}
The second integral vanishes by virtue of Cauchy's integral formula for domains
with infinity as interior point; see,
e.g.,~\cite[Problem~14.14]{Markushevich1965}.
\eop
\end{proof}

Note that in the case $N = 1$ this result reduces to the classical fact that
the $k$th Faber polynomial $F_k(z)$ is the polynomial part of the Laurent
series at infinity of $(\Phi(z))^k$.

We will now consider sets $E$ that are polynomial pre-images of simply
connected sets $\Omega$.  We first recall the following result
from~\cite{SeteLiesen2015} about the corresponding lemniscatic maps, where
$\Omega = \Omega^* = \{ \overline{z} : z \in \Omega \}$ means that $\Omega$ is
symmetric with respect to the real line.

\begin{theorem}[{\cite[Theorem~3.1]{SeteLiesen2015}}] \label{thm:preimage}
Let $\Omega = \Omega^* \subseteq \C$ be a simply connected compact set (not a
single point) with exterior Riemann mapping
\begin{equation*}
\widetilde{\Phi} : \widehat{\C} \backslash \Omega \to \{ w \in \widehat{\C} :
\abs{w} > 1 \}, \quad
\widetilde{\Phi}(\infty) = \infty, \quad \widetilde{\Phi}'(\infty) > 0.
\end{equation*}
Let $P(z) = \alpha z^n + \alpha_0$ with $\alpha > 0$, $n \geq 2$, and $\alpha_0 
< \min ( \Omega \cap \R )$.
Then $E \coloneq P^{-1}(\Omega)$ is the disjoint union of $n$ simply connected
compact sets, and
\begin{equation*}
\Phi : \widehat{\C} \backslash E \to \cL = \{ w \in \widehat{\C} : \abs{U(w)}
> \mu \}, \quad
\Phi(z) = z \Big( \frac{\mu^n}{z^n} [ \widetilde{\Phi}(P(z)) -
\widetilde{\Phi}(P(0)) ] \Big)^{\frac{1}{n}},
\end{equation*}
is the lemniscatic map of $E$, where we take the principal branch of the $n$th
root, and where
\begin{equation*}
\mu \coloneq \Big( \frac{1}{\alpha \widetilde{\Phi}'(\infty)}
\Big)^{\frac{1}{n}} > 0, \quad \text{and} \quad
U(w) \coloneq ( w^n + \mu^n \widetilde{\Phi}(P(0)) )^{\frac{1}{n}}.
\end{equation*}
\end{theorem}

Note that we consider the Riemann map onto the exterior of the unit disk, so
that $\widetilde{\Phi}'(\infty)$ is positive, but is not necessarily $1$.
Therefore, the Faber polynomials $F_k$ associated with this map have the
leading coefficients $(\widetilde{\Phi}'(\infty))^k$, and will in general
not be monic.

We then obtain the following ``transplantation result'' for Faber--Walsh
polynomials, which is an analogue of a similar result for Chebyshev polynomials
shown in~\cite{FischerPeherstorfer2001,KamoBorodin1994}.
For related results on polynomial pre-images see
also~\cite{Peherstorfer2003,PeherstorferSteinbauer2001}.

\begin{theorem} \label{thm:relation_fw_f}
In the notation of Theorem~\ref{thm:preimage}, denote the $n \geq 2$ distinct
roots of the polynomial $(U(w))^n$ by $a_1, \ldots, a_n$.
Then, the $(kn)$th Faber--Walsh polynomial for $E$ and $(\alpha_j)_{j=1}^\infty
= (a_1, a_2, \ldots, a_n, a_1, a_2, \ldots, a_n, \ldots)$ satisfies
\begin{equation}
b_{kn}(z) = \frac{1}{(\alpha \widetilde{\Phi}'(\infty))^k} F_k(P(z)),
\label{eqn:relation_fw_f}
\end{equation}
for all $k \geq 0$, where $F_k$ is the $k$th Faber polynomial for $\Omega$.
\end{theorem}

\begin{proof}
Theorem~\ref{thm:preimage} implies that
\begin{equation*}
\prod_{j=1}^n (\Phi(z)-a_j) = (U(\Phi(z)))^n = \mu^n \widetilde{\Phi}(P(z))
= \frac{1}{\alpha \widetilde{\Phi}'(\infty)} \widetilde{\Phi}(P(z)).
\end{equation*}
Then, for $k \geq 0$, we find
\begin{equation*}
u_{kn}(\Phi(z)) = \prod_{j=1}^n (\Phi(z)-a_j)^k
= \frac{1}{(\alpha \widetilde{\Phi}'(\infty))^k} ( \widetilde{\Phi}(P(z)) )^k.
\end{equation*}
Considering on both sides the polynomial part of the Laurent series at
infinity, we find~\eqref{eqn:relation_fw_f};
see Corollary~\ref{cor:bn_poly_part} for the Faber--Walsh polynomials,
and, for instance,~\cite[p.~33]{Suetin1998} for the Faber polynomials.
\eop
\end{proof}

In Theorem~\ref{thm:relation_fw_f} other choices of the sequence
$(\alpha_j)_{j=1}^\infty$ are possible:  If $(\alpha_j)_{j=1}^\infty$ satisfies
$u_{kn}(w) = \prod_{j=1}^n (w-a_j)^k$ for some $k$,
then~\eqref{eqn:relation_fw_f} holds for this $k$.

For a polynomial of degree $1$ the situation is different from the one in
Theorem~\ref{thm:relation_fw_f}, since the polynomial then is a linear
transformation and thus preserves the number of components of a set.
In this case we obtain the following stronger result.

\begin{proposition}
\label{prop:bn_under_moebius}
Let the notation be as in Theorem~\ref{thm:fw_poly}, and let $P(z) = \alpha z
+ \beta$ with $\alpha, \beta \in \C$ and $\alpha \neq 0$.  Then the 
Faber--Walsh polynomials $\widetilde{b}_k$ for $P(E)$ and 
$(P(\alpha_j))_{j=1}^\infty$ satisfy
$b_k(z) = \frac{1}{\alpha^k} \widetilde{b}_k(P(z))$ for all $k \geq 0$.
\end{proposition}

\begin{proof}
Let $\Phi : \widehat{\C} \backslash E \to \cL$ be the lemniscatic map of $E$.
Then $P(\cL)$ is a lemniscatic domain and $\widetilde{\Phi} \coloneq P \circ
\Phi \circ P^{-1}$ is the lemniscatic map of $P(E)$;
see~\cite[Lemma~2.3]{SeteLiesen2015}.  In particular, the polynomials $u_k(w) =
\prod_{j=1}^k (w-\alpha_j)$ for $\cL$ and $\widetilde{u}_k(\widetilde{w}) =
\prod_{j=1}^k ( \widetilde{w} - P(\alpha_j))$ for $P(\cL)$ satisfy
\begin{equation*}
\widetilde{u}_k( \widetilde{\Phi}(P(z)) )
= \widetilde{u}_k( P(\Phi(z)) )
= \prod_{j=1}^k ( P(\Phi(z)) - P(\alpha_j))
= \alpha^k u_k(\Phi(z)).
\end{equation*}
Corollary~\ref{cor:bn_poly_part} implies that $\widetilde{b}_k(P(z)) = \alpha^k
b_k(z)$.
\eop
\end{proof}

We now show that Faber--Walsh polynomials are asymptotically optimal in
the sense of the following definition introduced by Eiermann, Niethammer and
Varga~\cite{EiermannNiethammer1983,EiermannNiethammerVarga1985} in the context
of semi-iterative methods for solving linear algebraic systems.

\begin{definition}
For a compact set $E \subseteq \C$ and $z_0 \in \C$ the number
\begin{equation*}
R_{z_0}(E) \coloneq \limsup_{k \to \infty} \Big( \min_{ p \in \cP_k(z_0) }
\norm{p}_E \Big)^{1/k}
\end{equation*}
is called the \emph{asymptotic convergence factor} for polynomials from
$\cP_k(z_0) \coloneq \{ 1 + \sum_{j=1}^k a_j (z-z_0)^j : a_1, \ldots, a_k \in
\C \}$ on $E$.
A sequence of polynomials $p_k \in \cP_k(z_0)$, $k = 0, 1, \ldots$, is called
\emph{asymptotically optimal} on $E$ and with respect to $z_0$, if
\begin{equation*}
\lim_{k \to \infty} \norm{p_k}_E^{1/k} = R_{z_0}(E).
\end{equation*}
\end{definition}

For any compact set $E$ and $z_0 \in \C$ we have $R_{z_0}(E) \leq 1$, and
$R_{z_0}(E) = 1$ if $z_0 \in E$.  More precisely, one can show that
$R_{z_0}(E) < 1$ if and only if $z_0$ is in the unbounded component of
$\widehat{\C} \backslash E$.

Let $E$ be a compact set as in Theorem~\ref{thm:conformal_map} and let $g_\cK$ 
be the Green's function with pole at infinity for $\cK = \widehat{\C} 
\backslash E$, which is connected.
Then the Bernstein--Walsh Lemma (see, e.g.,~\cite[Theorem~5.5.7 (a)]{Ran95}
or~\cite[Section~4.6]{Walsh1969}) says that any polynomial $p$ of 
degree $k\geq 1$ satisfies
\begin{equation*}
\left(\frac{|p(z_0)|}{\|p\|_E}\right)^{1/k} \leq \exp(g_\cK(z_0))
\quad\mbox{for all $z_0\in\C\setminus E$.}
\end{equation*}
Together with~\cite[Theorem~5.5.7 (b)]{Ran95} this yields
\begin{equation}
R_{z_0}(E) = \exp(-g_\cK(z_0)) = \frac{\mu}{\abs{U(\Phi(z_0))}}
\label{eqn:ACF_green}
\end{equation}
(the second equality follows from~\eqref{eqn:greens_fct}) and the lower bound
\begin{equation}
\norm{p}_E \geq R_{z_0}(E)^k \quad\mbox{for any } p \in \cP_k(z_0).
\label{eqn:BernsteinWalsh}
\end{equation}
Both~\eqref{eqn:ACF_green} and~\eqref{eqn:BernsteinWalsh} have been shown for 
$z_0 = 1 \notin E$ in~\cite{EiermannLiVarga1989}, and the proof given there can 
be easily extended to all $z_0 \in \C \backslash E$.

Note that~\eqref{eqn:ACF_green} is a formula for the asymptotic convergence
factor for any complex set $E$ as in Theorem~\ref{thm:conformal_map} and any
$z_0 \in \C \backslash E$. (For all other $z_0$ we have $R_{z_0}(E)=1$ as
mentioned above.) Using~\eqref{eqn:ACF_green} together with the numerical
method from~\cite{NLS2015_numconf} for computing lemniscatic maps gives
a numerical method for computing the asymptotic convergence factor for a
large class of compact sets and arbitrary constraint points. This method is
used in our numerical examples in Sections~\ref{sect:intervals}
and~\ref{sect:examples} below.

The inequality~\eqref{eqn:BernsteinWalsh} belongs to the family of
Bernstein--Walsh type inequalities. They are widely used in the literature in particular
in the context of iterative methods for solving linear algebraic systems; see,
e.g.,~\cite{DriscollTohTrefethen1998,EiermannLiVarga1989,Schiefermayr2011_minres} 
and the references cited therein. For $E$ consisting of a finite number of real 
intervals an improved lower bound for $\norm{p}_E$, $p \in \cP_k(0)$, has 
been derived in~\cite{Schiefermayr2011_minres}.

\begin{proposition} \label{prop:FW_asympt_opt}
In the notation of Theorem~\ref{thm:fw_poly}, let
$z_0 \in \C \backslash E$ and let $\sigma_0 > 1$ be such that
$z_0 \in \Gamma_{\sigma_0}$.  Then
\begin{equation*}
R_{z_0}(E) = \frac{1}{\sigma_0} \quad \text{ and } \quad
R_{z_0}(\overline{\interior(\Gamma_\sigma)}) = \frac{\sigma}{\sigma_0} \quad
\text{for } 1 < \sigma \leq \sigma_0,
\end{equation*}
and the Faber--Walsh polynomials for $E$ satisfy
\begin{align}
\lim_{k \to \infty} \Big( \frac{\norm{b_k}_E}{\abs{b_k(z_0)}} \Big)^{1/k}
&= \frac{1}{\sigma_0} = R_{z_0}(E), \label{eqn:limit_nfw_on_E} \\
\lim_{k \to \infty} \Big( \frac{\norm{b_k}_{\Gamma_\sigma}}{\abs{b_k(z_0)}}
\Big)^{1/k}
&= \frac{\sigma}{\sigma_0} \quad \text{for any } \sigma > 1.
\label{eqn:limit_nfw_on_Gamma_sigma}
\end{align}
Hence the normalized Faber--Walsh polynomials $b_k(z)/b_k(z_0) \in \cP_k(z_0)$
are a\-symptotically optimal on $E$, and on 
$\overline{\interior(\Gamma_\sigma)}$ whenever $1 < \sigma \leq \sigma_0$.
\end{proposition}

\begin{proof}
The Green's function with pole at infinity for $\cK = \widehat{\C} \backslash
E$ is $g_\cK$ as in~\eqref{eqn:greens_fct}.  By the definition of $\sigma_0$,
we have $g_\cK(z_0) = \log(\sigma_0)$, so that $R_{z_0}(E) =
\frac{1}{\sigma_0}$ by~\eqref{eqn:ACF_green}.
The Green's function with pole at infinity for $\exterior(\Gamma_\sigma)$ is
$g_\cK(z) - \log(\sigma)$.
Hence, for $1 < \sigma \leq \sigma_0$,
$R_{z_0}(\overline{\interior(\Gamma_\sigma)}) = \tfrac{\sigma}{\sigma_0}$
by~\eqref{eqn:ACF_green}.

Let $\sigma > 1$.  By~\eqref{eqn:double_bound_bn} there exists constants $C_1,
C_2 > 0$ such that for sufficiently large $k$ we have
\begin{equation*}
C_1 \abs{ u_k(\Phi(z)) } < \abs{ b_k(z) } < C_2 \abs{u_k(\Phi(z))}, \quad z \in
\Gamma_\sigma.
\end{equation*}
Apply Lemma~\ref{lem:alpha_n} to bound $\abs{u_k(\Phi(z))}$: There exist
$A_1, A_2 > 0$ such that~\eqref{eqn:double_bound_un} holds for $w \in
\Phi(\Gamma_\sigma) = \Lambda_\sigma = \{ w : \abs{U(w)} = \sigma \mu \}$.  We
thus have
\begin{equation}
C_1 A_1 (\sigma \mu)^k < \abs{ b_k(z) } < C_2 A_2 (\sigma \mu)^k, \quad z \in
\Gamma_\sigma. \label{eqn:estimate_of_b_kz}
\end{equation}
We then have $b_k(z) \neq 0$ for $z \in \Gamma_\sigma$ and, in particular,
$b_k(z_0) \neq 0$ for sufficiently large $k$.
Now~\eqref{eqn:estimate_of_b_kz} implies $\lim_{k \to \infty} \abs{ b_k(z)
}^{1/k} = \sigma \mu$ for any $z \in \Gamma_\sigma$, and, in
particular,
\begin{equation}
\lim_{k \to \infty} \abs{ b_k(z_0) }^{1/k} = \sigma_0 \mu.
\label{eqn:limit_b_kz0^1/n}
\end{equation}
Moreover $\lim_{k \to \infty} \norm{b_k}_{ \Gamma_\sigma }^{1/k} =
\sigma \mu$, since~\eqref{eqn:estimate_of_b_kz} holds uniformly for $z \in
\Gamma_\sigma$.  This
establishes~\eqref{eqn:limit_nfw_on_Gamma_sigma}.

It remains to show~\eqref{eqn:limit_nfw_on_E}.  We have $\norm{b_k}_E \geq
\mu^k$, since the Faber--Walsh polynomials $b_k$ are monic of degree $k$ and 
$\mu$ is the capacity of $E$; see~\cite[Theorem~5.5.4]{Ran95}.
Hence
\begin{equation*}
\mu^k \leq \norm{b_k}_E \leq
\norm{b_k}_{\Gamma_\sigma} \leq C_2 A_2 (\sigma \mu)^k,
\end{equation*}
where $\sigma > 1$ is arbitrary; see~\eqref{eqn:estimate_of_b_kz}.  This shows
that
\begin{equation*}
\mu \leq \liminf_{k \to \infty} \norm{b_k}_E^{1/k}
\leq \limsup_{k \to \infty} \norm{b_k}_E^{1/k}
\leq \sigma \mu,
\end{equation*}
and $\lim_{k \to \infty} \norm{b_k}_E^{1/k} = \mu$, since
$\sigma > 1$ was arbitrary.  Together with~\eqref{eqn:limit_b_kz0^1/n} we
obtain~\eqref{eqn:limit_nfw_on_E}.
\eop
\end{proof}

 
\section{Faber--Walsh polynomials on two real intervals}
\label{sect:intervals}

In this section we consider Faber--Walsh polynomials on sets consisting of two 
real intervals.

Polynomial approximation problems on such sets have been studied in numerous 
publications, dating back (at least) to the classical works of 
Achieser~\cite{Akhiezer1932_I,Akhiezer1933_II},
who derived analytic formulae for the Chebyshev polynomials and the 
Green's function in terms of Jacobi's elliptic and theta functions.
For a modern treatment of this area with many references up to 1996 we refer to 
Fischer's book~\cite{Fischer1996}.  It also contains an analytic formula for 
the asymptotic convergence factor for two intervals and real $z_0$ in terms of 
Jacobi's elliptic and theta functions (through its characterization with the 
Green's function), as well as a MATLAB code for its numerical 
computation~\cite[p.~130]{Fischer1996}.
More recently, Peherstorfer and Schiefermayr studied Chebyshev polynomials on 
several real intervals in~\cite{PeherstorferSchiefermayr1999}, and Schiefermayr 
derived bounds for the asymptotic convergence factor for two intervals in terms 
of elementary functions~\cite{Schiefermayr2011}.
Related approximation problems have been studied 
in~\cite{Hasson2007,Schiefermayr2008_cheb,Schiefermayr2011_minres}.

In this section we show how the Faber--Walsh polynomials fit into this widely 
studied area.  We first consider the case of two intervals of the same length.  
Here the lemniscatic map is known explicitly, so that the Faber--Walsh 
polynomials can be explicitly computed and related to the classical Faber and 
Chebyshev polynomials.
We then consider the general case of two arbitrary intervals, where we compute 
the lemniscatic map and the Faber--Walsh polynomials numerically.

\subsection{Two intervals of the same length}
\label{subsec:sym_ints}

We consider the sets consisting of two real intervals of the same length which 
are symmetric with respect to the origin, i.e.,
\begin{equation} \label{eqn:sym_ints}
E = [-D,-C] \cup [C,D] \quad \text{with } 0 < C < D.
\end{equation}
The lemniscatic map of such a set is known analytically 
from~\cite{SeteLiesen2015}.

\begin{proposition}\label{prop:Phi_sym_ints}
Let $E$ be as in~\eqref{eqn:sym_ints}.  Then
\begin{equation*}
w = \Phi(z) = z \left( \frac{1}{2} + \frac{DC}{2} \frac{1}{z^2}
\pm \frac{1}{2 z^2} \sqrt{ (z^2-C^2)(z^2-D^2) } \right)^{1/2}
\end{equation*}
is the lemniscatic map of $E$, and the corresponding lemniscatic domain is
\begin{equation}
\cL = \left\{ w \in \widehat{\C} : \abslr{w - \tfrac{D+C}{2}}^{1/2} 
\abslr{w + \tfrac{D+C}{2}}^{1/2} > \tfrac{ \sqrt{D^2 - C^2} }{2} \right\}.
\label{eqn:lem_sym_ints}
\end{equation}
Moreover, the inverse of $\Phi$ is given by
\begin{equation*}
z = \psi(w) = w \sqrt{ 1 + \left( \frac{D-C}{2} \right)^2 \frac{1}{w^2 - 
(\frac{D+C}{2})^2} },
\end{equation*}
where we take the principal branch of the square root.
Its Laurent series at infinity is
\begin{equation}
\psi(w) = w + \sum_{k=0}^\infty \frac{ c_{2k+1} }{ w^{2k+1} },
\label{eqn:psi_Lseries}
\end{equation}
where the coefficients are given by $c_{2k} = 0$, $k \geq 0$, and
\begin{equation}
c_{2k+1} = \frac{1}{2} \left( \frac{D-C}{2} \right)^2 \left( \frac{D+C}{2} 
\right)^{2k} - \frac{1}{2} \sum_{j=1}^{k}
c_{2j-1} c_{2(k-j)+1}, \quad k \geq 0. \label{eqn:psi_Lseries_coeff}
\end{equation}
In particular, $c_1 = \frac{1}{2} \left( \frac{D-C}{2} \right)^2$.
\end{proposition}

\begin{proof}
The construction of $\Phi$ and $\psi$ is given 
in~\cite[Corollary~3.3]{SeteLiesen2015}. It thus remains to show the series 
expansion~\eqref{eqn:psi_Lseries}--\eqref{eqn:psi_Lseries_coeff}.

The function $w \mapsto \sqrt{1 + \left( \frac{D-C}{2} \right)^2 \tfrac{1}{w^2 
- (\frac{D+C}{2})^2} }$ is analytic and even in $\cL$, and thus has a uniformly 
convergent Laurent series at infinity of the form
\begin{equation}
\sqrt{1 + \left( \frac{D-C}{2} \right)^2 \frac{1}{w^2 - (\frac{D+C}{2})^2} } = 
\sum_{k=0}^\infty \frac{d_k}{w^{2k}}. \label{eqn:Ansatz_L_series_psi}
\end{equation}
Setting $w = \infty$ shows $d_0  = 1$.  
Squaring~\eqref{eqn:Ansatz_L_series_psi} and expanding the left-hand side into 
a Laurent series yields
\begin{equation*}
1 + \sum_{k=1}^\infty \left( \frac{D-C}{2} \right)^2 \left( \frac{D+C}{2} 
\right)^{2(k-1)} \frac{1}{w^{2k}}
= \left( \sum_{k=0}^\infty \frac{d_k}{w^{2k}} \right)^2
= \sum_{k=0}^\infty \sum_{j=0}^k d_j d_{k-j} \frac{1}{w^{2k}}.
\end{equation*}
For $k \geq 1$ we see that
\begin{equation*}
\left( \frac{D-C}{2} \right)^2 \left( \frac{D+C}{2} \right)^{2(k-1)}
= \sum_{j=0}^k d_j d_{k-j} = 2 d_0 d_k + \sum_{j=1}^{k-1} d_j d_{k-j},
\end{equation*}
and thus
\begin{equation*}
d_k = \frac{1}{2} \left( \frac{D-C}{2} \right)^2 \left( \frac{D+C}{2} 
\right)^{2(k-1)} - \frac{1}{2} \sum_{j=1}^{k-1} d_j d_{k-j}, \quad k \geq 1.
\end{equation*}
Since $\psi(w) = w + \sum_{k=0}^\infty \frac{d_{k+1}}{w^{2k+1}}$, we find that 
$c_{2k} = 0$ and
\begin{equation*}
\begin{split}
c_{2k+1} &= d_{k+1}
= \frac{1}{2} \left( \frac{D-C}{2} \right)^2 \left( \frac{D+C}{2} \right)^{2k} 
- \frac{1}{2} \sum_{j=1}^k d_j d_{k+1-j} \\
&= \frac{1}{2} \left( \frac{D-C}{2} \right)^2 \left( \frac{D+C}{2} \right)^{2k} 
- \frac{1}{2} \sum_{j=1}^k c_{2j-1} c_{2(k-j)+1}
\end{split}
\end{equation*}
for $k \geq 0$.
\eop
\end{proof}

The definition of the lemniscatic domain in~\eqref{eqn:lem_sym_ints} shows the 
well-known fact that the logarithmic capacity of $E$ is given by $\frac{1}{2} 
\sqrt{D^2-C^2}$; cf. e.g.~\cite[p.~288]{Achieser1956} or \cite{Hasson2003}.

Using the series 
expansion~\eqref{eqn:psi_Lseries}--\eqref{eqn:psi_Lseries_coeff} and the 
recurrence stated in Proposition~\ref{prop:recursion_bn}, we can compute the 
Faber--Walsh polynomials for $E$ and the sequence $(\frac{D+C}{2}, 
-\frac{D+C}{2}, \frac{D+C}{2}$, $-\frac{D+C}{2}, \ldots)$, were
$\frac{D+C}{2}$ and $-\frac{D+C}{2}$ are the two foci of the lemniscatic 
domain $\cL$ in~\eqref{eqn:lem_sym_ints}.
In Figure~\ref{fig:bn_sym_ints} we plot the polynomials $b_k$ for the set
$E = [-1, -\frac{1}{4}] \cup [\frac{1}{4}, 1]$ and $k = 4, 6$ (left), $k = 5,
7$ (right).  We observe that the polynomials of even degrees $k = 4, 6$ 
have $k+2$ extremal points on $E$.  This suggests that they are the Chebyshev 
polynomials for $E$, i.e., that $b_{2k}(z) = T_{2k}(z; E)$, where, for all $j 
\geq 0$,
\begin{equation*}
T_j(z; E) \coloneq \argmin \{ \norm{p}_E : p \text{ monic and } \deg(p) = j \}
\end{equation*}
is the (uniquely determined) $j$th Chebyshev polynomial for the compact set $E$.
We prove the following more general result, which for $n = 2$ gives the result 
for two intervals.

\begin{figure}
\centerline{
\includegraphics[width=0.5\textwidth]{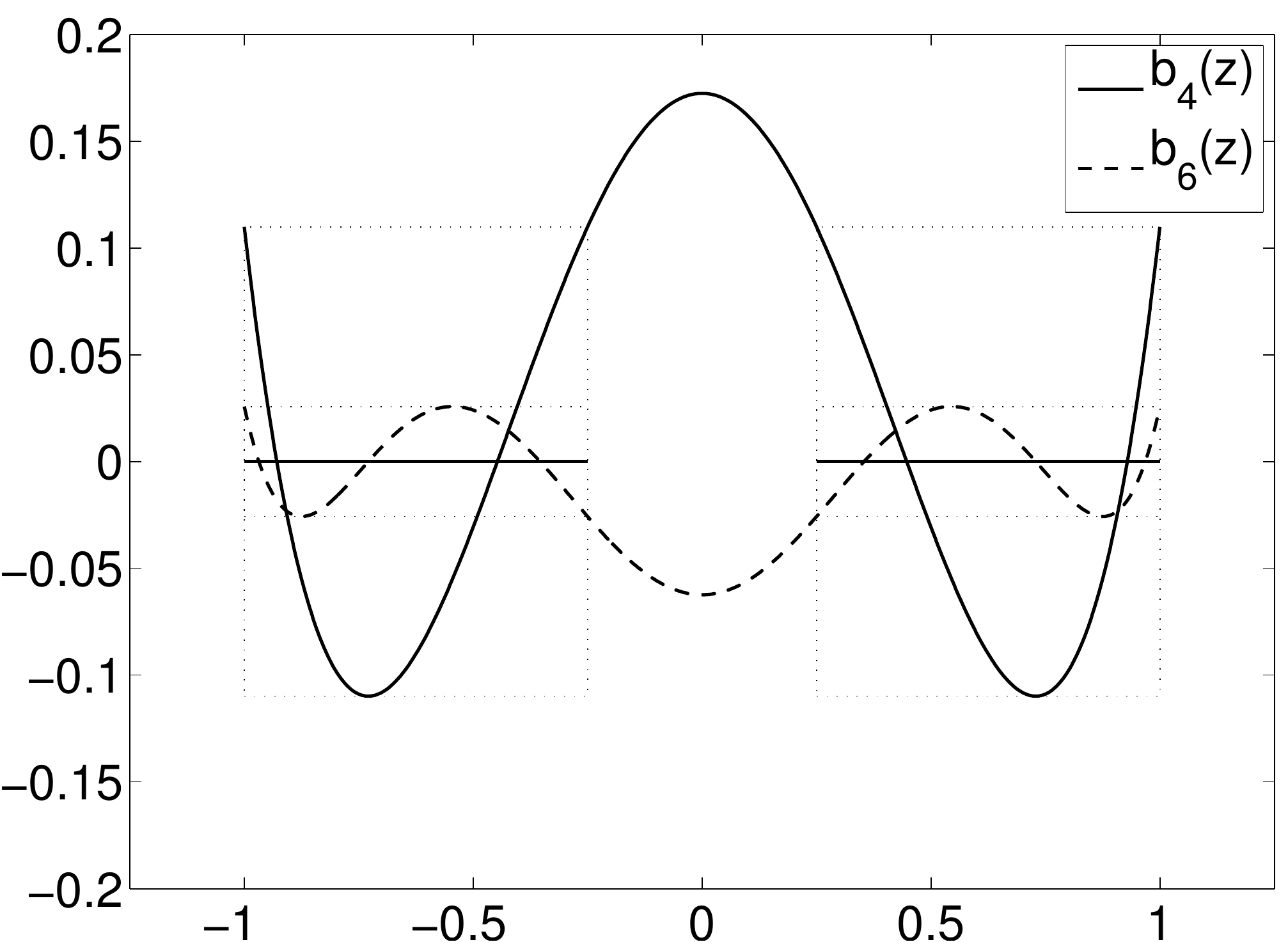}
\includegraphics[width=0.5\textwidth]{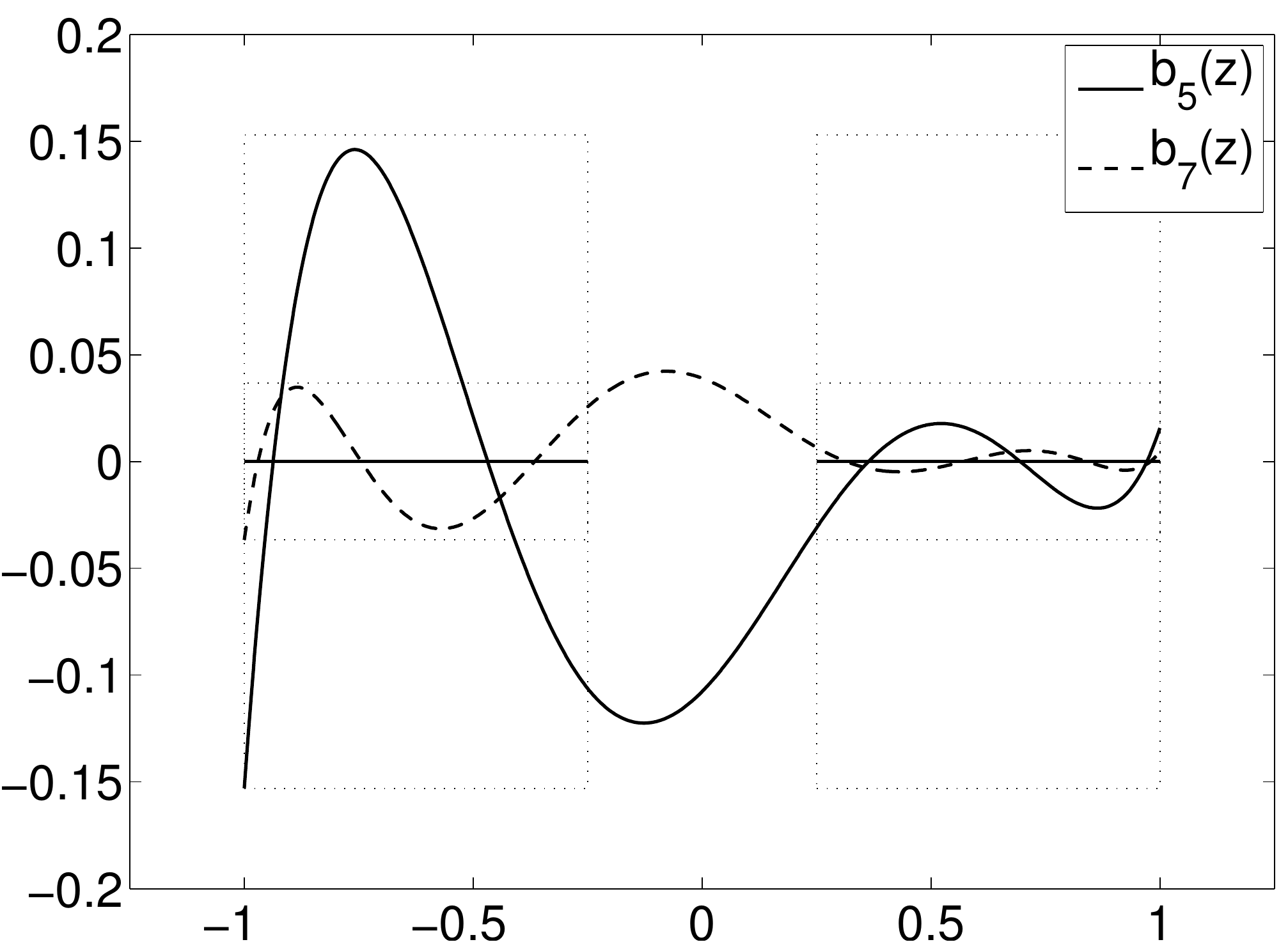}}
\caption{Faber--Walsh polynomials $b_k$ for $E = [-1, -\frac{1}{4}] \cup 
[\frac{1}{4}, 1]$.}
\label{fig:bn_sym_ints}
\end{figure}

\begin{theorem} \label{thm:bn_chebyshev}
Let $E = \cup_{j=1}^n e^{i 2 \pi j/n} [C, D]$ with $0 < C < D$.  Then the 
Faber--Walsh polynomials $b_{nk}$ for $E$ and $(\alpha_j)_{j=1}^\infty = \left( 
e^{i 2 \pi (j-1)/n } \left( \frac{D^{n/2}+C^{n/2}}{2} \right)^{n/2} 
\right)_{j=1}^\infty$ are the Chebyshev polynomials for $E$, i.e.,
\begin{equation*}
b_{nk}(z) = T_{nk}(z; E), \quad k \geq 0.
\end{equation*}
\end{theorem}

\begin{proof}
The result is trivial for $k = 0$, so we may consider $k \geq 1$.
The idea of the proof is to consider $E$ as a polynomial pre-image of $[-1,1]$ 
and to relate both the Faber--Walsh polynomial and the Chebyshev polynomial for 
$E$ to the Chebyshev polynomials of the first kind.

The polynomial
\begin{equation*}
P(z) = \frac{2}{D^n-C^n} z^n - \frac{D^n+C^n}{D^n-C^n}
= \frac{2 z^n - C^n-D^n}{D^n-C^n},
\end{equation*}
satisfies $E = P^{-1}([-1,1])$, and the exterior Riemann map $\widetilde{\Phi}$ 
for $[-1, 1]$ satisfies $\widetilde{\Phi}'(\infty) = 2$.  Therefore, 
Theorem~\ref{thm:relation_fw_f} shows that
\begin{equation*}
b_{nk}(z) = \left( \frac{D^n-C^n}{4} \right)^k F_k(P(z)), \quad k \geq 0,
\end{equation*}
where $F_k$ is the $k$th Faber polynomial for $[-1,1]$.
On the other hand, we have
\begin{equation*}
T_{nk}(z; E) = \left( \frac{D^n-C^n}{2} \right)^k T_k( P(z); [-1,1])
\end{equation*}
from~\cite[Corollary~2.2]{FischerPeherstorfer2001}.
For $k \geq 1$, the $k$th Chebyshev polynomial for $[-1,1]$ is
\begin{equation*}
T_k(z; [-1,1]) = \frac{1}{2^{k-1}} T_k(z),
\end{equation*}
where $T_k(z)$ is the $k$th Chebyshev polynomial of the first kind;
see e.g.~\cite[Theorem~3.2.2]{Fischer1996}.
Moreover, the $k$th Faber polynomial for $[-1,1]$ is given by $F_k(z) = 2 
T_k(z)$ for $k \geq 1$; see~\cite[p.~37]{Suetin1998}.  Thus
\begin{equation*}
T_{nk}(z; E) = \left( \frac{D^2-C^2}{4} \right)^k 2 T_k(P(z))
= \left( \frac{D^2-C^2}{4} \right)^k F_k(P(z)) = b_{nk}(z),
\end{equation*}
which completes the proof.
\eop
\end{proof}

This theorem generalizes the classical relation of Faber and Chebyshev 
polynomials on the interval $[-1, 1]$; see~\cite[p.~37]{Suetin1998}.
For $n = 2$ the statement of the theorem also holds for any two real intervals 
of equal length, which can be seen as follows.  Let $E = [A, B] \cup [C, D]$ 
with $B-A = D-C > 0$, and let $P(z) = z - \frac{B+C}{2}$.  Then $\widetilde{E} 
= P(E)$ consists of two intervals of equal length which are symmetric with 
respect to the origin.  If we denote the Faber--Walsh polynomials for 
$\widetilde{E}$ by $\widetilde{b}_n$, we find
\begin{equation*}
b_{2k}(z) = \widetilde{b}_{2k}(P(z)) = T_{2k}(P(z); \widetilde{E}) = T_{2k}(z; 
E),
\end{equation*}
where we used Proposition~\ref{prop:bn_under_moebius}, 
Theorem~\ref{thm:bn_chebyshev} 
and~\cite[Corollary~2.2]{FischerPeherstorfer2001}.

In Proposition~\ref{prop:FW_asympt_opt} we have shown that the normalized 
Faber--Walsh polynomials are asymptotically optimal.  For sets $E$ of the 
form~\eqref{eqn:sym_ints} this result can be strengthened as follows.

\begin{corollary} \label{cor:bn_optimal}
Let $E = [-D, -C] \cup [C, D]$ with $0 < C < D$ and $z_0 \in \R \backslash E$.  
Then the normalized Faber--Walsh polynomials for $E$ and 
$(\alpha_j)_{j=1}^\infty = (\frac{D+C}{2}, -\frac{D+C}{2}, \frac{D+C}{2}, 
-\frac{D+C}{2}, \ldots)$ of even degree are optimal in the sense that
\begin{equation*}
\frac{\norm{b_{2k}}_E}{\abs{b_{2k}(z_0)}} = \min_{ p \in \cP_{2k}(z_0) } 
\norm{p}_E, \quad k \geq 0.
\end{equation*}
\end{corollary}

\begin{proof}
By Theorem~\ref{thm:bn_chebyshev}, the Faber--Walsh polynomials of even degree 
are the Chebyshev polynomials for $E$.
For $z_0 \in \R \backslash [-D, D]$, Corollary~3.3.8 in~\cite{Fischer1996}
shows that the optimal polynomial is the normalized Chebyshev polynomial.
For $z_0 \in ]-C, C[$, a little more work is required.  First, it is not 
difficult to show that $T_2(z;E) = z^2 - \frac{D^2+C^2}{2}$ with 
$\norm{T_2(z;E)}_E = \frac{D^2-C^2}{2}$, and that $T_2(z;E)$ has the four 
extremal points $\pm C, \pm D$.  Therefore, the optimal polynomial is the 
normalized Chebyshev polynomial; see~\cite[Corollary~3.3.6]{Fischer1996}.
\eop
\end{proof}

We point out that the argument in the previous proof is restricted to 
real $z_0$, since the proofs in~\cite{Fischer1996} are based on the 
alternation property of the Chebyshev polynomials for subsets of the real line.

Note that the normalized Faber--Walsh polynomials $b_k(z) / b_k(z_0)$ of odd 
degrees are not optimal: 
If $z_0 \in \: ]-C, C[$, it is known that the optimal polynomial is 
``defective'', 
i.e., the optimal polynomial for degree $2k+1$ is the same as for degree 
$2k$; see~\cite[Corollary~3.3.6]{Fischer1996}.
If $z_0 \in \R \backslash [-D, D]$, the optimal polynomial is the normalized 
Chebyshev polynomial (\cite[Corollary~3.3.8]{Fischer1996}), while
in general the Faber--Walsh polynomials of odd degree are not the Chebyshev 
polynomials, since they do not have $k+1$ extremal points on 
$E$~\cite[Corollary~3.1.4]{Fischer1996}; see the example in 
Figure~\ref{fig:bn_sym_ints}.

Let us continue with a numerical study of the maximum norm of the 
normalized Faber--Walsh polynomials, where we focus on the constraint point 
$z_0 = 0$.  We compute the Faber--Walsh polynomials $b_{k,j}$, $k = 1, 2, 
\ldots$, for the sets
\begin{equation}\label{eqn:Ej_sym}
E_j\coloneq [-1,-2^{-j}] \cup [2^{-j},1],\quad j=1,2,3,4,
\end{equation}
and the sequence 
$(\frac{D+C}{2},-\frac{D+C}{2},\frac{D+C}{2},-\frac{D+C}{2},\dots)$ using the 
coefficients of the Laurent 
series~\eqref{eqn:psi_Lseries}--\eqref{eqn:psi_Lseries_coeff} and 
Proposition~\ref{prop:recursion_bn}.
Figure~\ref{fig:nfw_sym_ints_norm} (left) shows the values
$\frac{\norm{b_{k,j}}_{E_j}}{\abs{b_{k,j}(0)}}$ of the normalized Faber--Walsh 
polynomials for the sets $E_j$.
A comparison with the values $R_0(E_j)^k$ shows that the actual convergence 
speed of $\frac{\norm{b_{k,j}}_{E_j}}{\abs{b_{k,j}(0)}}$ to zero almost exactly 
matches the rate predicted by the asymptotic analysis even for small values of 
$k$.
Recall from~\eqref{eqn:BernsteinWalsh} that $R_{z_0}(E)^k$ is a lower 
bound on $\norm{p}_E$ for any $p \in \cP_k(z_0)$.
The ``zigzags'' in the curves are due to the fact that for even degrees 
$b_{k,j}(z) / b_{k,j}(0)$ is the optimal polynomial (as shown in 
Corollary~\ref{cor:bn_optimal}), while for odd degrees it is not.

\begin{figure}
\centerline{
\includegraphics[width=0.5\textwidth]{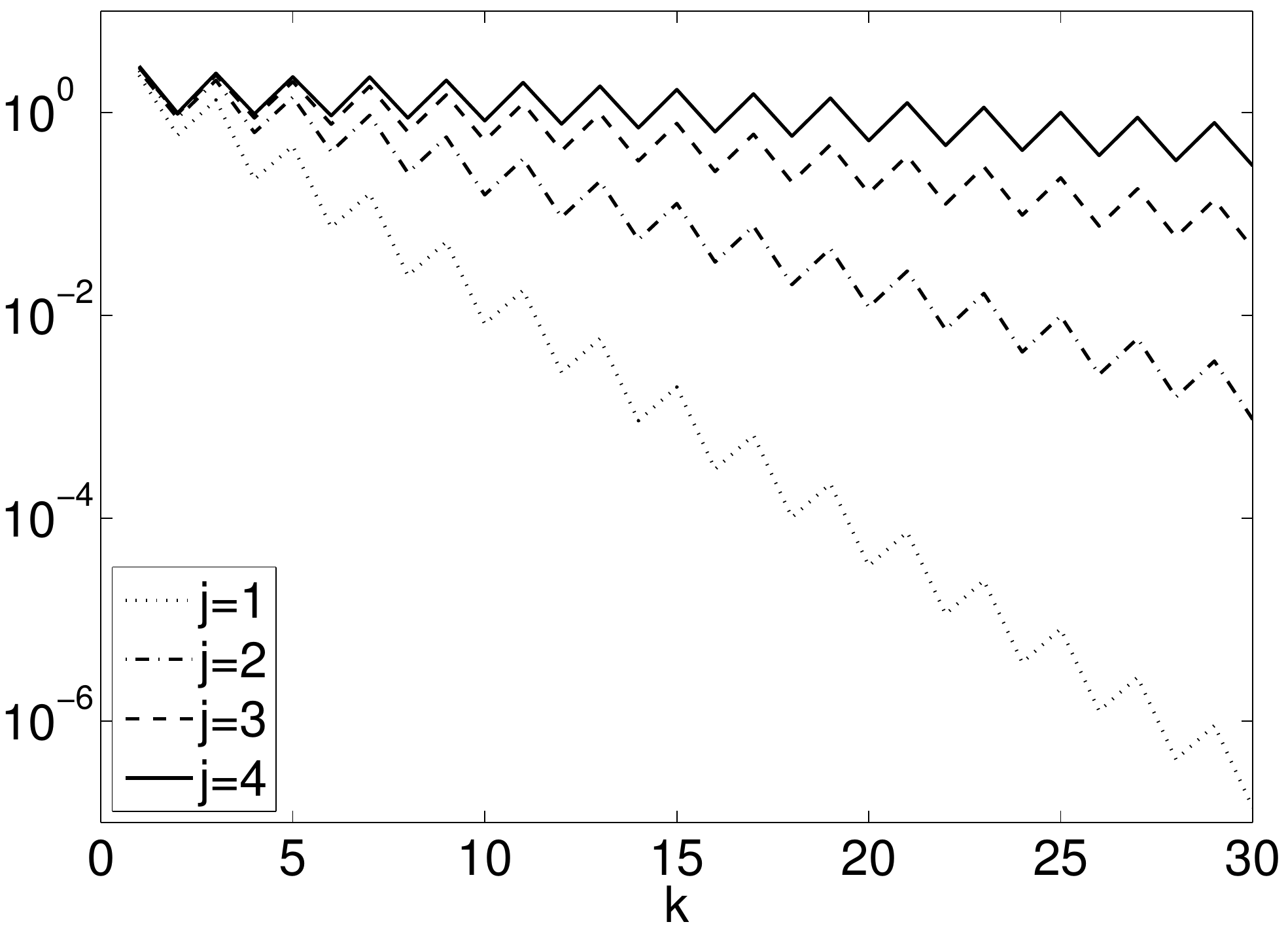}
\includegraphics[width=0.5\textwidth]{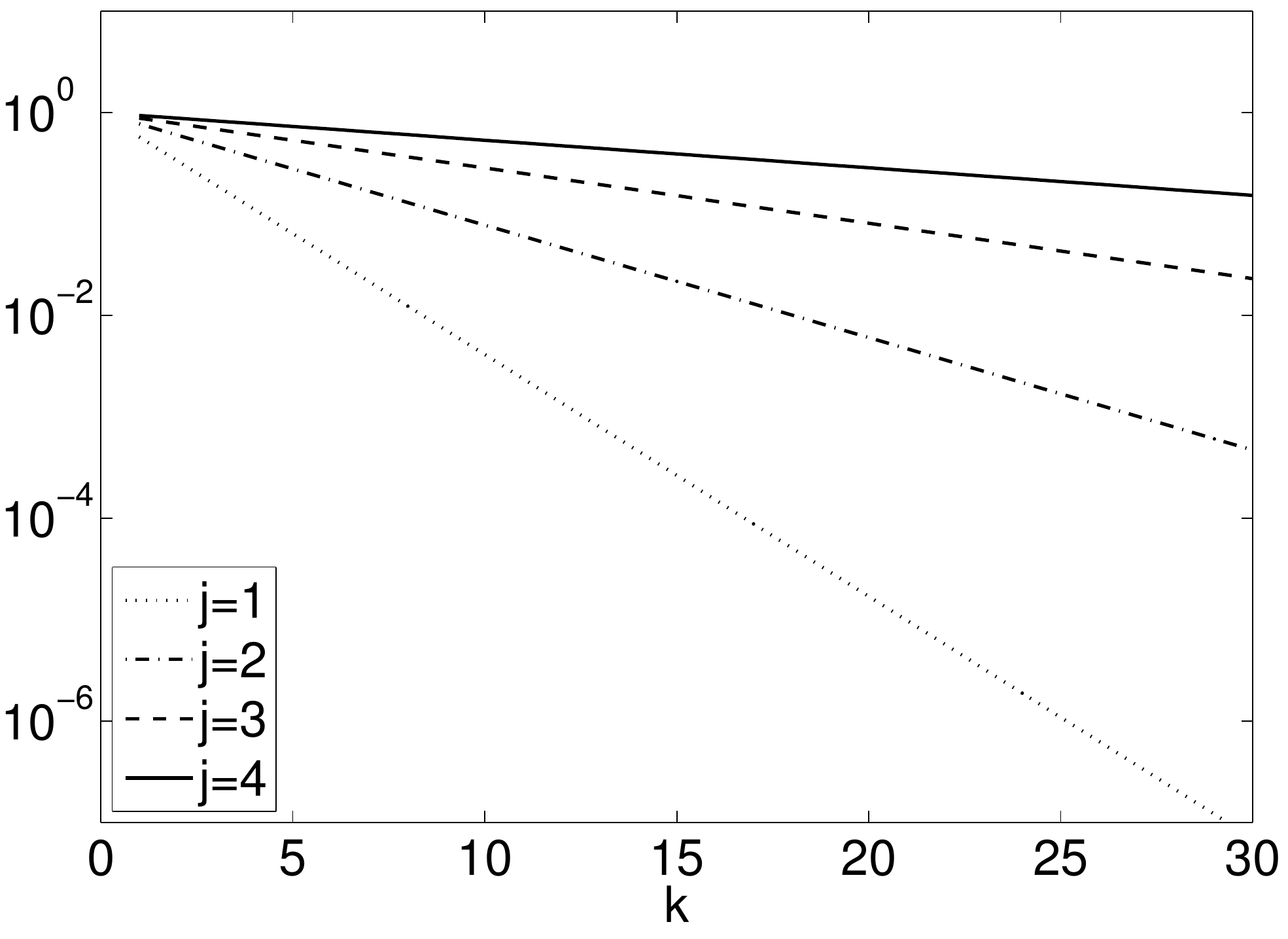}
}
\caption{The values $\frac{\|b_{k,j}\|_{E_j}}{|b_{k,j}(0)|}$ (left) and 
$R_0(E_j)^k$ (right) for the sets $E_j$ from~\eqref{eqn:Ej_sym}.}
\label{fig:nfw_sym_ints_norm}
\end{figure}

Let us discuss the asymptotic convergence factor of the set $E$ 
from~\eqref{eqn:sym_ints}.  For two arbitrary real intervals and a real 
constraint point $z_0$, the asymptotic convergence factor can be expressed in 
terms of Jacobi's elliptic and theta functions~\cite{Akhiezer1932_I}; see 
also~\cite{Fischer1996}.
Estimates of the asymptotic convergence factor in terms of elementary functions 
have been derived in~\cite{Schiefermayr2011}.
For the case of two intervals as in~\eqref{eqn:sym_ints} and for an 
arbitrary \emph{real or complex} constraint point $z_0 \in \C \backslash E$, we 
have from~\eqref{eqn:ACF_green} and Proposition~\ref{prop:Phi_sym_ints}
\begin{equation*}
R_{z_0}(E) = \frac{\mu}{\abs{U(\Phi(z_0))}}
= \frac{1}{\sqrt{ \frac{2}{D^2-C^2} \abslr{ z_0^2 - \frac{D^2+C^2}{2} \pm 
\sqrt{(z_0^2-C^2)(z_0^2-D^2)} } }},
\end{equation*}
where the sign of the square root is chosen to maximize the absolute value of 
the denominator.  We thus obtain $R_{z_0}(E)$ in terms of elementary functions 
and for all complex $z_0 \notin E$.
For the special case $z_0 = 0$ we have $\Phi(0) = 0$ and hence
\begin{equation*}
R_0(E) = \sqrt{ \frac{ D-C }{ D+C } }.
\end{equation*}
In Figure~\ref{fig:ACF_sym_ints_3d} we plot the asymptotic convergence factor 
$R_{z_0}( [-1, -\frac{1}{2}] \cup [\frac{1}{2}, 1] )$ as a function of $z_0 
\in \C$.

\begin{figure}
\centerline{
\includegraphics[width=0.5\textwidth]{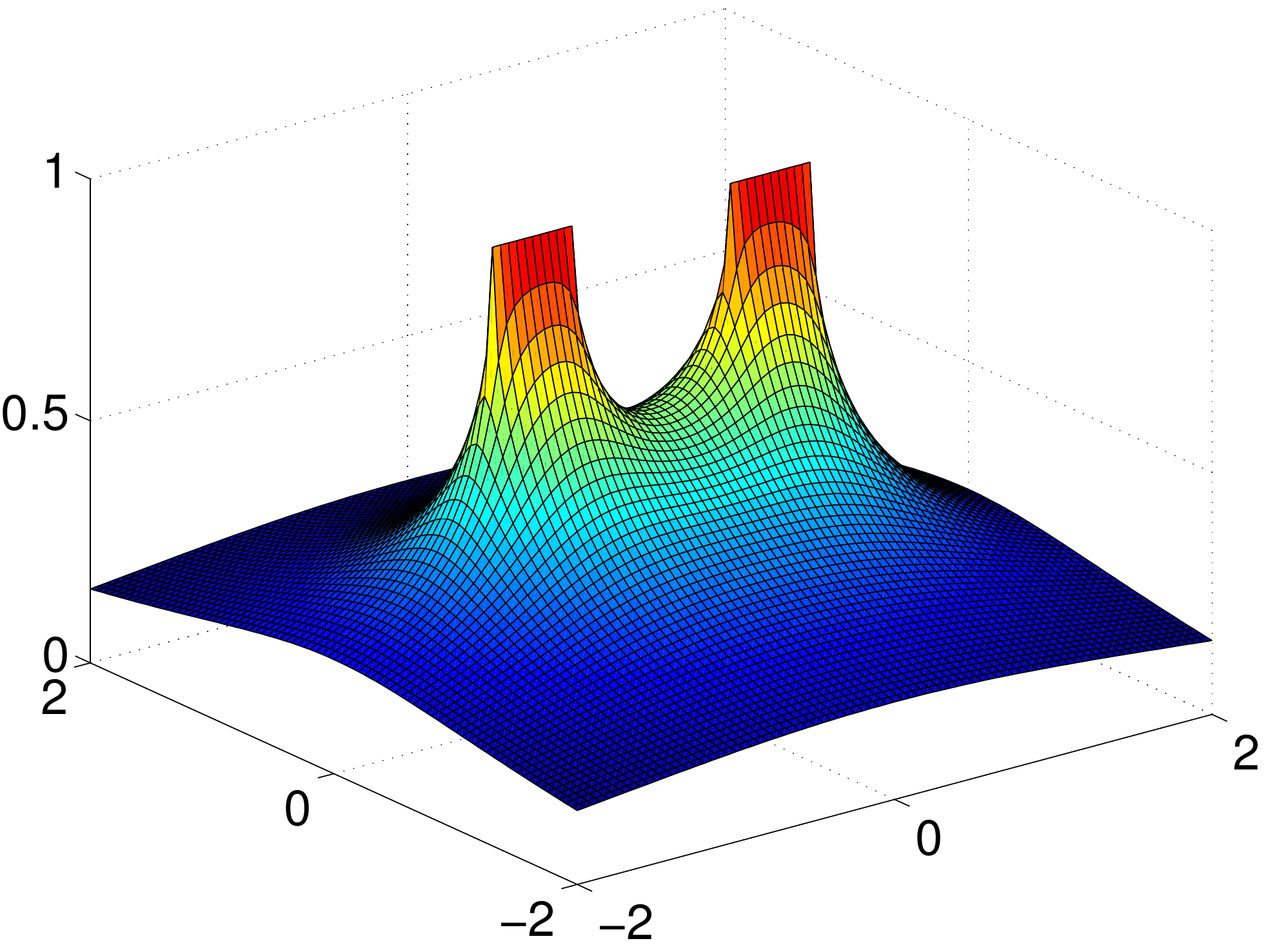}
\includegraphics[width=0.5\textwidth]{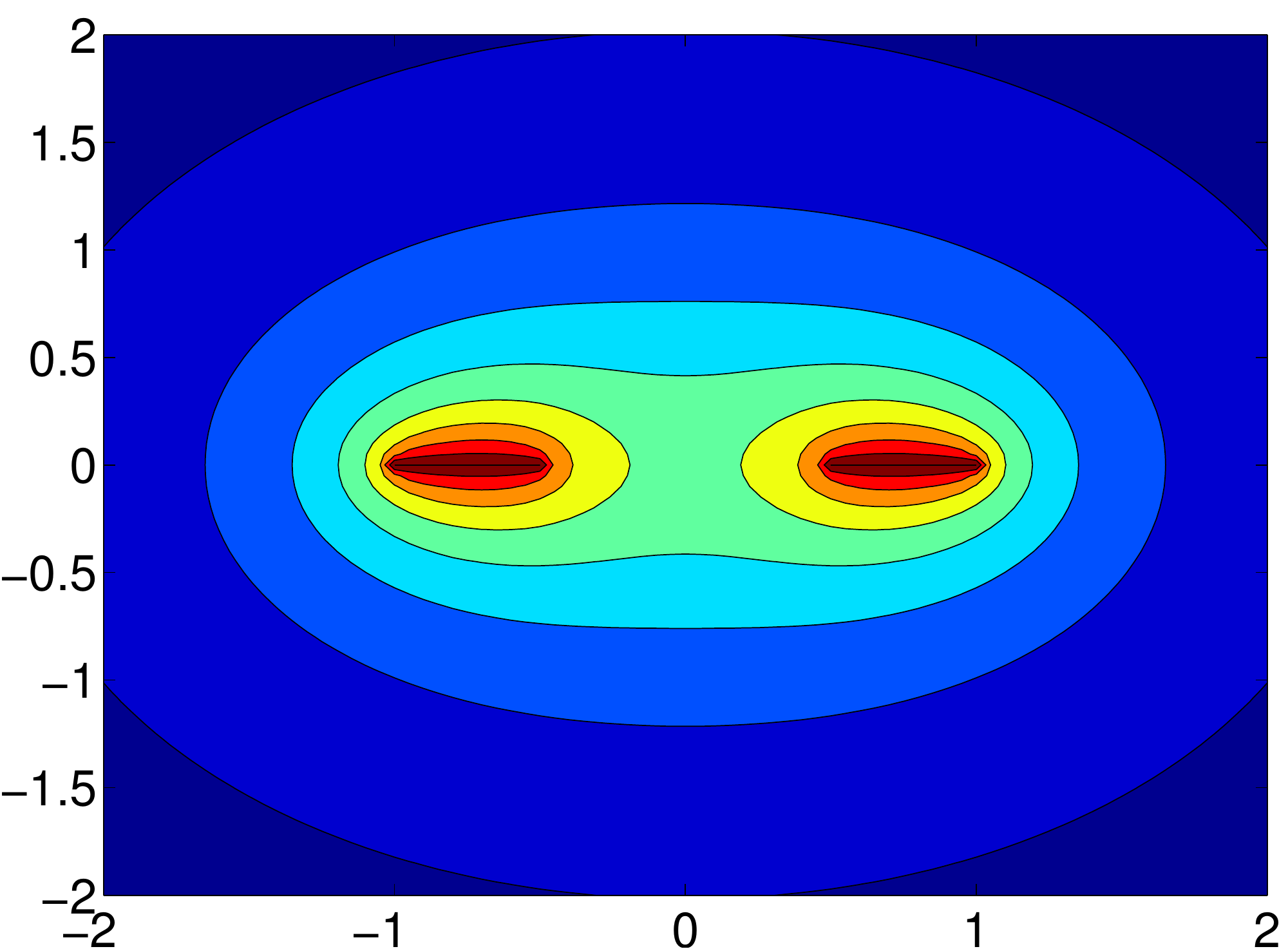}
}
\caption{Asymptotic convergence factor $R_{z_0}([-1,-\frac{1}{2}] \cup 
[\frac{1}{2}, 1])$ as a function of $z_0 \in \C$.}
\label{fig:ACF_sym_ints_3d}
\end{figure}

In Figure~\ref{fig:ACF_sym_ints_real} we plot the asymptotic convergence
factors $R_{z_0}(E_j)$ for the sets $E_j$ of the form~\eqref{eqn:Ej_sym} and 
real $z_0$ ranging from $-2$ to $2$.  Note that when $z_0$ is to the left 
or the right of the two intervals, i.e. $|z_0|>1$, the asymptotic convergence 
factors $R_{z_0}(E_j)$ are almost identical for all $j$, and they decrease 
quickly with increasing $|z_0|$.  On the other hand, when $z_0$ is between the 
two intervals the asymptotic convergence factors $R_{z_0}(E_j)$ strongly depend 
on $j$, and for a fixed $z_0$ they increase quickly with increasing $j$.  
Moreover, $R_{z_0}(E_j)$ for $|z_0|<2^{-j}$ is minimal when $z_0=0$,  i.e., 
when $z_0$ is the midpoint between the two intervals.

\begin{figure}
\begin{center}
\includegraphics[width=0.5\textwidth]{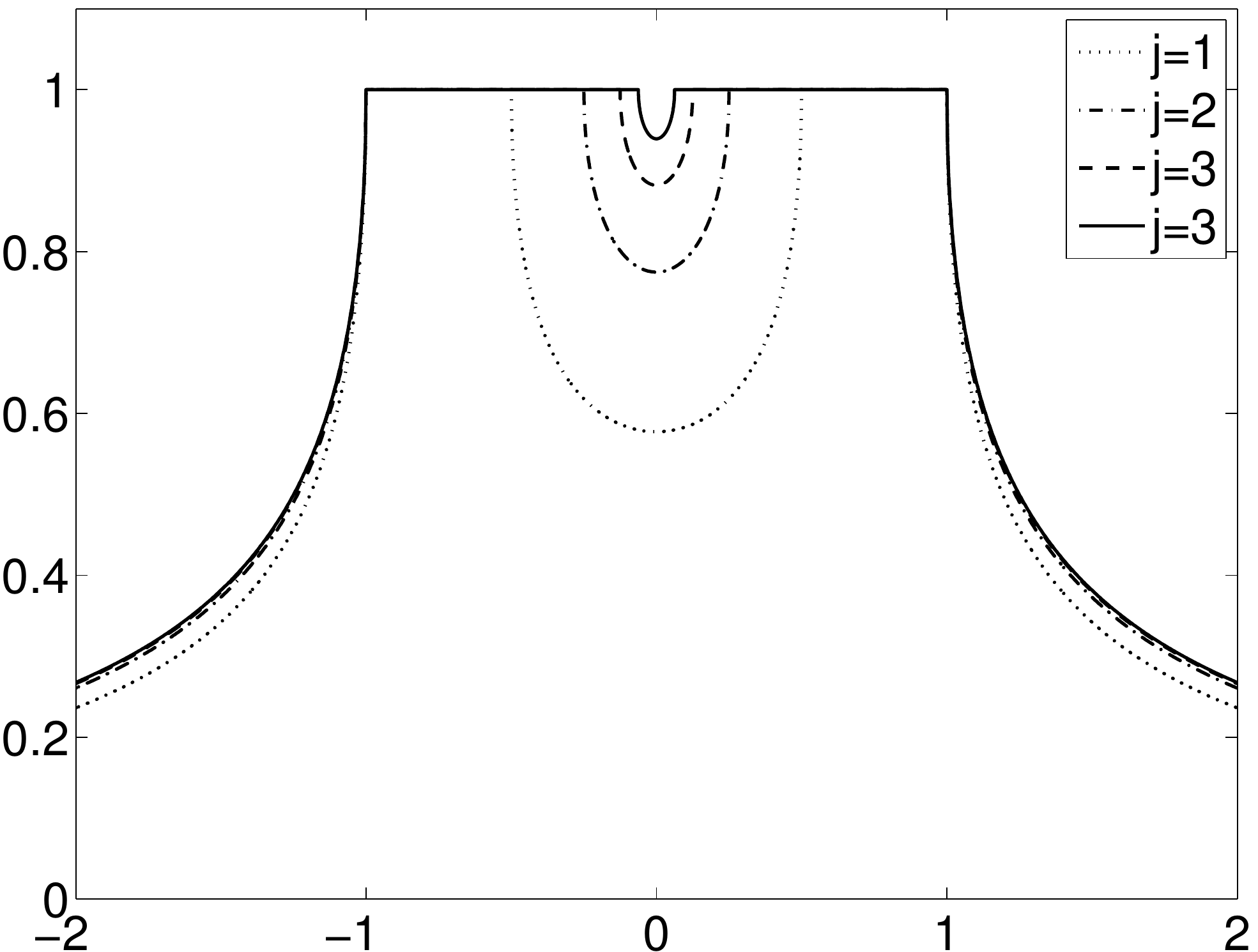}
\caption{Asymptotic convergence factors $R_{z_0}(E_j)$ for $E_j$ 
from~\eqref{eqn:Ej_sym} and $z_0 \in [-2,2]$.}
\label{fig:ACF_sym_ints_real}
\end{center}
\end{figure}

\subsection{Two arbitrary real intervals}
\label{subsec:non_sym_ints}

In this section we consider the general case of two real intervals, i.e.,
\begin{equation} \label{eqn:nonsym_ints}
E = [A, B] \cup [C, D] \quad \text{with } A < B < C < D.
\end{equation}
For such sets, the lemniscatic map and lemniscatic domain are not known 
analytically.  We therefore compute the map and the domain numerically using 
the method introduced in~\cite{NLS2015_numconf}.  This methods needs as its 
input a discretization of the boundary of the set $E$, which is assumed to 
consist of Jordan curves.  The numerical examples in~\cite{NLS2015_numconf} 
show that the method is efficient and works accurately even for domains with 
close-to-touching boundaries, non-convex boundaries, piecewise smooth 
boundaries, and of high connectivity.

We will apply two preliminary conformal maps in order to map $\widehat{\C} 
\backslash E$ for the set $E$ in~\eqref{eqn:nonsym_ints} onto a domain bounded 
by Jordan curves.
The preliminary maps are basically (inverse) Joukowski maps, as stated in the
following lemma, which can be proven by elementary means.

\begin{lemma} \label{lem:joukowski}
Let $\alpha, \beta \in \R$ with $\alpha < \beta$.  Then
\begin{equation*}
\begin{split}
\Phi &: \widehat{\C} \backslash [\alpha, \beta] \to \left\{ w \in \widehat{\C} 
: \abslr{w-\frac{\beta+\alpha}{2} } > \frac{\beta-\alpha}{4} \right\}, \\
&w = \Phi(z) = \frac{1}{2} \left( z + \frac{\beta+\alpha}{2} \pm 
\sqrt{(z-\alpha)(z-\beta)} \right),
\end{split}
\end{equation*}
where we take the branch of the square root such that $\abs{ \Phi(z) - 
\frac{\beta+\alpha}{2} } > \frac{\beta-\alpha}{4}$, is a bijective conformal 
map which is normalized at infinity by $\Phi(z) = z + \cO(1/z)$.
\end{lemma}

We now construct the lemniscatic map of $E = [A, B] \cup [C, D]$.  The 
construction is illustrated in Figure~\ref{fig:Phi_nonsym_ints} for the set
$E = [-1, -h] \cup [h^2, 1]$ with $h = 0.15$.
\begin{figure}
\centerline{
\subfigure[Original domain $\cK = \widehat{\C} \backslash E$]{
\includegraphics[width=0.5\textwidth]{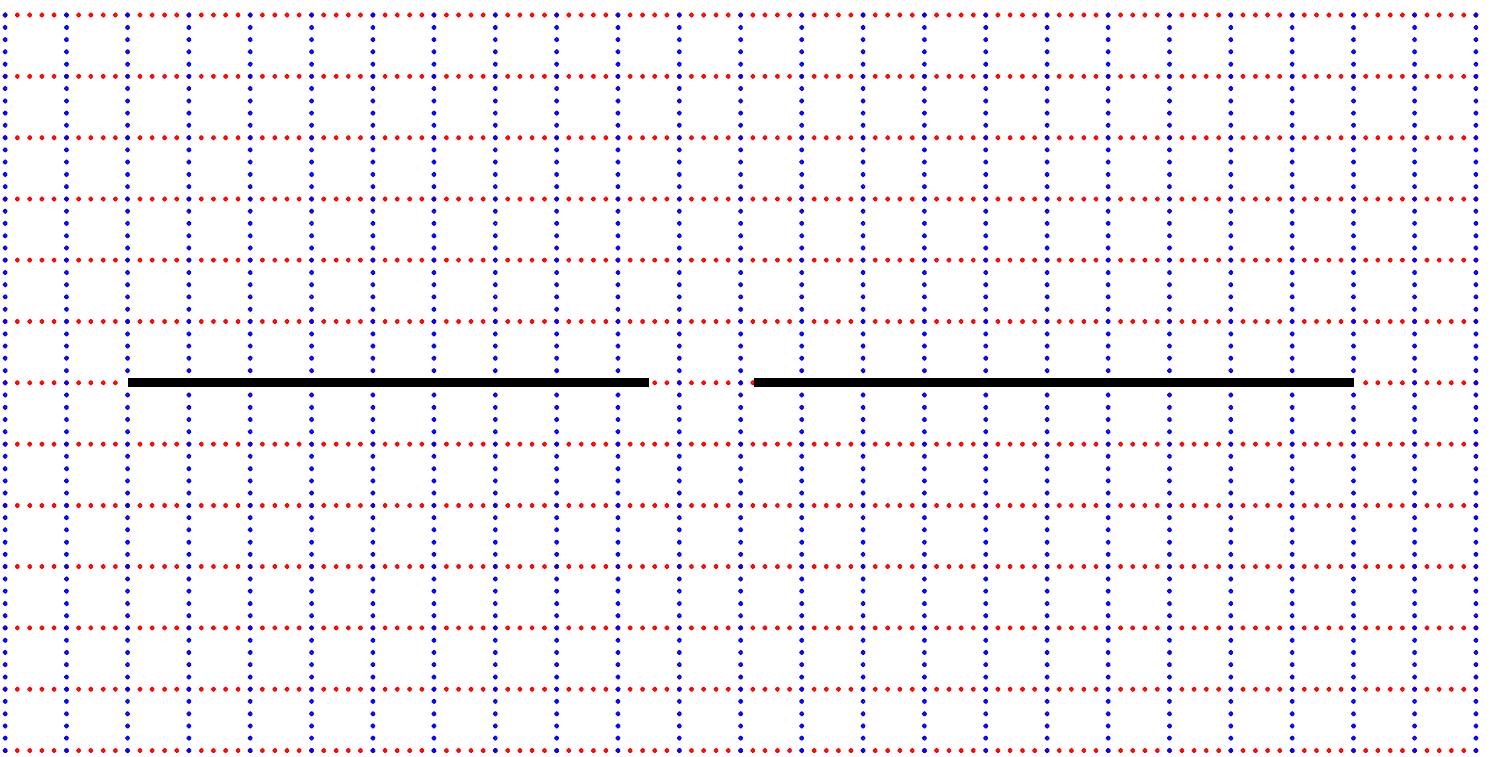}}
\subfigure[$\cK_1 = \Phi_1(\cK)$]{ \label{fig:Phi_nonsym_ints_step1}
\includegraphics[width=0.5\textwidth]{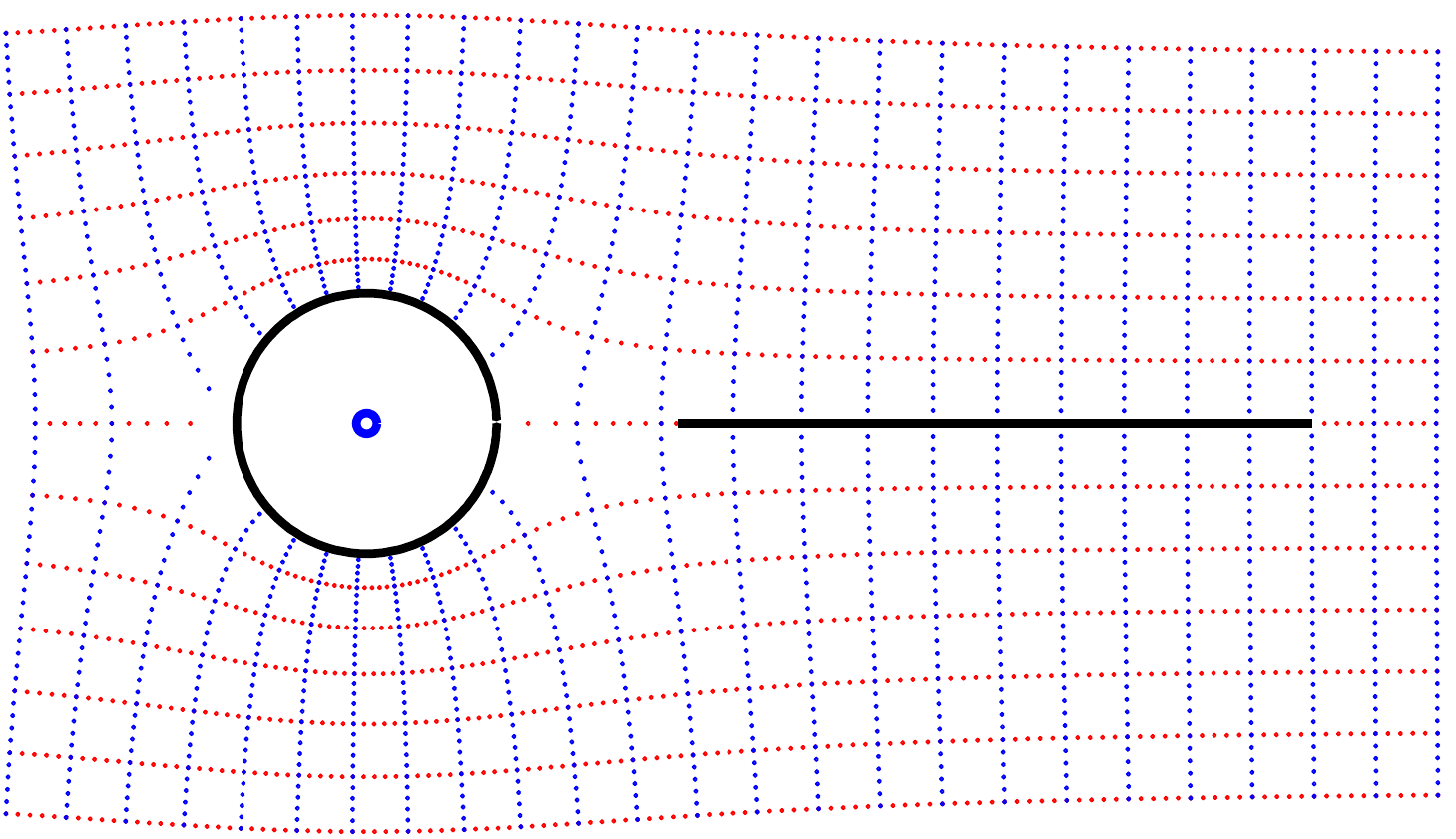}}
}
\centerline{
\subfigure[$\cK_2 = \Phi_2(\cK_1)$]{ \label{fig:Phi_nonsym_ints_step2}
\includegraphics[width=0.5\textwidth]{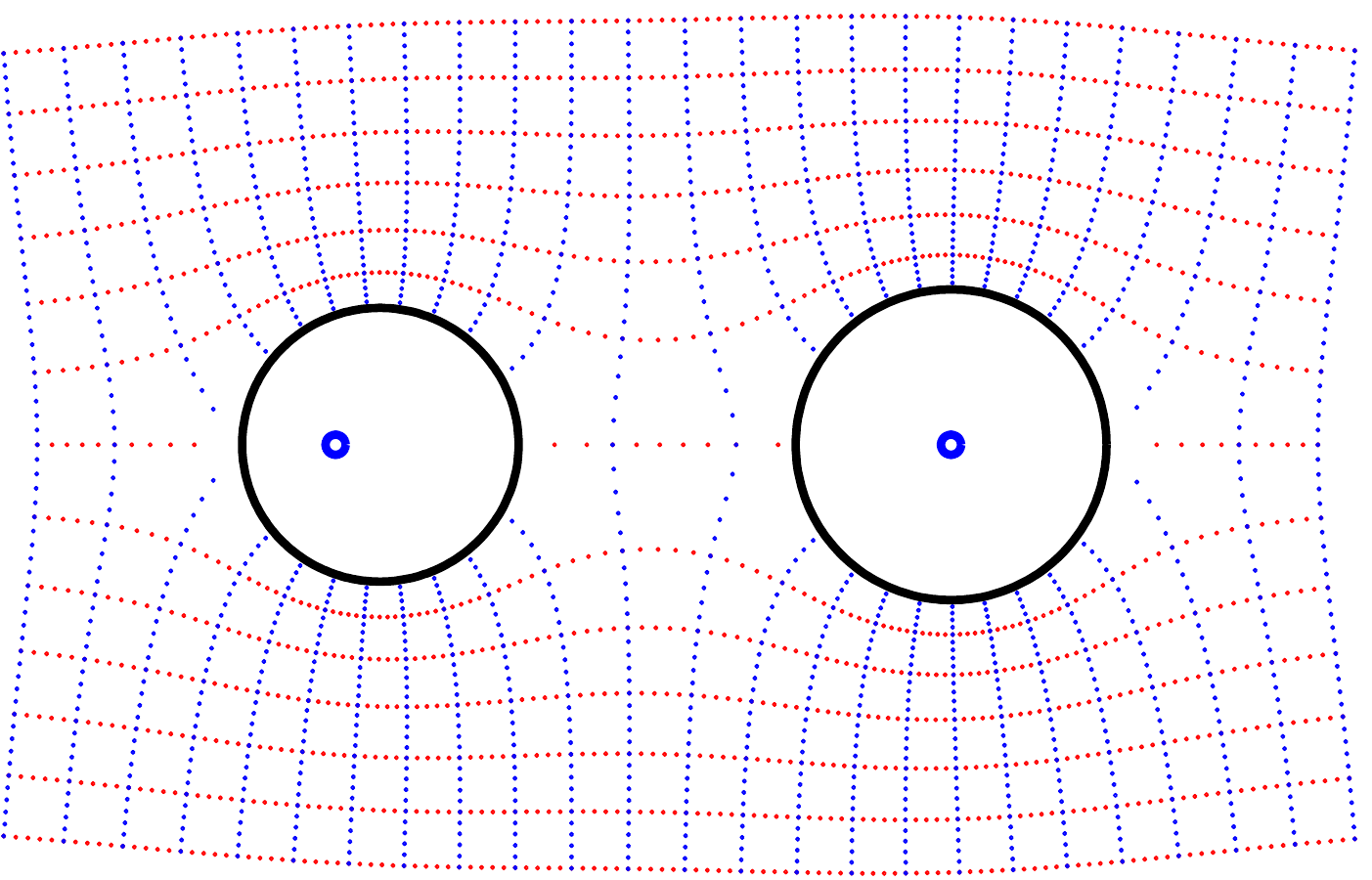}}
\subfigure[Lemniscatic domain $\cL = \Phi_3(\cK_2)$]{
\includegraphics[width=0.5\textwidth]{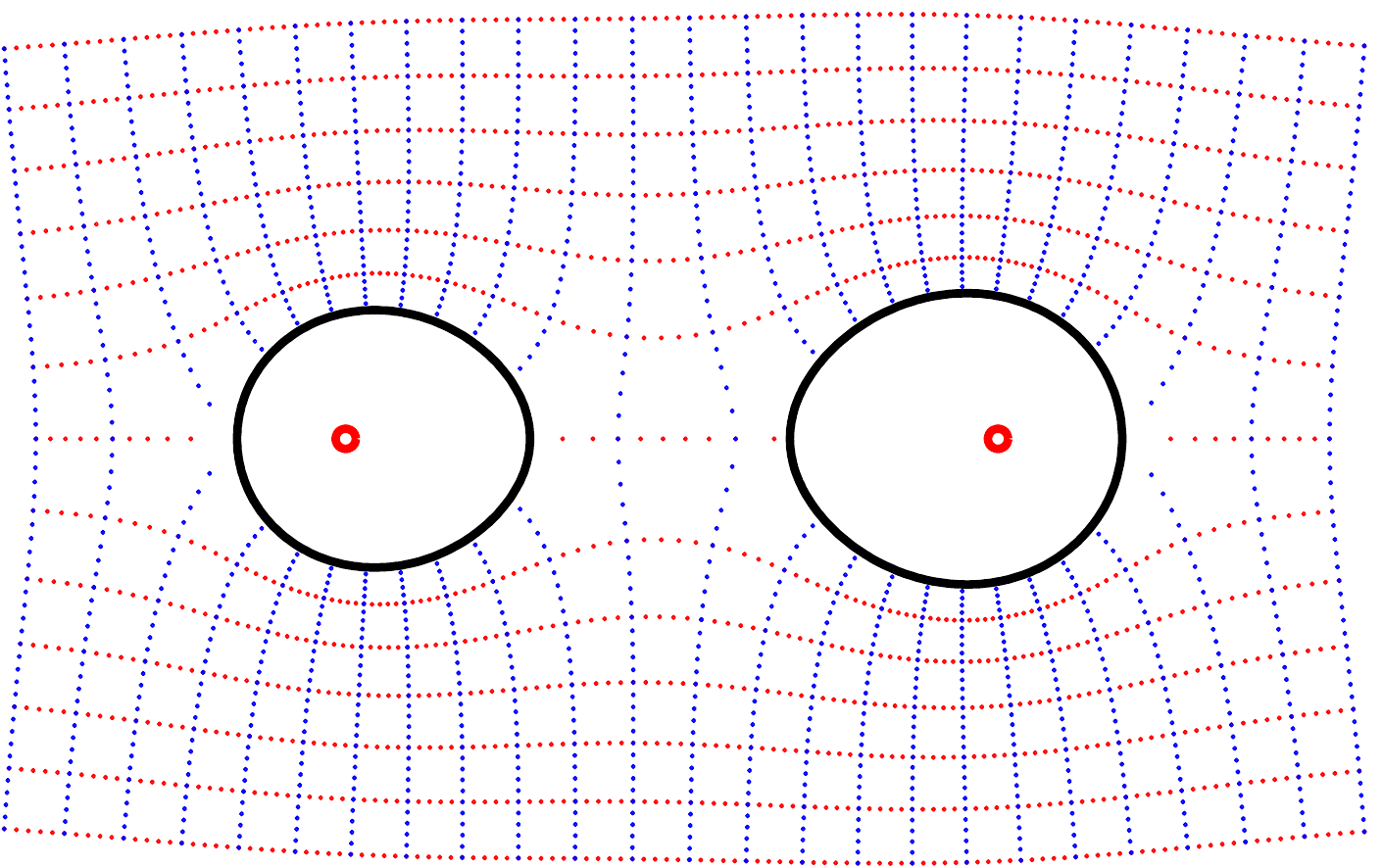}}
}
\caption{Construction of the lemniscatic map of $[-1, -h] \cup [h^2, 1]$ with 
$h = 0.15$.}
\label{fig:Phi_nonsym_ints}
\end{figure}
First we apply Lemma~\ref{lem:joukowski} to the interval $[A, B]$ to obtain the 
conformal map $\Phi_1$, mapping $\cK$ onto $\cK_1 = \Phi_1(\cK)$, which is 
the exterior of
\begin{equation*}
\{ z_1 : \abs{z_1-w_1} = r_1 \} \cup [\widetilde{C}, \widetilde{D}],
\end{equation*}
with $w_1 = \frac{B+A}{2}$, $r_1 = \frac{B-A}{4}$, $\widetilde{C} = 
\Phi_1(C)$, 
and $\widetilde{D} = \Phi_1(D)$; see Figure~\ref{fig:Phi_nonsym_ints_step1}.
In a second step, let $\Phi_2$ be given by Lemma~\ref{lem:joukowski} for the 
interval $[\widetilde{C}, \widetilde{D}]$.  Then $\Phi_2$ maps $\cK_1$ onto the 
exterior of
\begin{equation*}
\Phi_2 ( \{ z_1 : \abs{z_1-w_1} = r_1 \} ) \cup \{ z_2 : \abs{z_2-w_2} = r_2 \},
\end{equation*}
where $w_2 = \frac{\widetilde{D}+\widetilde{C}}{2}$ and $r_2 = 
\frac{\widetilde{D}-\widetilde{C}}{4}$; see 
Figure~\ref{fig:Phi_nonsym_ints_step2}.
This domain is bounded by two analytic Jordan curves, which we parametrize by
\begin{equation} \label{eqn:param}
\eta_1(t) = \Phi_2( w_1 + r_1 e^{-it} )
\quad \text{and} \quad
\eta_2(t) = w_2 + r_2 e^{-it}, \quad t \in [0, 2\pi].
\end{equation}
We then apply the numerical method from~\cite{NLS2015_numconf} to compute the 
lemniscatic domain $\cL$ and lemniscatic map $\Phi_3$ of $\cK_2 = 
\Phi_2(\cK_1)$.  As mentioned above, the input of this method is a 
discretization of the boundary, which is easily computable using the 
parameterization~\eqref{eqn:param}.  The result is the lemniscatic map
\begin{equation*}
\Phi \coloneq \Phi_3 \circ \Phi_2 \circ \Phi_1 : \widehat{\C} \backslash E \to 
\cL,
\end{equation*}
where $\Phi_1$ and $\Phi_2$ are given analytically as above, and $\Phi_3$ is 
computed numerically.  The output of the numerical method 
from~\cite{NLS2015_numconf} are the parameters of $\cL$ and the boundary values 
of $\Phi_3$ at the discretization points on the boundary.
The values of $\Phi_3$ at other points can be computed by Cauchy's 
integral formula for domains with infinity as interior point, applied to the 
function $\Phi_3(z) - z$, which is analytic in $\cK_2$ and vanishes at 
infinity.  We thus have
\begin{equation*}
\Phi_3(z) = z + \frac{1}{2 \pi i} \int_{\partial \cK_2} 
\frac{\Phi_3(\zeta)-\zeta}{\zeta-z} \, d\zeta,
\end{equation*}
where the boundary is parametrized such that $\cK_2$ is to the left of the 
contour;  see~\cite{NLS2015_numconf} for details on the practical computation 
of this integral.
Note that this numerical method extends to any finite number of intervals.

With the lemniscatic map $\Phi$ and lemniscatic domain $\cL$ available, we can 
compute the Faber--Walsh polynomials by
\begin{equation*}
b_k(z)
= \frac{1}{2 \pi i} \int_\gamma \frac{u_k(\Phi(\zeta))}{\zeta-z} \, d\zeta
= \frac{1}{2 \pi i} \int_\gamma \frac{\prod_{j=1}^k 
(\Phi(\zeta)-\alpha_j)}{\zeta-z} \, d\zeta,
\end{equation*}
see~\eqref{eqn:def_bk}, where $\gamma$ is any positively oriented contour 
containing $E$ and $z$ in its interior, and where the sequence 
$(\alpha_j)_{j=1}^\infty$ is chosen from the foci of the lemniscatic domain 
as indicated in the discussion below Lemma~\ref{lem:alpha_n}.

Figure~\ref{fig:bn_nonsym_ints} shows some computed Faber--Walsh polynomials of 
even (left) and odd (right) degrees for the set $E = [-1, -\frac{1}{4}] \cup 
[\frac{1}{3}, 1]$.  Unlike in the case of two equal intervals, the 
Faber--Walsh polynomials for two unequal intervals are in general not
Chebyshev polynomials, nor are the normalized Faber--Walsh polynomials the
optimal polynomials.
We numerically compute the Faber--Walsh polynomials $b_{k,j}$ for the sets
\begin{equation}\label{eqn:Ej_nonsym}
E_j\coloneq [-1,-2^{-j}] \cup [3^{-1},1], \quad j=1,2,3,4.
\end{equation}
Figure~\ref{fig:nfw_nonsym_ints_norm} (left) shows the corresponding values 
$\norm{b_{k,j}}_E / \abs{b_{k,j}(0)}$ for $k = 1, \ldots, 30$.
As for the equal intervals, a comparison with the values $R_0(E_j)^k$ shows 
that the convergence speed to zero of the norms matches 
closely the predicted asymptotic rate, already for small values of $k$.
Unlike in the case of two equal intervals, the sequence of norms has a few 
irregular jumps, which are due to the lack of symmetry in the problem.
More precisely, all jumps happen when one of the foci $a_1, a_2$ of the 
lemniscatic domain appears twice in a row in the sequence 
$(\alpha_j)_{j=1}^\infty$.  This happens in the construction of the sequence 
as described below Lemma~\ref{lem:alpha_n}, since for two unequal intervals 
the exponents $m_1$ and $m_2$ of the lemniscatic domain are different.

\begin{figure}
\centerline{
\includegraphics[width=0.5\textwidth]{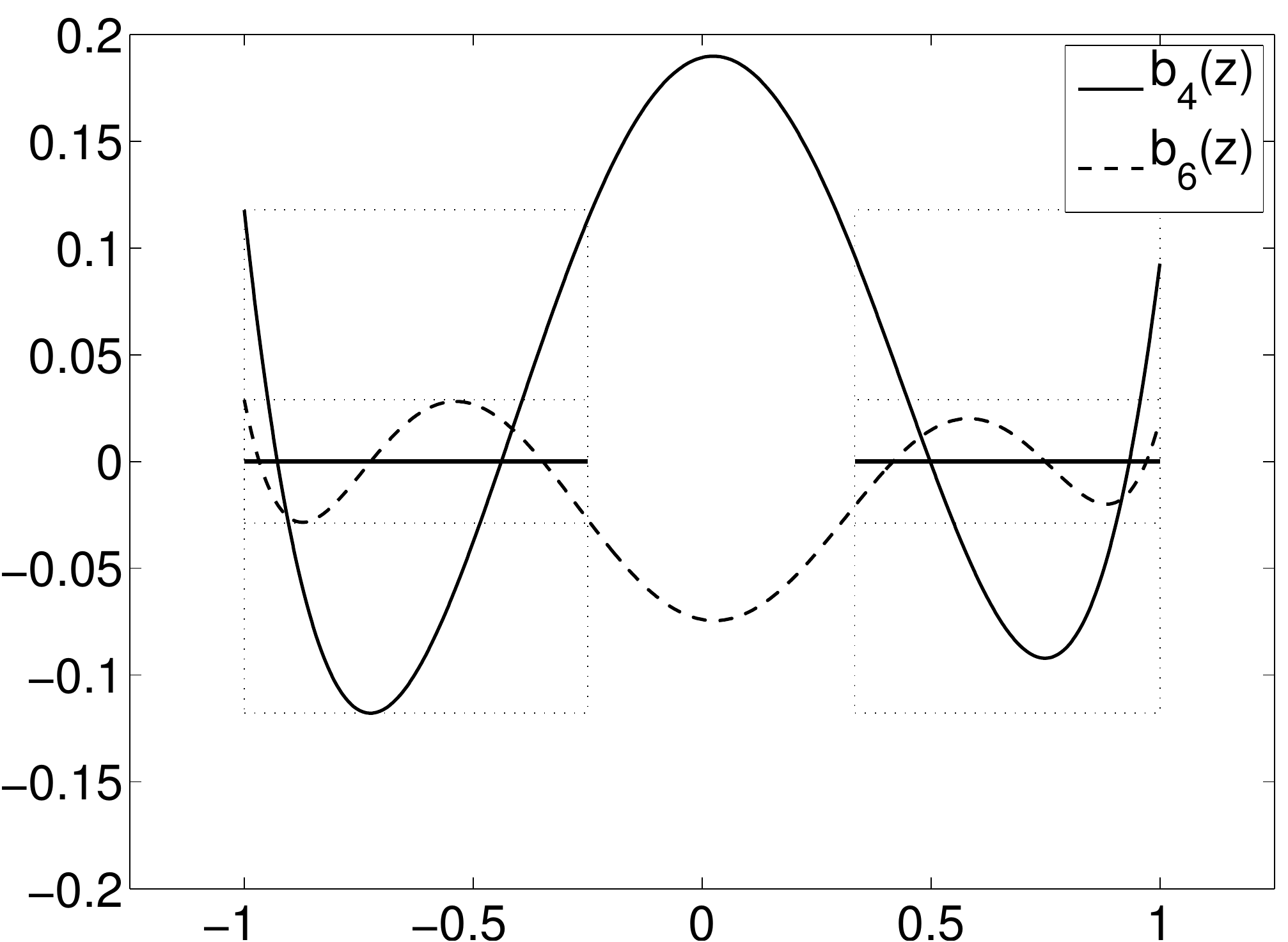}
\includegraphics[width=0.5\textwidth]{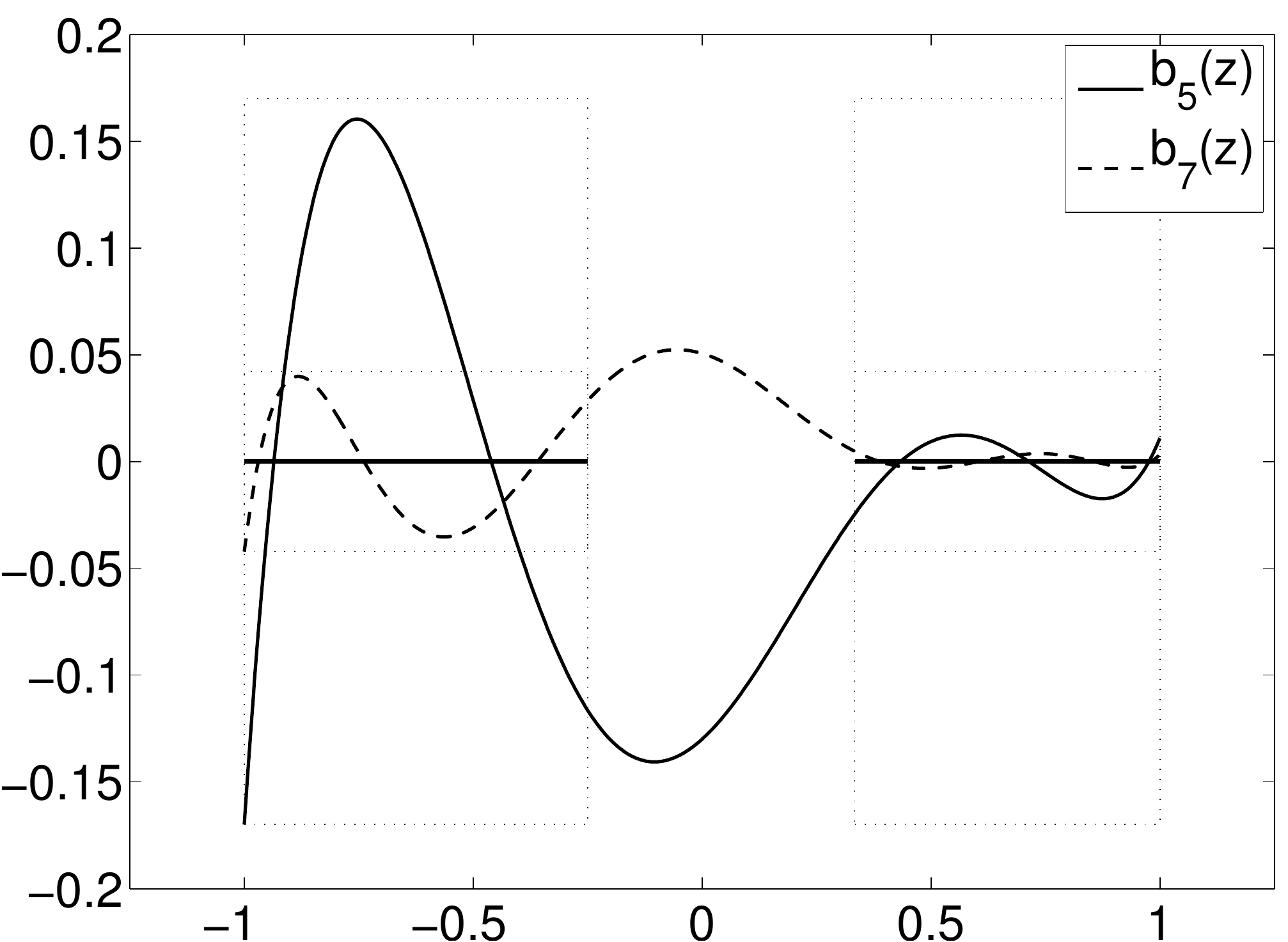}
}
\caption{Faber--Walsh polynomials $b_k$ for $E = [-1, -\frac{1}{4}] \cup 
[\frac{1}{3}, 1]$.}
\label{fig:bn_nonsym_ints}
\end{figure}

\begin{figure}
\centerline{
\includegraphics[width=0.5\textwidth]{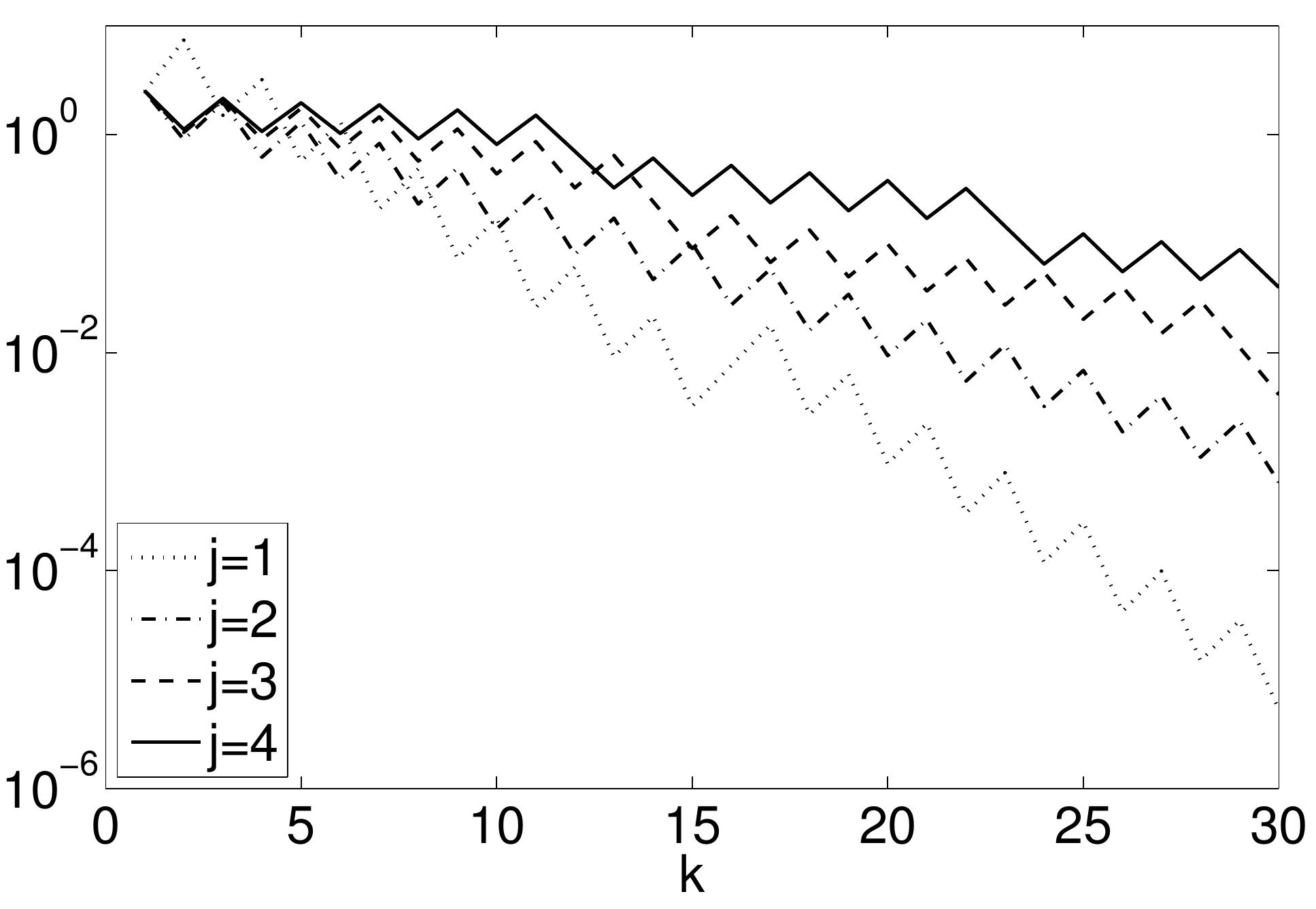}
\includegraphics[width=0.5\textwidth]{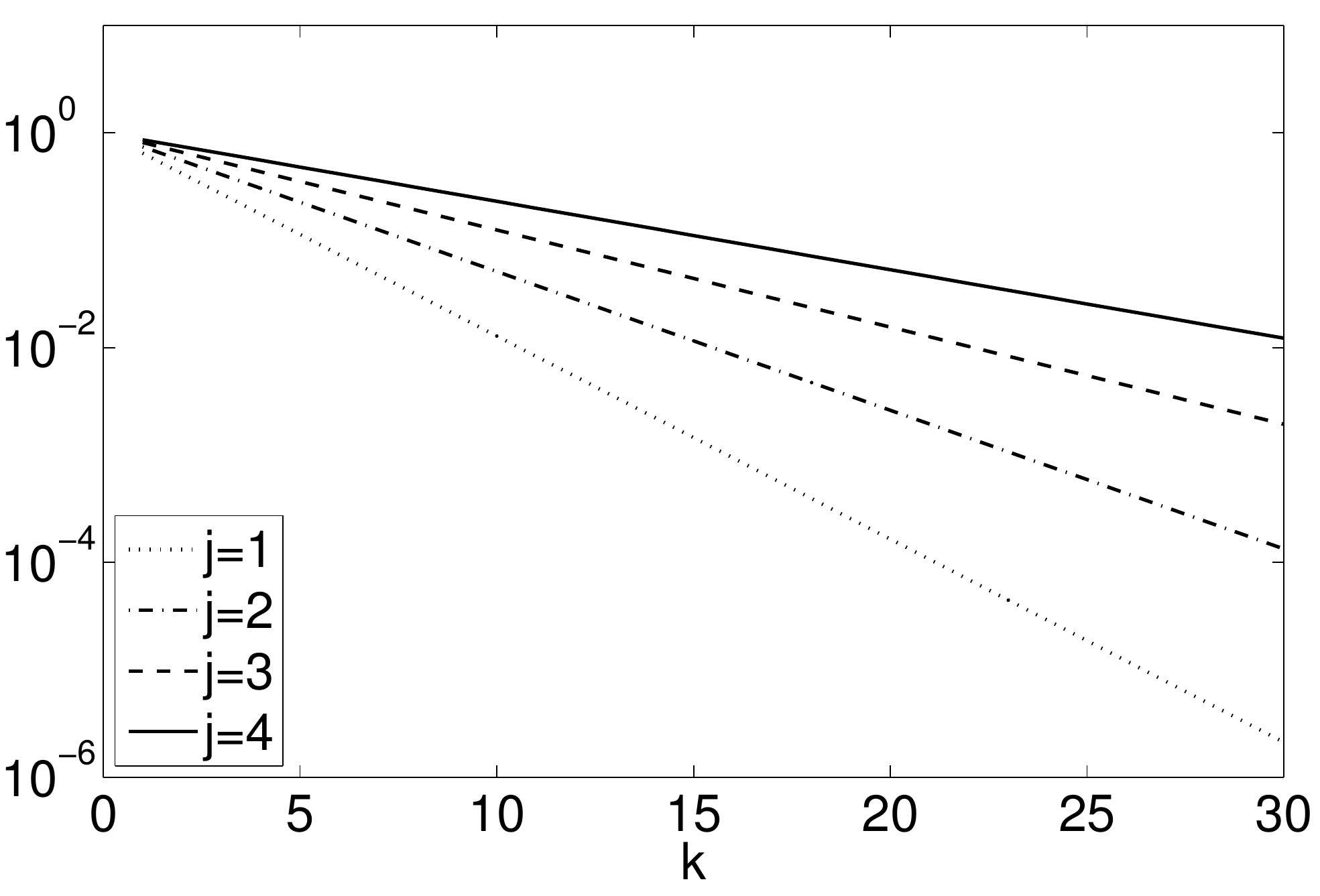}
}
\caption{The values $\frac{\|b_{k,j}\|_{E_j}}{|b_{k,j}(0)|}$ (left) and 
$R_0(E_j)^k$ (right) for the sets $E_j$ from~\eqref{eqn:Ej_nonsym}.}
\label{fig:nfw_nonsym_ints_norm}
\end{figure}

Finally, the numerical method from~\cite{NLS2015_numconf} yields the 
lemniscatic map $\Phi$, as well as the parameters of the lemniscatic domain 
$\cL$.  Therefore, the asymptotic convergence factor can be numerically 
computed by its characterization~\eqref{eqn:ACF_green}.
In Figure~\ref{fig:ACF_nonsym_ints_3d} we plot the asymptotic convergence 
factor $R_{z_0}( [-1, - \frac{1}{2}] \cup [\frac{1}{3}, 1] )$ as a function of 
$z_0 \in \C$.

\begin{figure}
\centerline{
\includegraphics[width=0.5\textwidth]{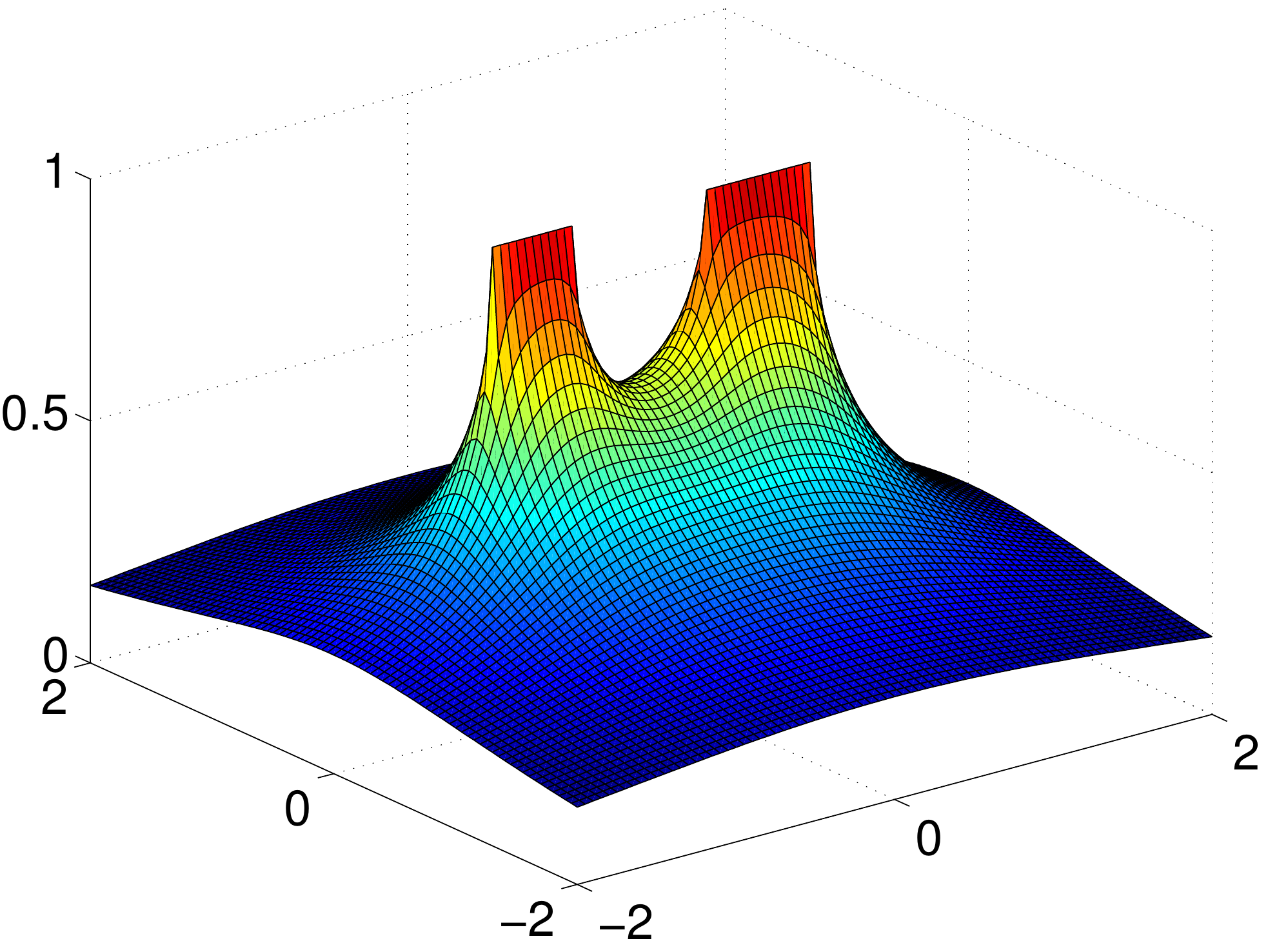}
\includegraphics[width=0.5\textwidth]{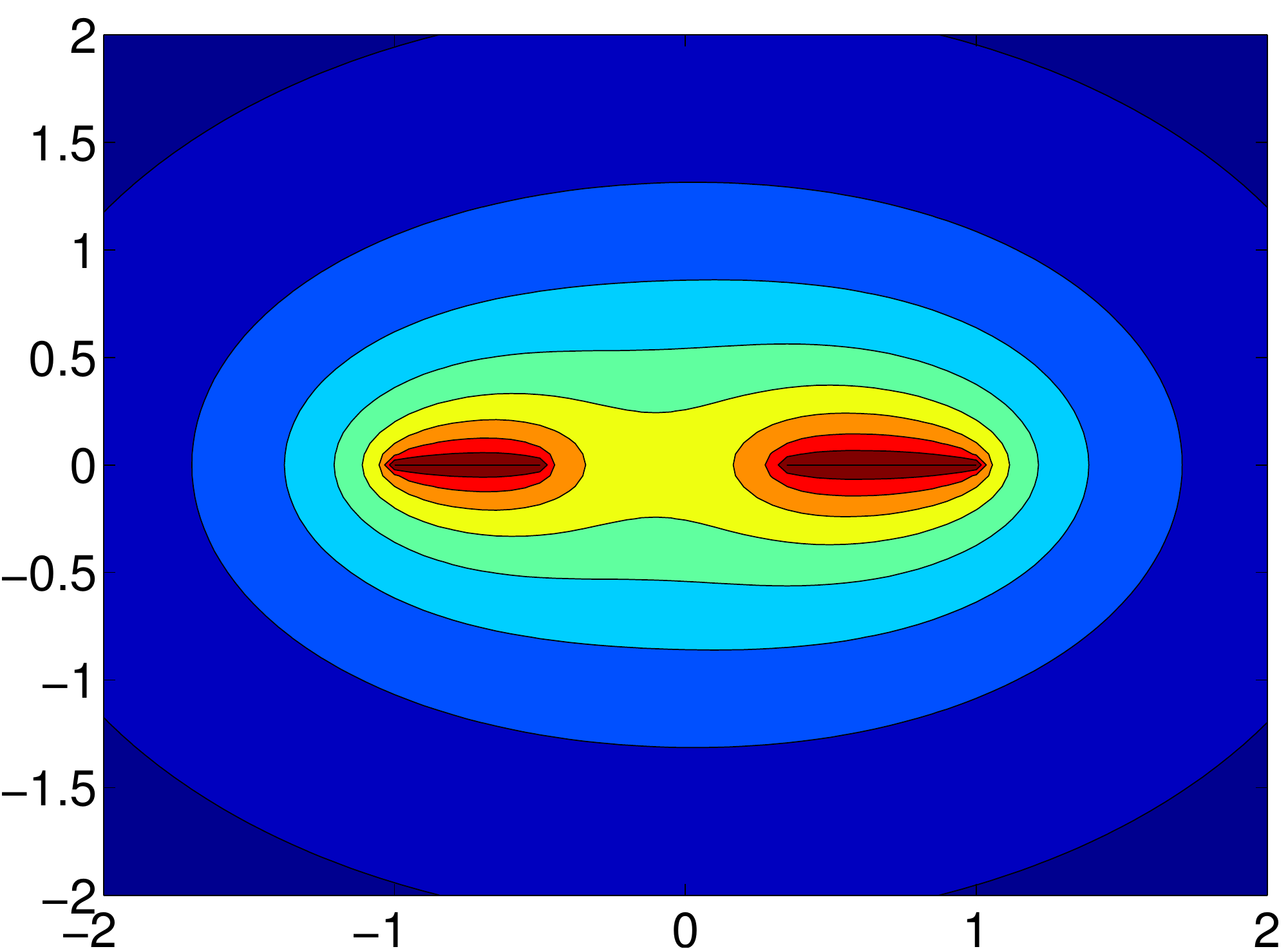}
}
\caption{Asymptotic convergence factor $R_{z_0}([-1,-\frac{1}{2}] \cup 
[\frac{1}{3}, 1])$ as a function of $z_0 \in \C$.}
\label{fig:ACF_nonsym_ints_3d}
\end{figure}

In Figure~\ref{fig:ACF_nonsym_ints_real} we plot the asymptotic convergence
factors $R_{z_0}(E_j)$ for the sets $E_j$ from~\eqref{eqn:Ej_nonsym} and real 
$z_0$ ranging from $-2$ to $2$.
Similar remarks as in the case of two equal intervals apply, with one 
exception:
Here the computation shows that $R_{z_0}(E_j)$ for $z_0 \in [-2^{-j}, 3^{-1}]$ 
attains its minimum not at the midpoint between the two intervals, but slightly 
to its left for $j = 1$, and to its right for $j = 2, 3, 4$.
A similar observation has already been made by 
Fischer~\cite[Example~3.4.5]{Fischer1996}.

\begin{figure}
\begin{center}
\includegraphics[width=0.5\textwidth]{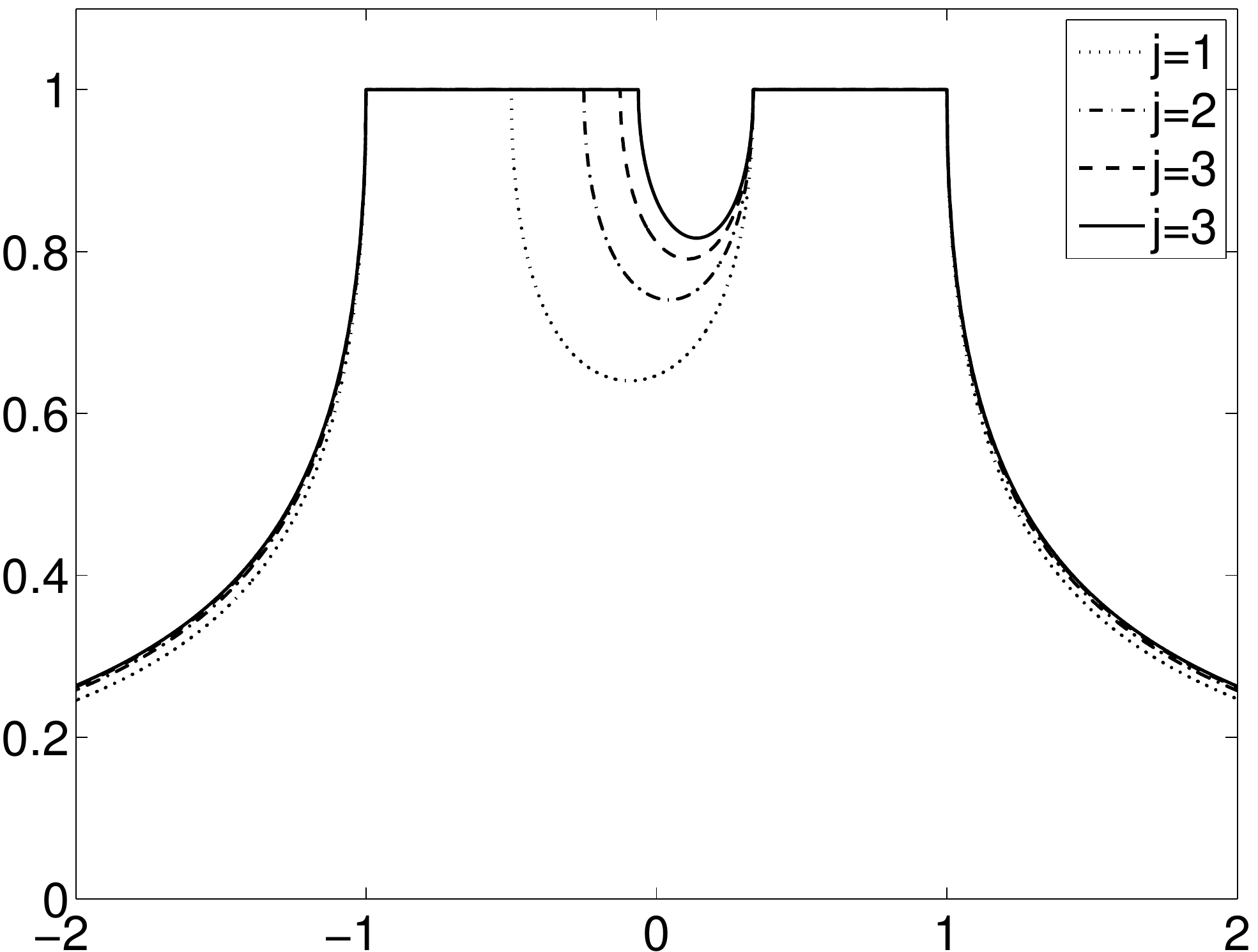}
\caption{Asymptotic convergence factors $R_{z_0}(E_j)$ for $E_j$ 
from~\eqref{eqn:Ej_nonsym} and $z_0 \in [-2,2]$.}
\label{fig:ACF_nonsym_ints_real}
\end{center}
\end{figure}

 
\section{Two non-real examples}
\label{sect:examples}

In this section we give examples of Faber--Walsh polynomials for sets $E$ that 
are not subsets of the real line.

\subsection{Two disks}
\label{subsec:disks}

We consider sets $E$ consisting of two disks of the same radius, i.e.,
\begin{equation*}
E = E(z_0, r) \coloneq \{ z \in \C : \abs{z-z_0} \leq r \} \cup \{ z \in \C : 
\abs{z+z_0} \leq r \},
\end{equation*}
where we take $z_0, r \in \R$ with $0 < r < z_0$.  The lemniscatic map of 
$E(z_0, r)$ is known analytically from~\cite[Theorem~4.2]{SeteLiesen2015}, and 
the lemniscatic domain is of the form
\begin{equation*}
\cL = \cL(z_0, r) = \{ w \in \widehat{\C} : \abs{w-a_1}^{1/2} \abs{w+a_1}^{1/2} 
> \mu \}
\end{equation*}
for some $a_1 > 0$ and $\mu > 0$.  With these $\cL$ and $\Phi$ we obtain the 
Faber--Walsh polynomials for $E(z_0, r)$ and $(a_1, -a_1, a_1, -a_1, \ldots)$ 
by their integral representation; see also the discussion in 
Section~\ref{subsec:non_sym_ints}.

In Figure~\ref{fig:bn_2disks_pp} we plot the phase portraits of several 
Faber--Walsh polynomials for $E(1, 0.8)$; 
see~\cite{Wegert2012,WegertSemmler2011} for details on phase portraits.
The figure shows that the $k$ roots of $b_k$ are all contained in $E$.
\begin{figure}
\centerline{
\includegraphics[width=0.5\textwidth]{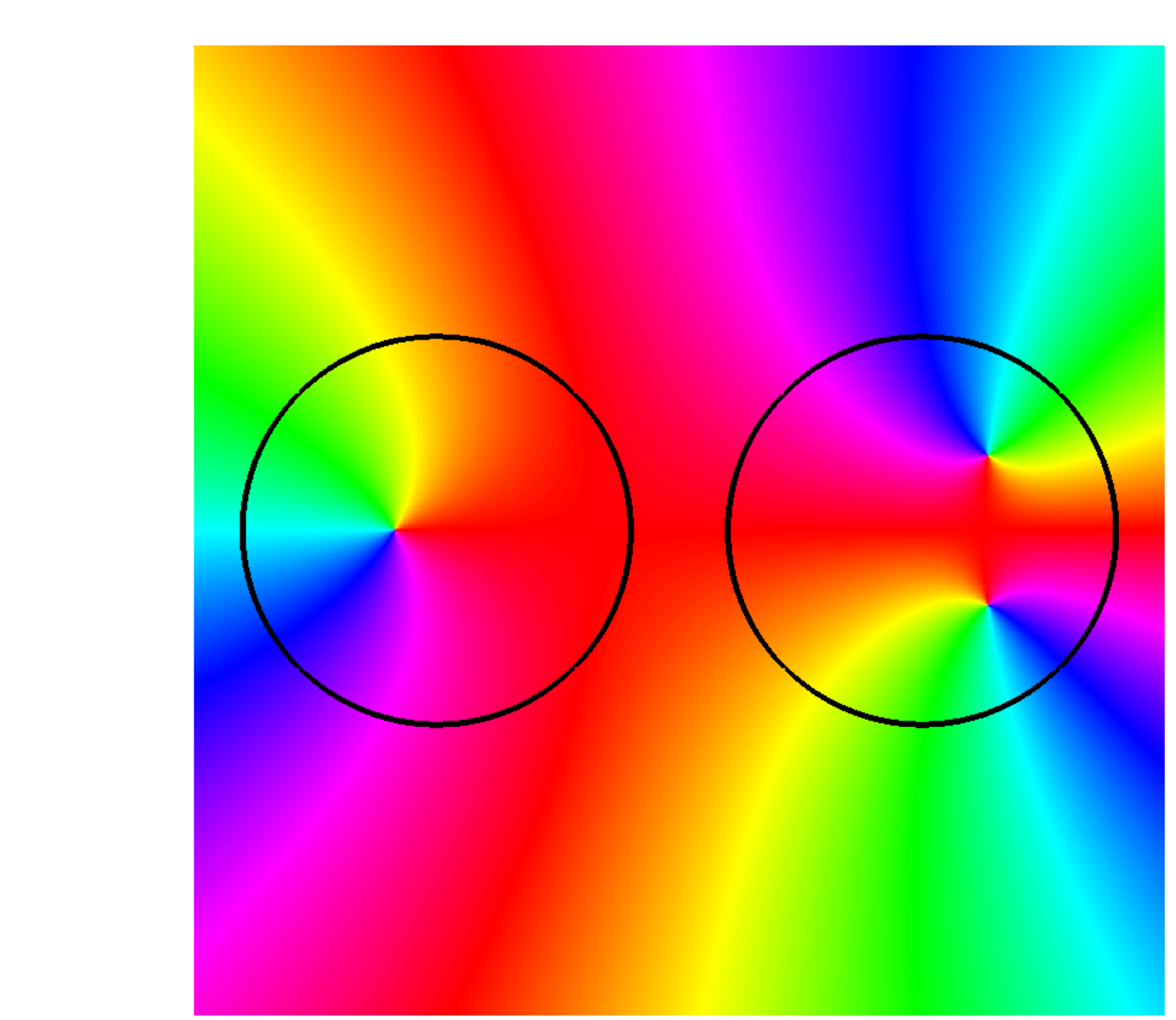}
\includegraphics[width=0.5\textwidth]{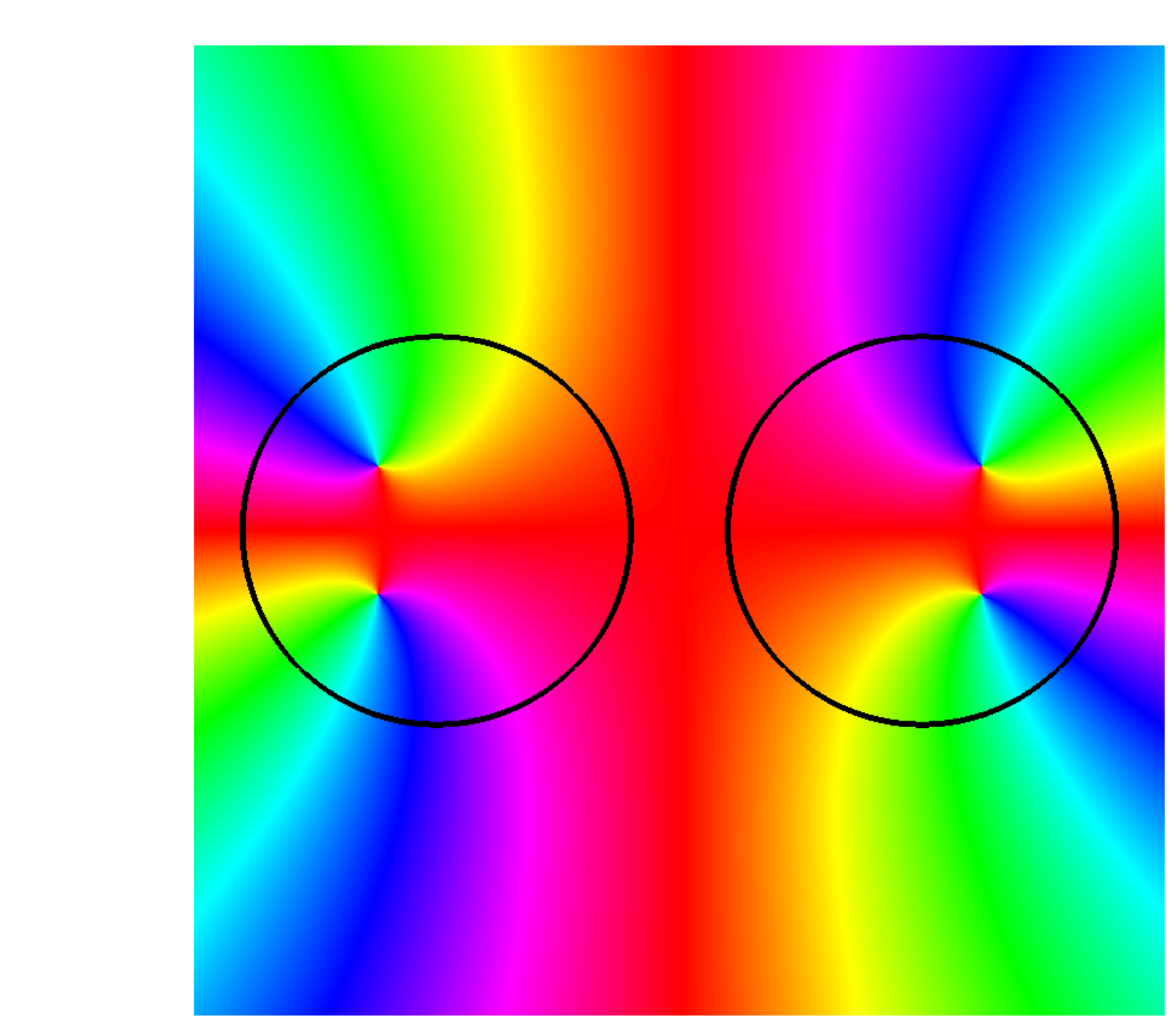}
}
\centerline{
\includegraphics[width=0.5\textwidth]{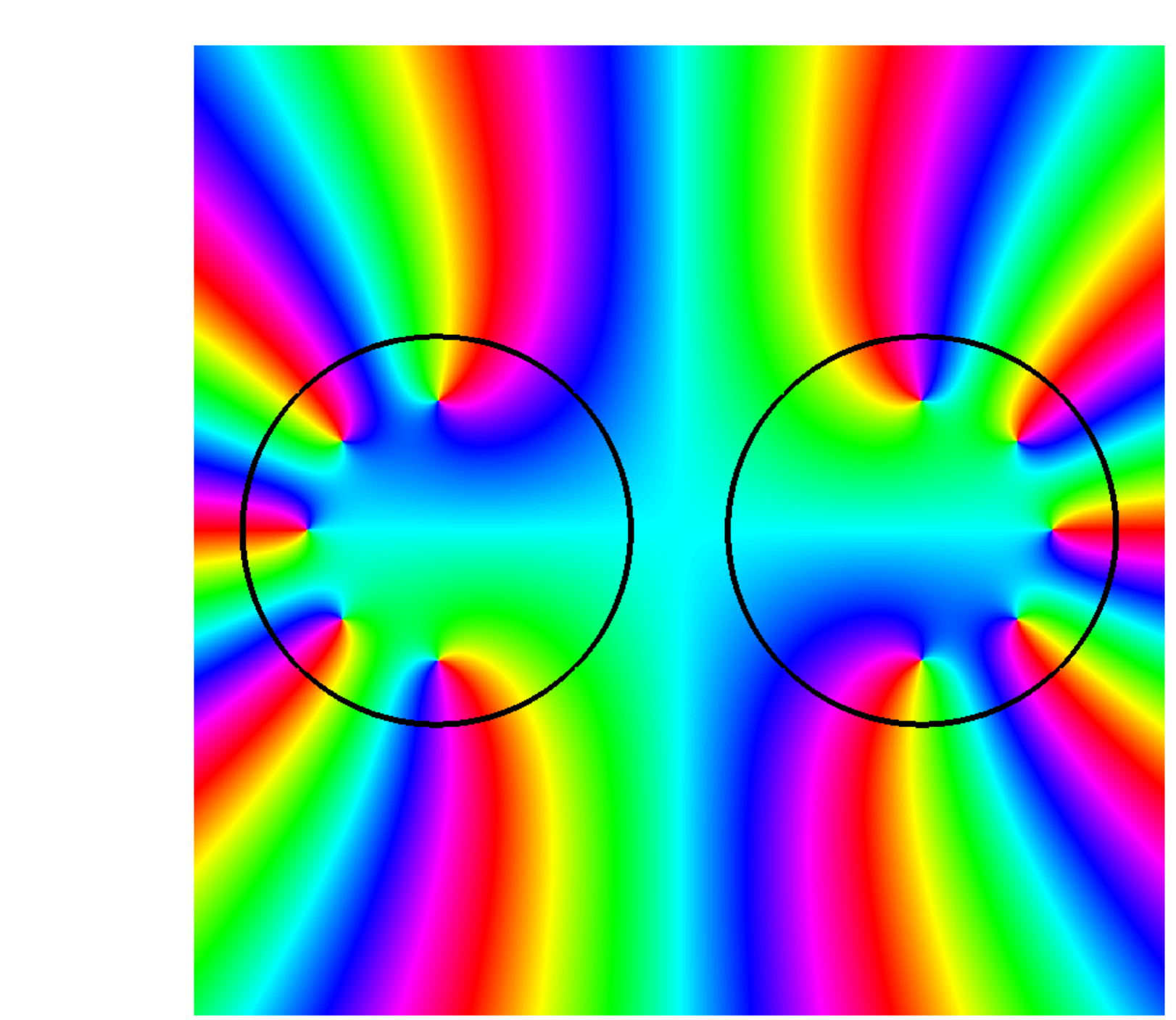}
\includegraphics[width=0.5\textwidth]{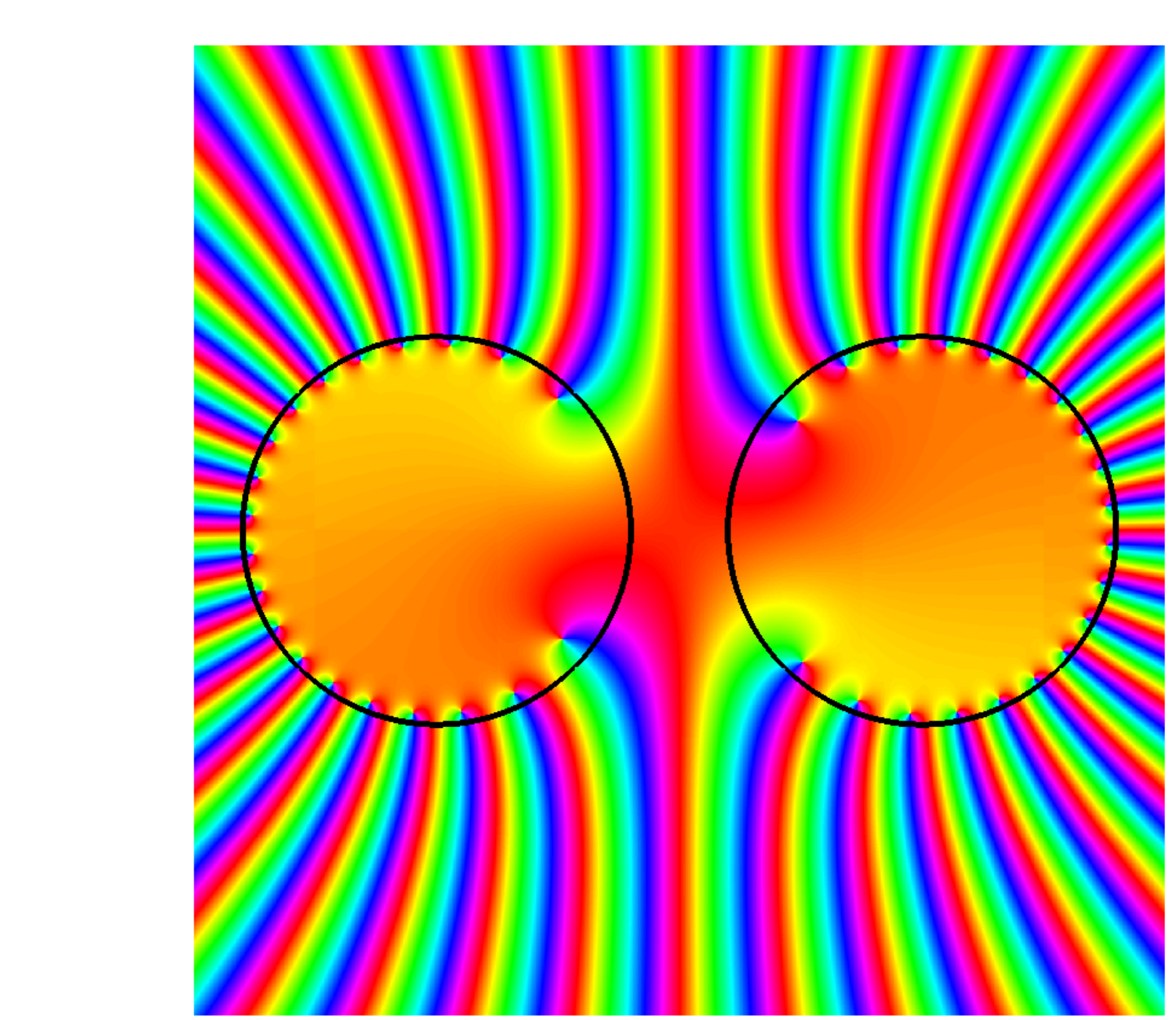}
}
\caption{Phase portraits of Faber--Walsh polynomials $b_k$ for $E(1, 0.8)$ for
$k = 3, 4, 10, 40$.}
\label{fig:bn_2disks_pp}
\end{figure}
We further compute the Faber--Walsh polynomials $b_{k,r}$ for the sets
\begin{equation}
E_r = E(1, r) \quad \text{for } r = 0.5, 0.7, 0.9, \label{eqn:Er_disks}
\end{equation}
and the sequence $(a_1, -a_1, a_1, -a_1, \ldots)$.  In 
Figure~\ref{fig:nfw_2disks_norm} (left) we plot the values 
$\frac{\norm{b_{k,r}}_{E_r}}{\abs{b_{k,r}(0)}}$ for $k = 1, \ldots, 30$.
\begin{figure}
\centerline{
\includegraphics[width=0.5\textwidth]{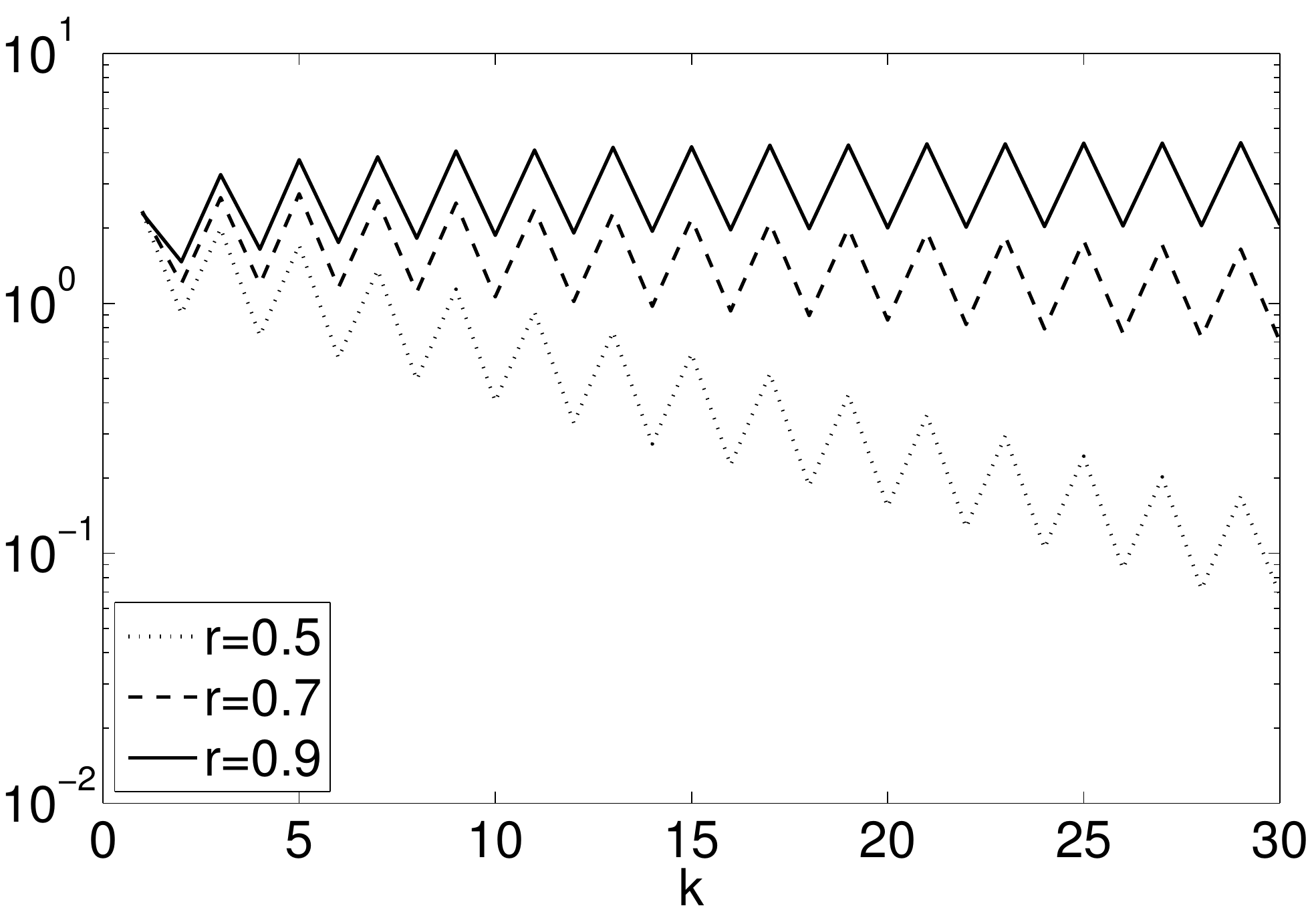}
\includegraphics[width=0.5\textwidth]{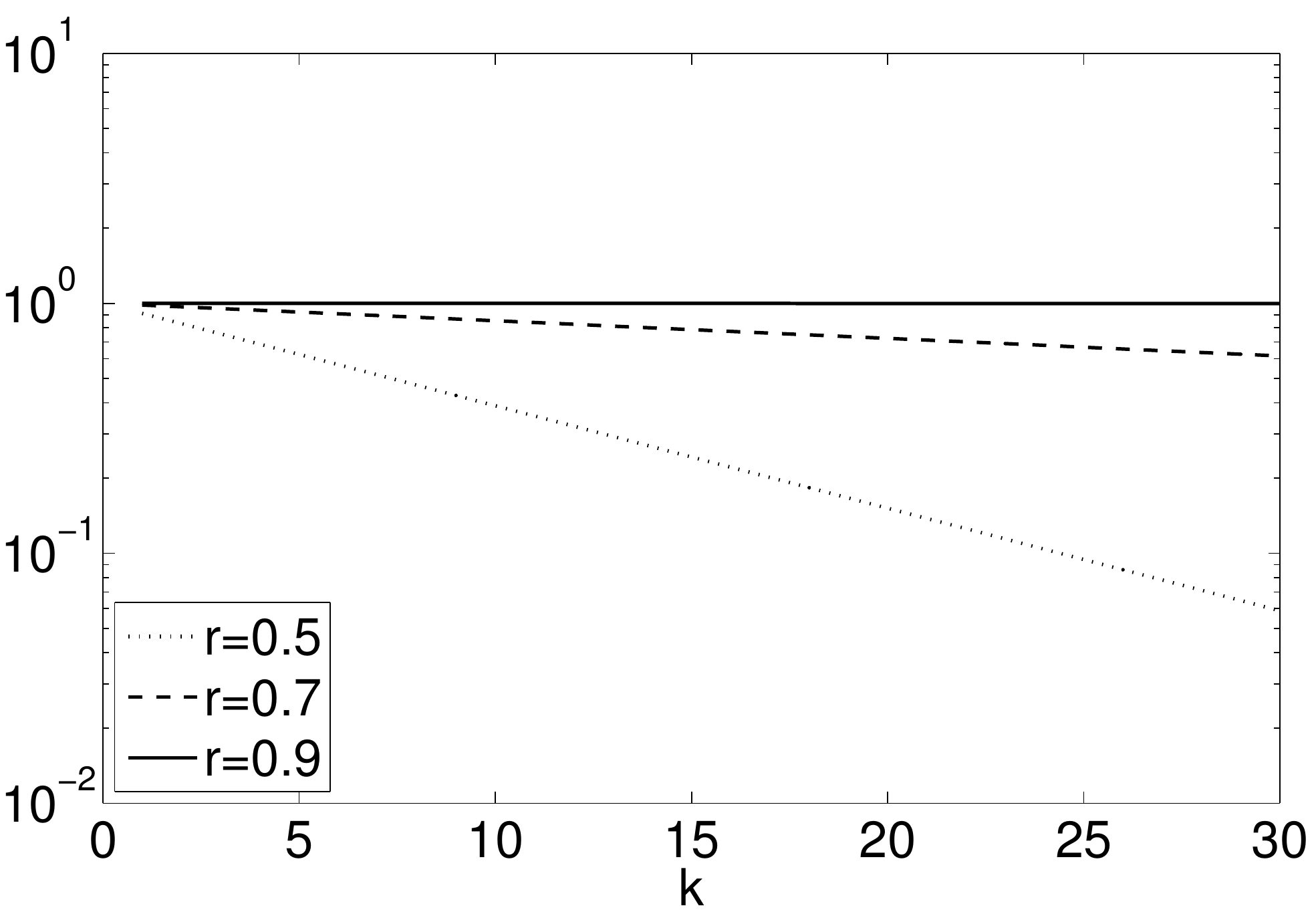}
}
\caption{The values $\frac{\|b_{k,r}\|_{E_r}}{|b_{k,r}(0)|}$ (left) and 
$R_0(E_r)^k$ (right) for the sets $E_r$ from~\eqref{eqn:Er_disks}.}
\label{fig:nfw_2disks_norm}
\end{figure}
As in the case of two intervals, we observe that the convergence speed to zero 
of the norms almost exactly matches the rate predicted by the asymptotic 
analysis, already for small $k$.  The numerically computed asymptotic 
convergence factors (rounded to five digits) for the three sets $E_r$ are 
$R_0(E_{0.5}) = 0.9099$, $R_0(E_{0.7}) = 0.9839$ and $R_0(E_{0.9}) = 0.9999$, 
which in particular explains the slow convergence to zero for $r = 0.7$ and the 
(almost) stagnation for $r = 0.9$.

Figure~\ref{fig:ACF_2disks_3d} shows the numerically computed asymptotic 
convergence factor $R_{z_0}(E(1,0.7))$ as a function of $z_0 \in \C$.

\begin{figure}
\centerline{
\includegraphics[width=0.5\textwidth]{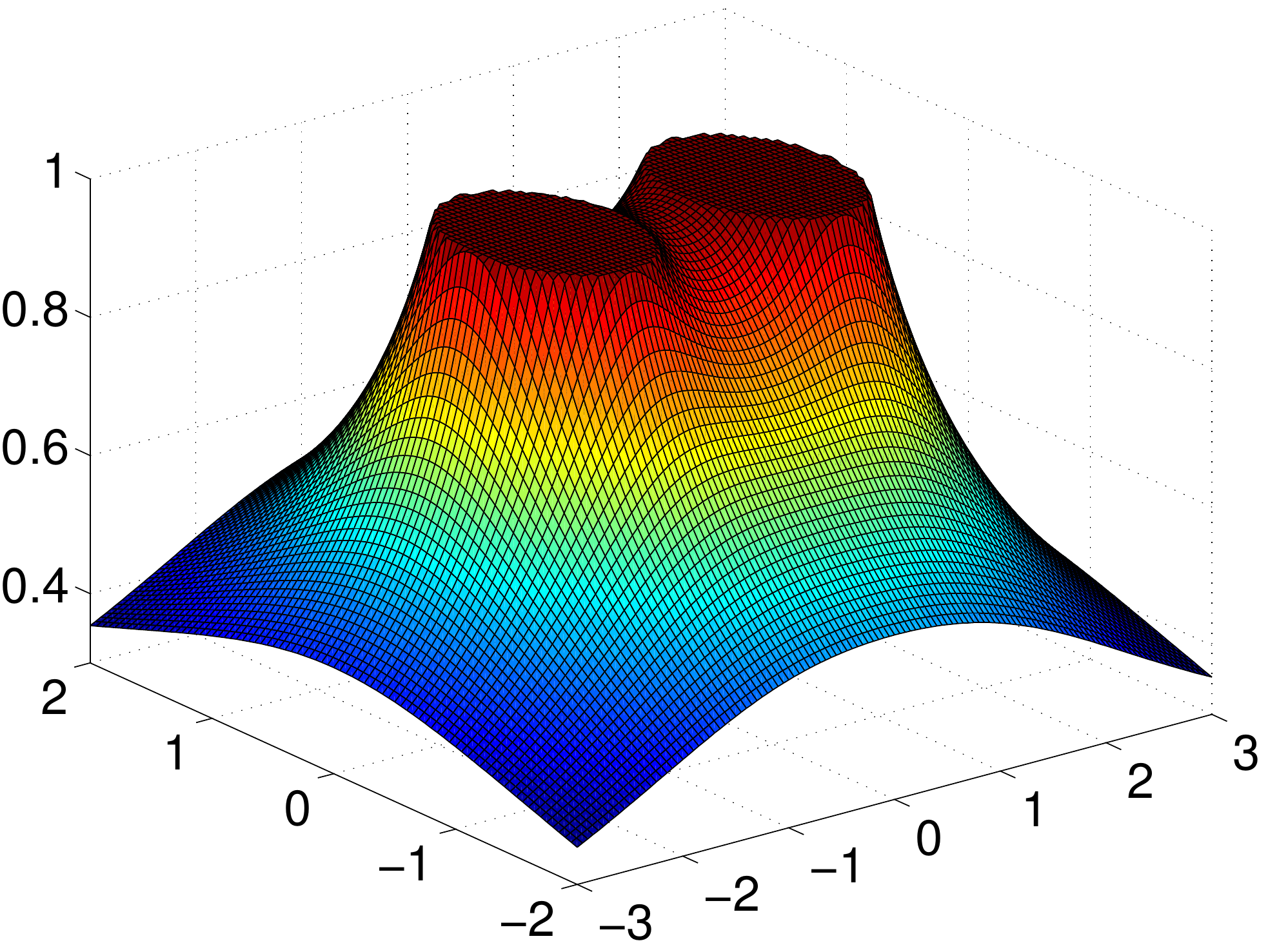}
\includegraphics[width=0.5\textwidth]{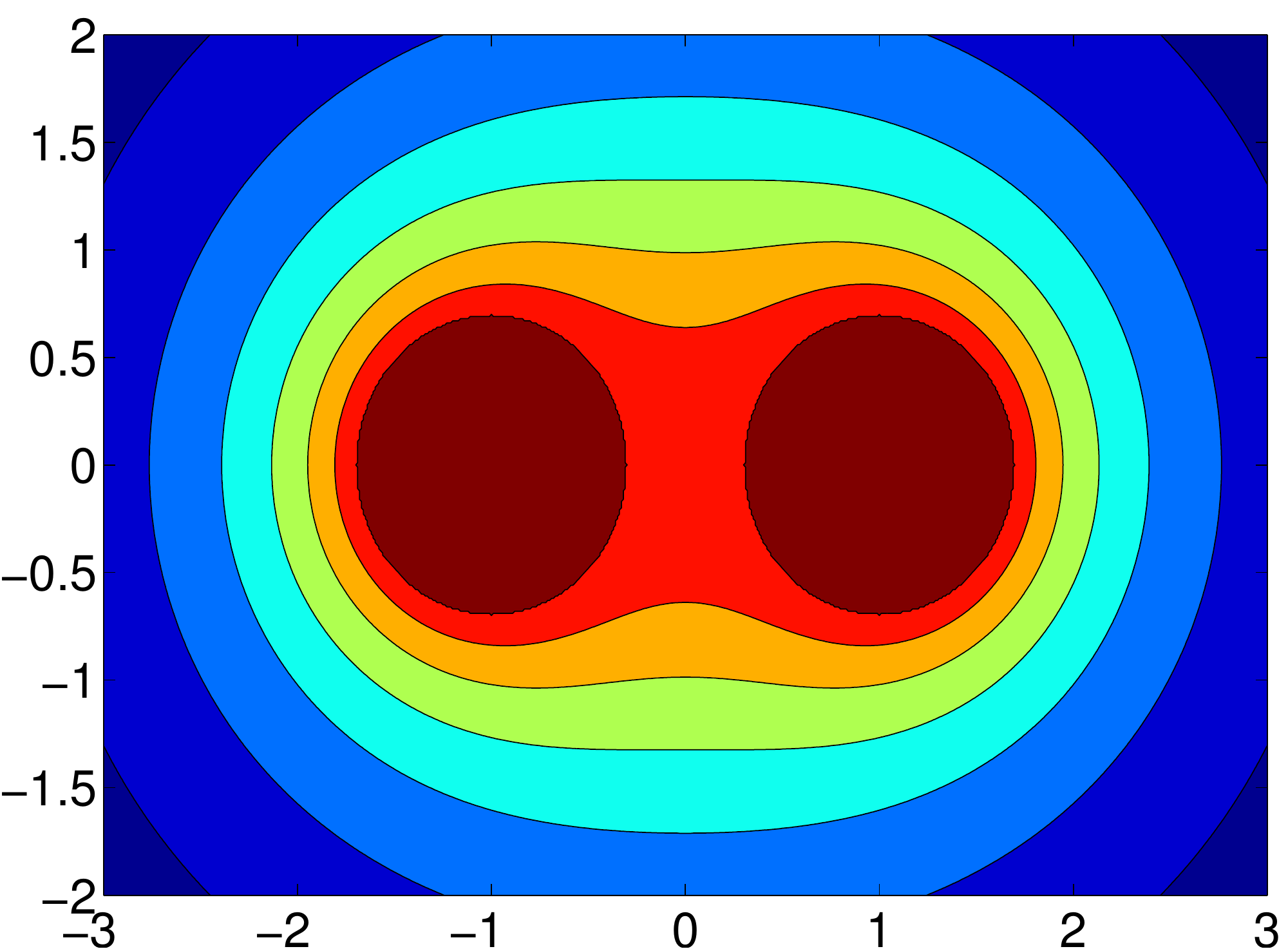}
}
\caption{Asymptotic convergence factor $R_{z_0}(E(1, 0.7))$ as a function of 
$z_0 \in \C$.}
\label{fig:ACF_2disks_3d}
\end{figure}

\subsection{A set with more components}

In this section we will give another illustration of 
Theorem~\ref{thm:relation_fw_f}, starting from a simply connected set of the 
form introduced in~\cite[Theorem~3.1]{KochLiesen2000}; see 
Figure~\ref{fig:set_bw} for an illustration.

\begin{theorem}[{\cite[Theorem~3.1]{KochLiesen2000}}]
\label{thm:bw_set}
Let $\lambda \in \C$ with $\abs{\lambda} = 1$, $\phi \in \: ]0, 2\pi[$ and $R 
\in [1, P[$, where 
\begin{equation*}
P \coloneq \tan(\phi/4) + \frac{1}{\cos(\phi/4)}.
\end{equation*}
Then $\Omega = \Omega(\lambda, \phi, R)$ is the compact set bounded by 
$\widetilde{\psi}( \{ w \in \C : \abs{w} = 1 \} )$, where
\begin{equation*}
\widetilde{\psi}(w) = \frac{ (w - \lambda N) (w-\lambda M)}{ (N-M) w + \lambda 
(NM-1) },
\text{with }
N = \frac{1}{2} \left( \frac{P}{R} + \frac{R}{P} \right),
M = \frac{ R^2 - 1 }{ 2 R \tan(\phi/4) },
\end{equation*}
is a bijective conformal map from the exterior of the unit circle onto 
$\widehat{\C} \backslash \Omega$ and satisfies $\widetilde{\psi}(\infty) = 
\infty$ and $\widetilde{\psi}'(\infty) = t = 1/(N-M) > 0$.
We further have $\lambda \notin \Omega$ and $\{ \lambda e^{i \beta} : \phi/2 
\leq \beta \leq 2 \pi - \phi/2 \} \subseteq \Omega$.
\end{theorem}

Let us consider polynomial pre-images of such sets.  As an example we consider 
the set $\Omega = \Omega(-1, 2 \pi/3, 1.1)$, which satisfies $\Omega = 
\Omega^*$.
Let $E = P^{-1}(\Omega)$ with $P(z) = z^n$; see Figure~\ref{fig:set_bw} 
and~\ref{fig:set_preim_bw} for an illustration with $n = 5$.
By Theorem~\ref{thm:preimage}, the lemniscatic domain corresponding to $E$ is
\begin{equation*}
\cL = \{ w \in \widehat{\C} : \abs{U(w)} = \abs{w^n - t N }^{1/n} > t^{1/n} \},
\end{equation*}
where the logarithmic capacity $t$ of $\Omega$ is given as in
Theorem~\ref{thm:bw_set}.  Since $t N > 0$, the foci of the 
lemniscate are $a_j = e^{2 \pi i \frac{j-1}{n} } (t N)^{1/n}$ for $j = 1, 2, 
\ldots, n$.
By Theorem~\ref{thm:relation_fw_f}, the $(nk)$th Faber--Walsh polynomials for 
$E$ and 
\begin{equation*}
(a_1, a_2, \ldots, a_n, a_1, a_2, \ldots, a_n , \ldots )
\end{equation*}
are given by
\begin{equation}
b_{nk}(z) = t^k F_k(z^n), \label{eqn:bk_Fk_bw}
\end{equation}
where the $F_k$ are the Faber polynomials for $\Omega$, which are explicitly 
known from~\cite[Lemma~4.1]{KochLiesen2000}.  There, the Faber polynomials are 
computed by a recursion involving all previous Faber polynomials.  
The Faber polynomials for $\Omega$ can also be computed by a short (three term) 
recursion; see~\cite{Liesen2001}.
\begin{figure}
\centerline{
\subfigure[Set $\Omega = \Omega(-1, 2 \pi/3, 1.1)$]{\label{fig:set_bw}
\includegraphics[width=0.5\textwidth]{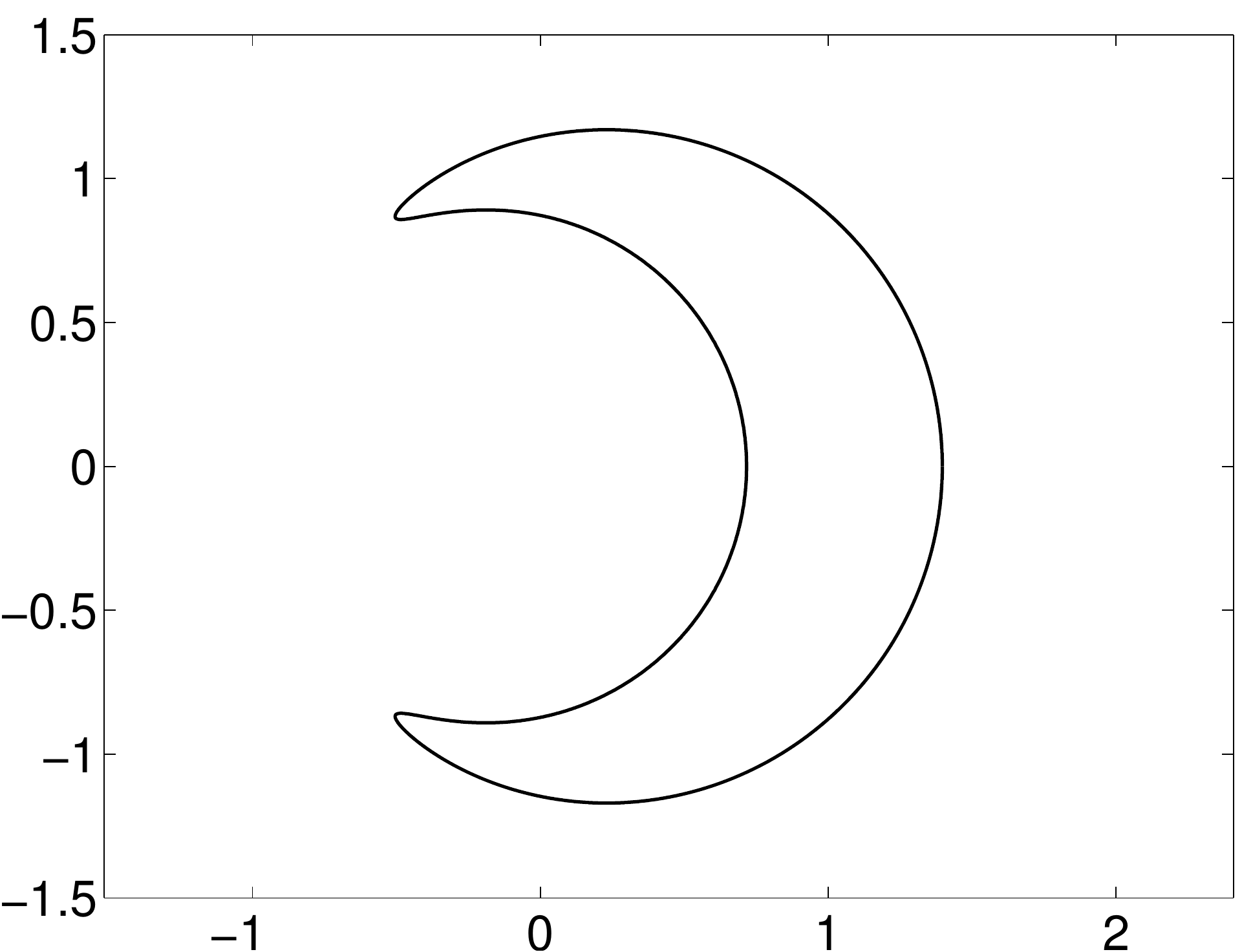}}
\subfigure[$E = P^{-1}(\Omega)$ with $P(z) = z^5$.]{\label{fig:set_preim_bw}
\includegraphics[width=0.5\textwidth]{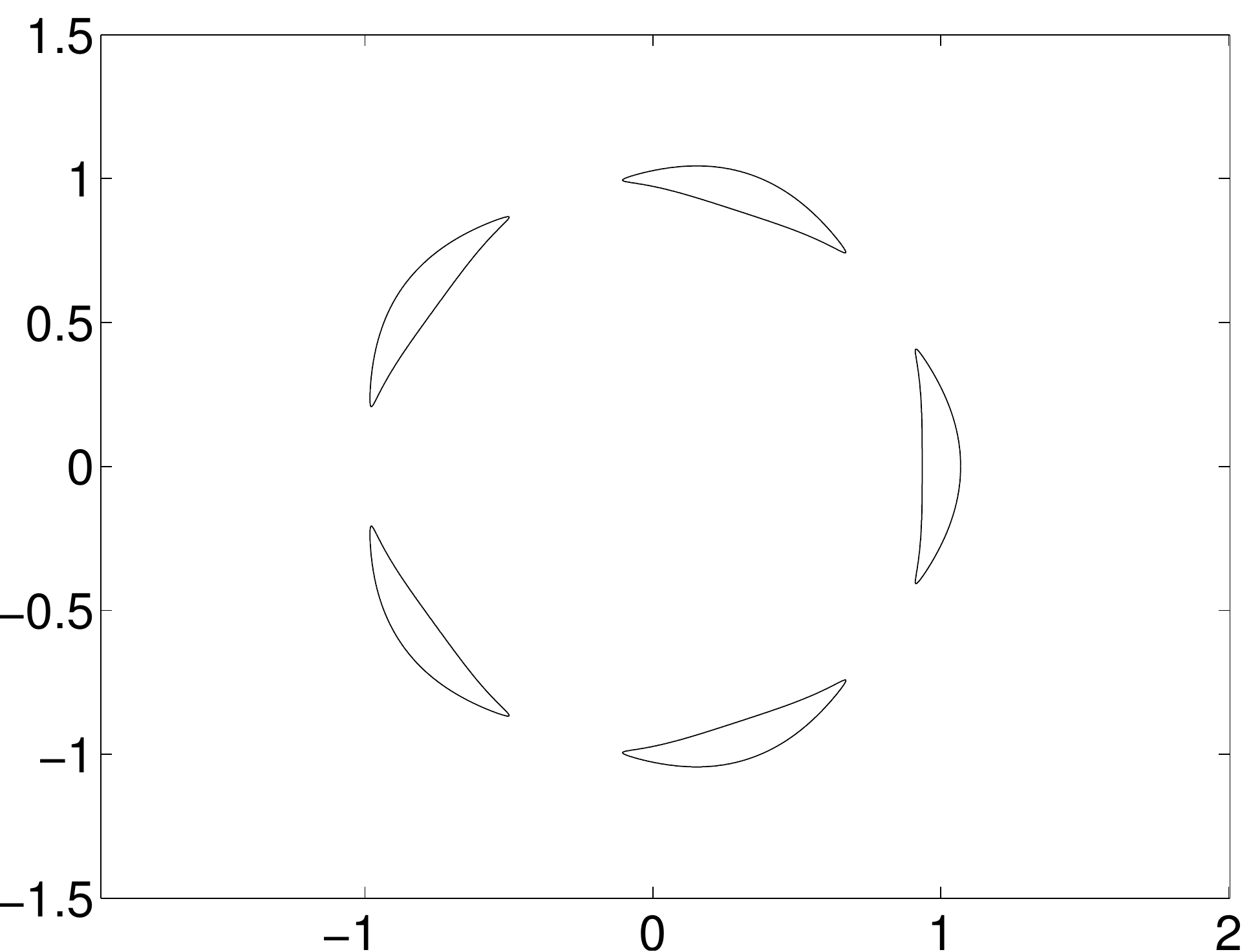}}
}
\centerline{
\includegraphics[width=0.5\textwidth]{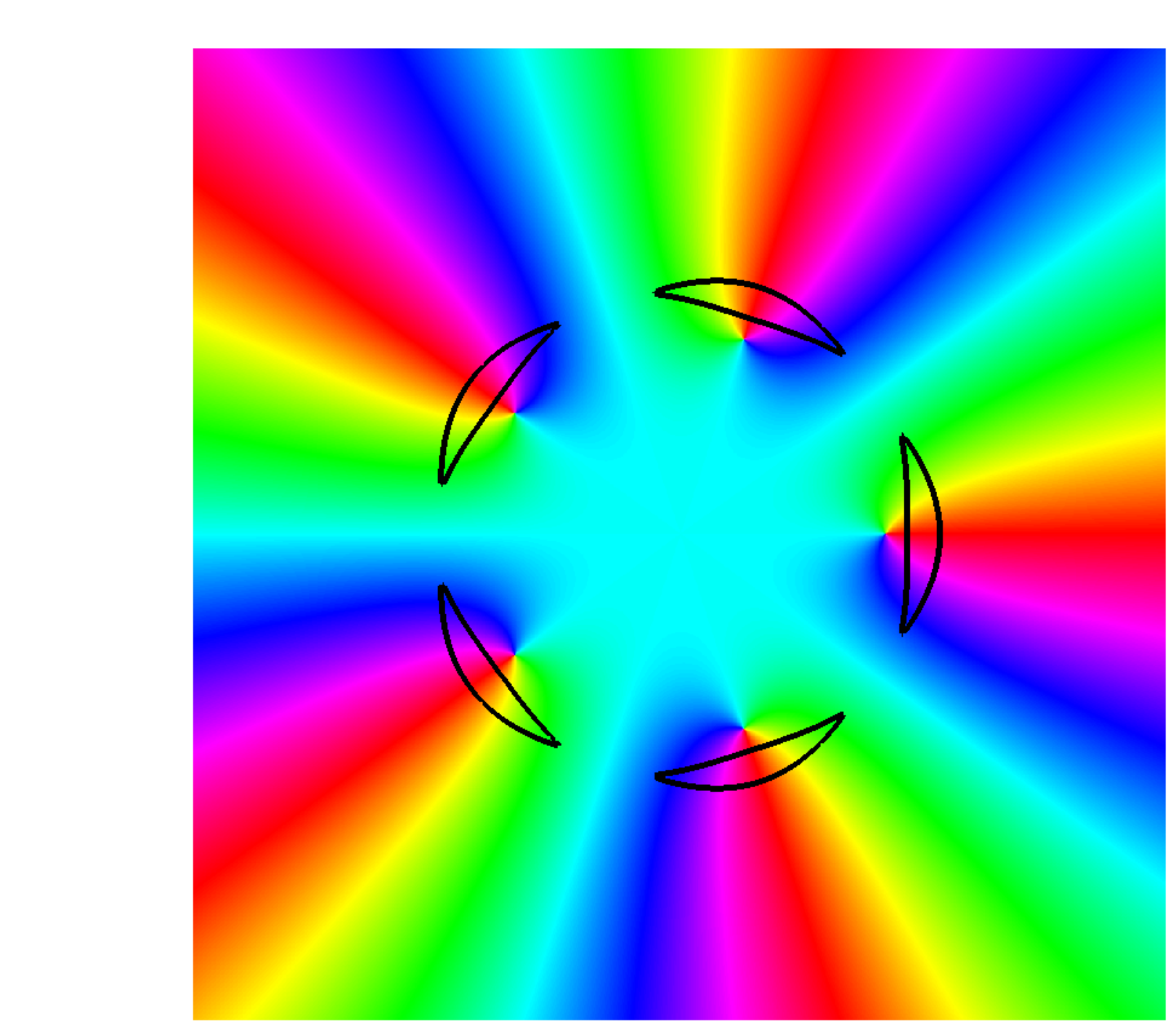}
\includegraphics[width=0.5\textwidth]{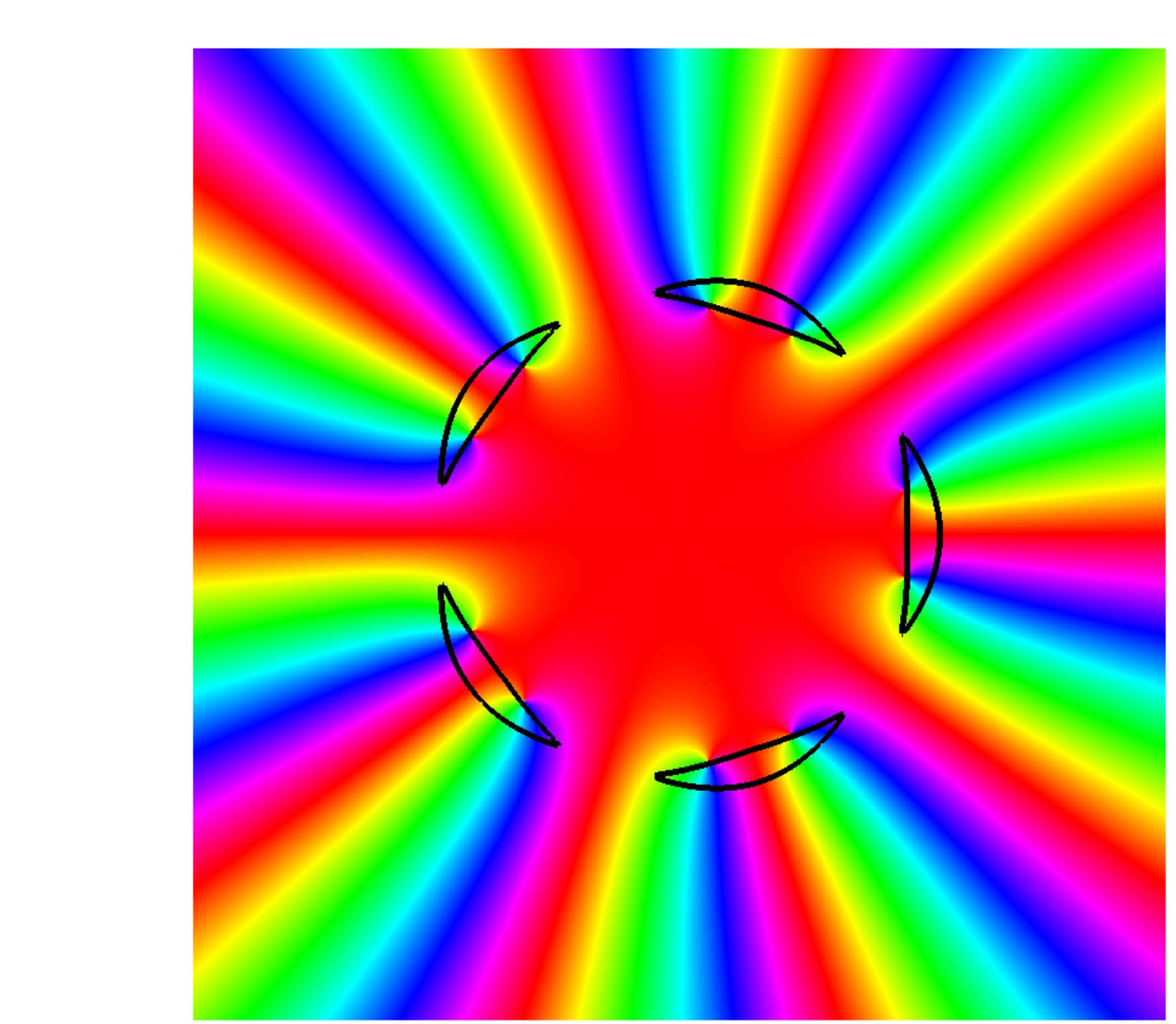}
}
\centerline{
\includegraphics[width=0.5\textwidth]{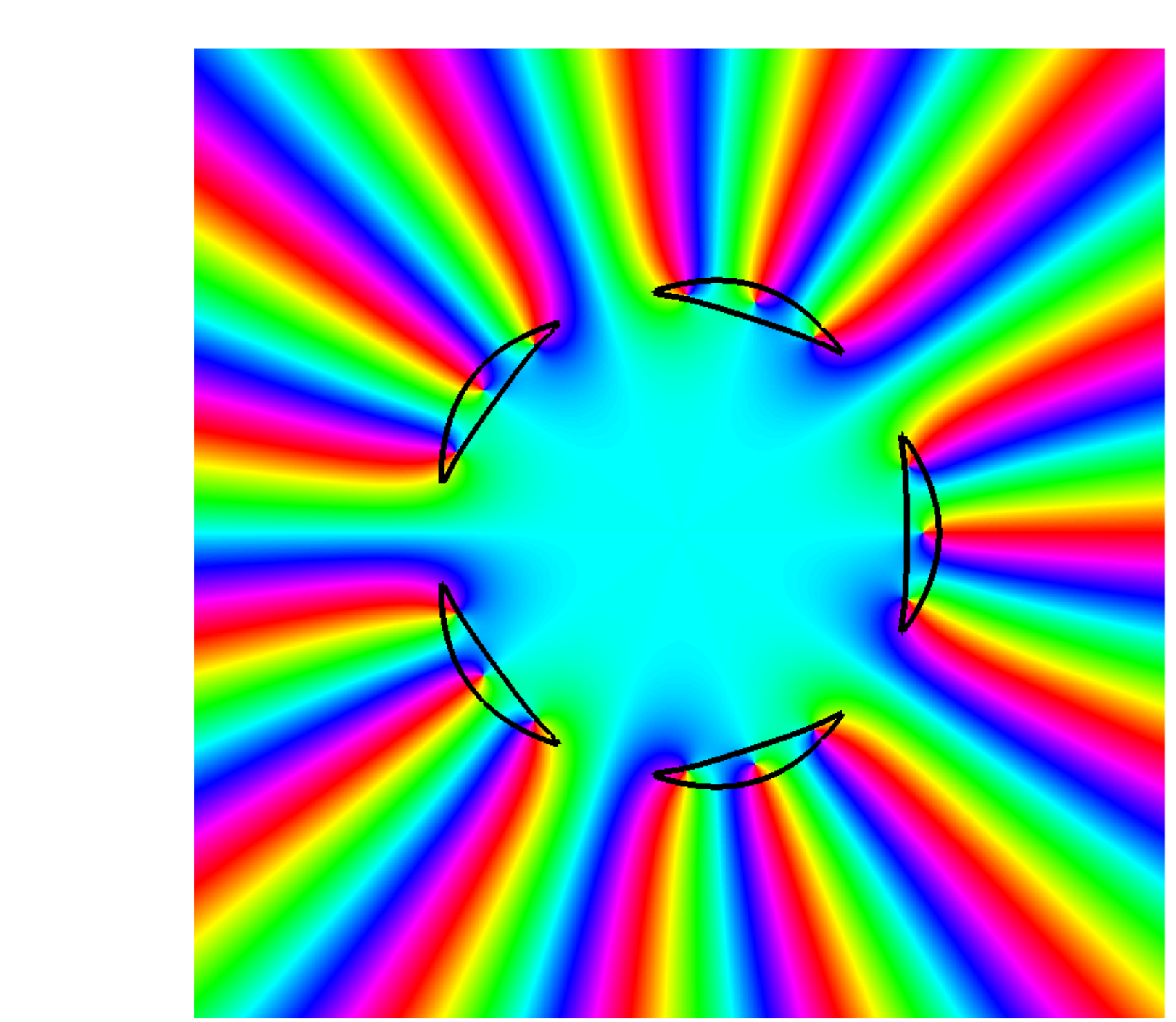}
\includegraphics[width=0.5\textwidth]{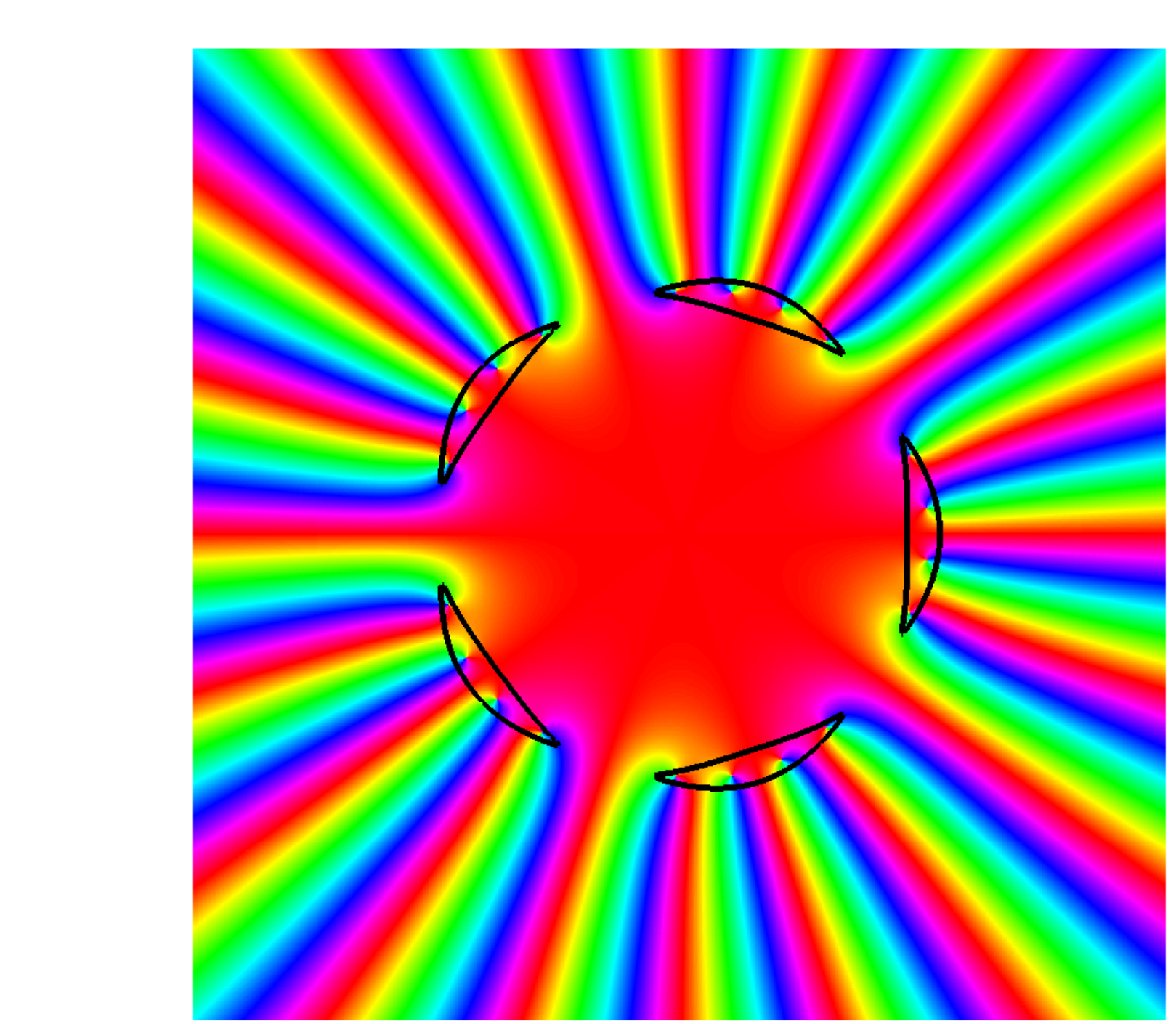}
}
\caption{Set $\Omega = \Omega(-1, 2 \pi/3, 1.1)$, its pre-image $E$, and phase 
portraits of Faber--Walsh polynomials for $E$ of degrees $k 
= 5, 10, 15, 20$.}
\label{fig:bn_preim_bw_pp}
\end{figure}
The ``missing'' Faber--Walsh polynomials can be computed numerically using 
their definition~\eqref{eqn:def_bk}, where we obtain the lemniscatic map $\Phi$ 
of $E$ from $\widetilde{\psi}^{-1}$ by Theorem~\ref{thm:preimage}.  Note that 
$\widetilde{\psi}$ is a composition of two M\"{o}bius transformations and the 
Joukowski map, so that its inverse is easily computable; 
see~\cite[Theorem~3.1]{KochLiesen2000}.

We plot the phase portraits of the Faber--Walsh polynomials $b_{5k}$ for $k = 
1, 2, 3, 4$ in Figure~\ref{fig:bn_preim_bw_pp}.
For degrees $5$ and $10$ we observe that not all zeros of the Faber--Walsh 
polynomial are in $E$, in contrast to the case of the two disks.
This follows from the relation~\eqref{eqn:bk_Fk_bw} and the 
fact that the zeros of the Faber polynomials $F_1$ and $F_2$ for $\Omega$ 
do not lie in $\Omega$; see~\cite{Liesen2000} for details on these Faber 
polynomials.
This also shows that the lower bound in~\eqref{eqn:double_bound_bn} cannot, 
in general, hold for all $k$.
In Figure~\ref{fig:nfw_preim_bw_norm} we plot the values $\norm{b_{5k}}_E / 
\abs{b_{5k}(0)}$ and, for comparison, the values $R_0(E)^k$, where 
$R_0(E) = 0.9803$ (rounded to five digits).

\begin{figure}
\centerline{
\includegraphics[width=0.5\textwidth]{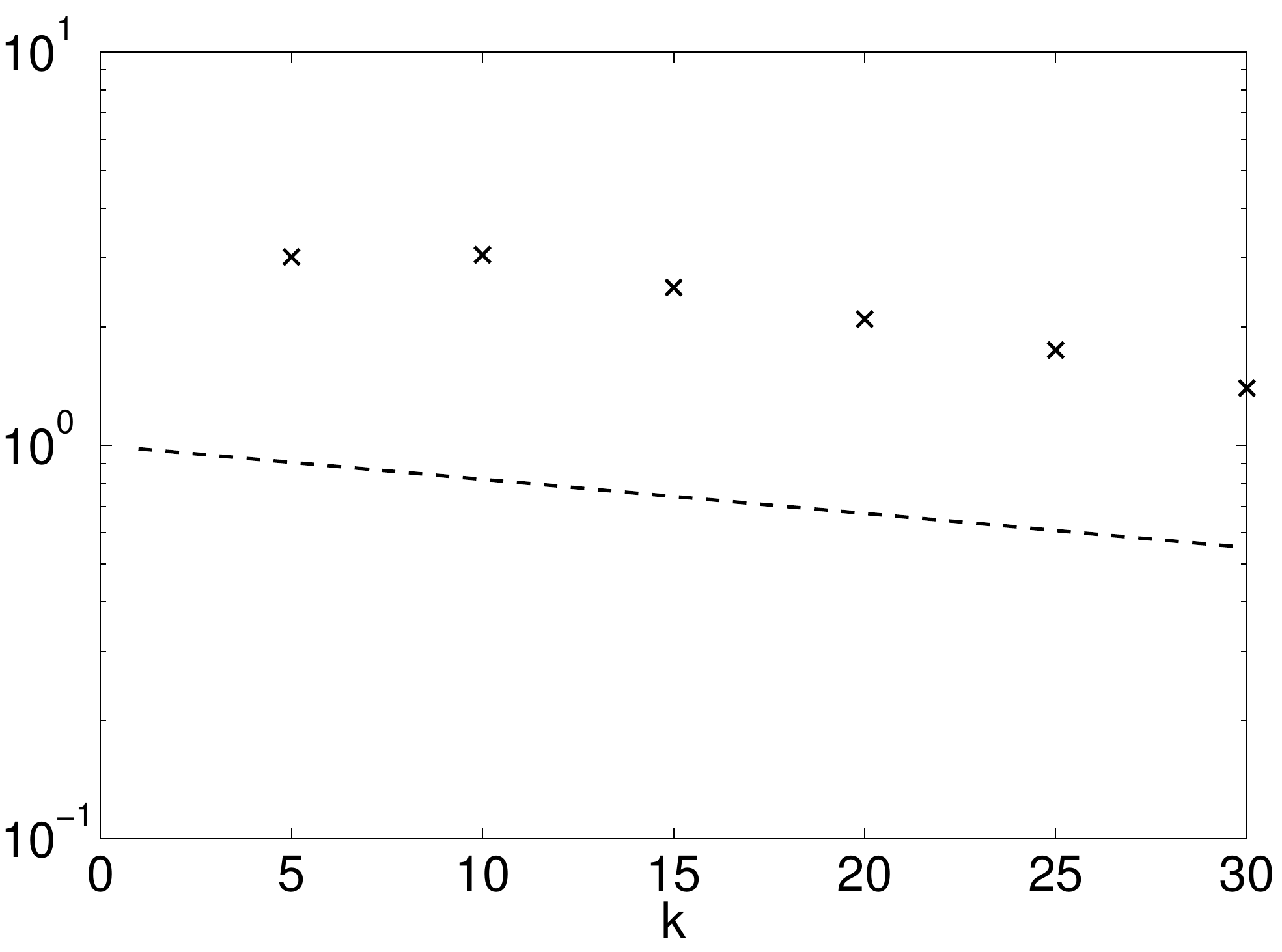}
}
\caption{The values $\frac{\norm{b_k}_E}{ \abs{b_k(0)} }$ and $R_0(E)^k$.}
\label{fig:nfw_preim_bw_norm}
\end{figure}

From~\eqref{eqn:ACF_green} and Theorem~\ref{thm:preimage} the asymptotic 
convergence factor for $E$ and $z_0 \in \C \backslash E$ is given by
\begin{equation*}
R_{z_0}(E) = \frac{\mu}{\abs{U(\Phi(z_0))}}
= \frac{1}{ \abs{\widetilde{\Phi}(z_0^n)}^{1/n} }.
\end{equation*}
Figure~\ref{fig:ACF_preim_bw_3d} shows the asymptotic convergence factor for 
$E$ as a function of $z_0 \in \C$.

\begin{figure}
\centerline{
\includegraphics[width=0.5\textwidth]{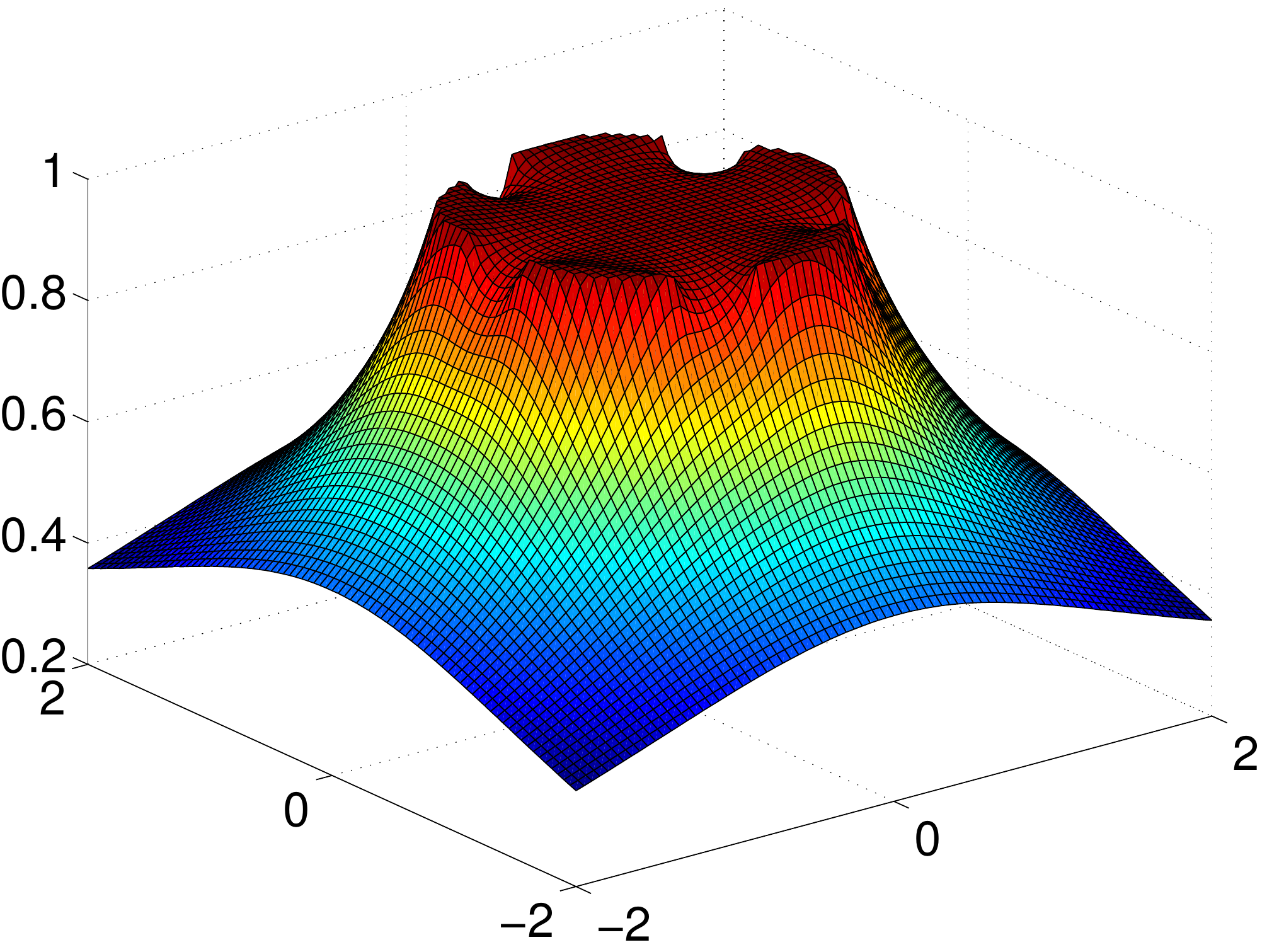}
\includegraphics[width=0.5\textwidth]{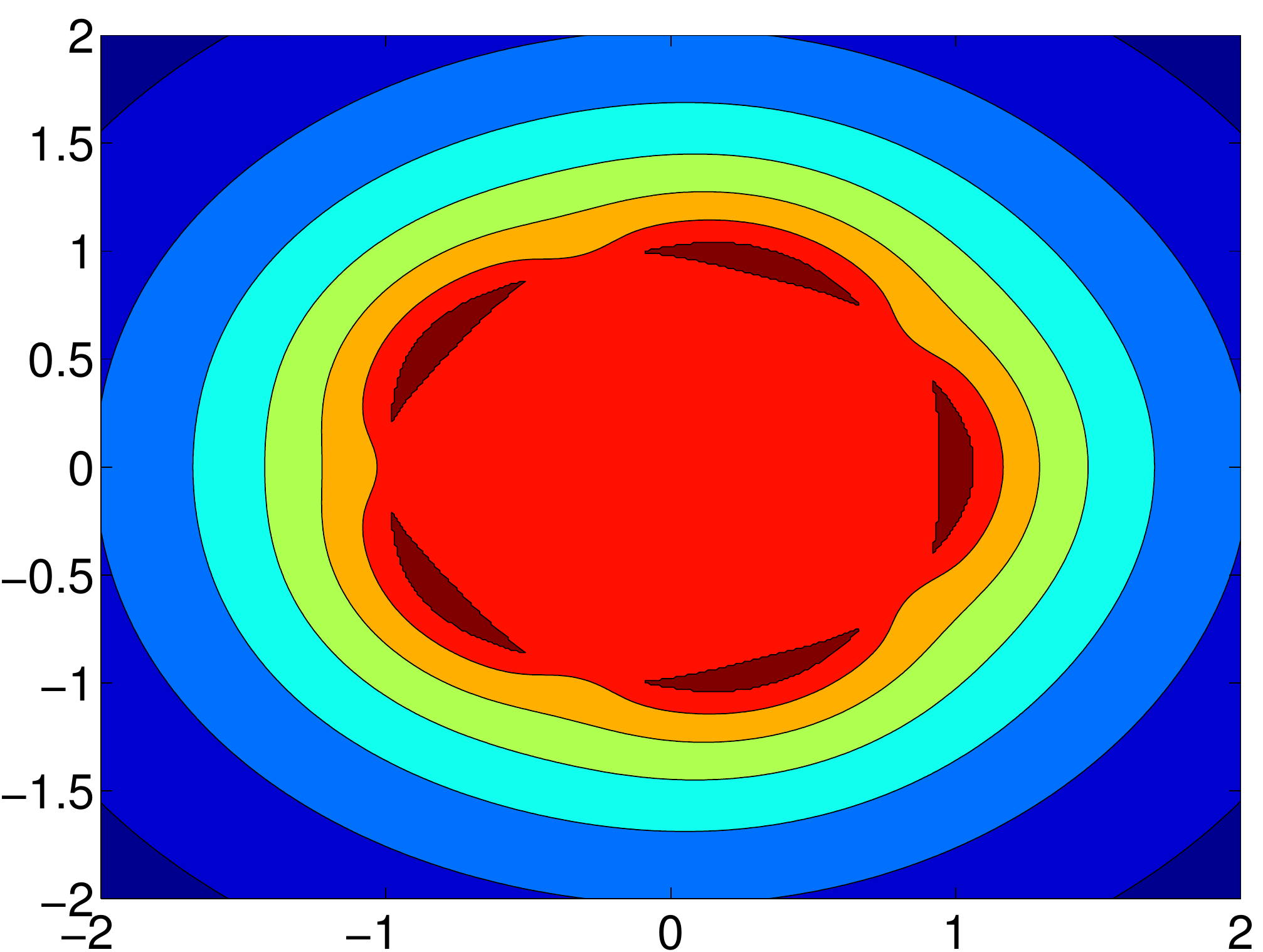}
}
\caption{Asymptotic convergence factor $R_{z_0}(P^{-1}(\Omega))$ as a function 
of $z_0 \in \C$, where $P(z) = z^5$ and $\Omega = \Omega(-1, 2 \pi/3, 1.1)$.}
\label{fig:ACF_preim_bw_3d}
\end{figure}

\paragraph{Acknowledgements}
We thank the anonymous referees for helpful comments.

\bibliographystyle{siam}
\setlength{\bibsep}{1pt}
\bibliography{fw_prop}

\begin{thebibliography}{10}

\bibitem{Achieser1956}
{\sc N.~I. Achieser}, {\em Theory of approximation}, Translated by Charles J.
  Hyman, Frederick Ungar Publishing Co., New York, 1956.

\bibitem{Akhiezer1932_I}
{\sc N.~{Akhiezer}}, {\em {\"Uber einige Funktionen, welche in zwei gegebenen
  Intervallen am wenigsten von Null abweichen. I.}}, {Bull. Acad. Sci. URSS},
  1932 (1932), pp.~1163--1202.

\bibitem{Akhiezer1933_II}
\leavevmode\vrule height 2pt depth -1.6pt width 23pt, {\em {\"Uber einige
  Funktionen, welche in zwei gegebenen Intervallen am wenigsten von Null
  abweichen. II.}}, {Bull. Acad. Sci. URSS}, 1933 (1933), pp.~309--344.

\bibitem{BeckermannReichel2009}
{\sc B.~Beckermann and L.~Reichel}, {\em Error estimates and evaluation of
  matrix functions via the {F}aber transform}, SIAM J. Numer. Anal., 47 (2009),
  pp.~3849--3883.

\bibitem{Curtiss1971}
{\sc J.~H. Curtiss}, {\em Faber polynomials and the {F}aber series}, Amer.
  Math. Monthly, 78 (1971), pp.~577--596.

\bibitem{DriscollTohTrefethen1998}
{\sc T.~A. Driscoll, K.-C. Toh, and L.~N. Trefethen}, {\em From potential
  theory to matrix iterations in six steps}, SIAM Rev., 40 (1998),
  pp.~547--578.

\bibitem{Eiermann1989}
{\sc M.~Eiermann}, {\em On semiiterative methods generated by {F}aber
  polynomials}, Numer. Math., 56 (1989), pp.~139--156.

\bibitem{EiermannLiVarga1989}
{\sc M.~Eiermann, X.~Li, and R.~S. Varga}, {\em On hybrid semi-iterative
  methods}, SIAM J. Numer. Anal., 26 (1989), pp.~152--168.

\bibitem{EiermannNiethammer1983}
{\sc M.~Eiermann and W.~Niethammer}, {\em On the construction of semi-iterative
  methods}, SIAM J. Numer. Anal., 20 (1983), pp.~1153--1160.

\bibitem{EiermannNiethammerVarga1985}
{\sc M.~Eiermann, W.~Niethammer, and R.~S. Varga}, {\em A study of
  semi-iterative methods for nonsymmetric systems of linear equations}, Numer.
  Math., 47 (1985), pp.~505--533.

\bibitem{Ellacott1983}
{\sc S.~W. Ellacott}, {\em Computation of {F}aber series with application to
  numerical polynomial approximation in the complex plane}, Math. Comp., 40
  (1983), pp.~575--587.

\bibitem{Ellacott1986}
\leavevmode\vrule height 2pt depth -1.6pt width 23pt, {\em A survey of {F}aber
  methods in numerical approximation}, Comput. Math. Appl. Part B, 12 (1986),
  pp.~1103--1107.

\bibitem{Faber1903}
{\sc G.~Faber}, {\em \"{U}ber polynomische {E}ntwickelungen}, Math. Ann., 57
  (1903), pp.~389--408.

\bibitem{Fischer1996}
{\sc B.~Fischer}, {\em Polynomial based iteration methods for symmetric linear
  systems}, Wiley-Teubner Series Advances in Numerical Mathematics, John Wiley
  \& Sons, Ltd., Chichester; B. G. Teubner, Stuttgart, 1996.

\bibitem{FischerPeherstorfer2001}
{\sc B.~Fischer and F.~Peherstorfer}, {\em Chebyshev approximation via
  polynomial mappings and the convergence behaviour of {K}rylov subspace
  methods}, Electron. Trans. Numer. Anal., 12 (2001), pp.~205--215.

\bibitem{Grunsky1957}
{\sc H.~Grunsky}, {\em \"{U}ber konforme {A}bbildungen, die gewisse
  {G}ebietsfunktionen in elementare {F}unktionen transformieren. {I}}, Math.
  Z., 67 (1957), pp.~129--132.

\bibitem{Grunsky1957a}
\leavevmode\vrule height 2pt depth -1.6pt width 23pt, {\em \"{U}ber konforme
  {A}bbildungen, die gewisse {G}ebietsfunktionen in elementare {F}unktionen
  transformieren. {II}}, Math. Z., 67 (1957), pp.~223--228.

\bibitem{Grunsky1978}
\leavevmode\vrule height 2pt depth -1.6pt width 23pt, {\em Lectures on {t}heory
  of {f}unctions in {m}ultiply {c}onnected {d}omains}, Vandenhoeck \& Ruprecht,
  G\"ottingen, 1978.

\bibitem{Hasson2003}
{\sc M.~Hasson}, {\em The capacity of some sets in the complex plane}, Bull.
  Belg. Math. Soc. Simon Stevin, 10 (2003), pp.~421--436.

\bibitem{Hasson2007}
\leavevmode\vrule height 2pt depth -1.6pt width 23pt, {\em The degree of
  approximation by polynomials on some disjoint intervals in the complex
  plane}, J. Approx. Theory, 144 (2007), pp.~119--132.

\bibitem{HeuvelineSadkane1997}
{\sc V.~Heuveline and M.~Sadkane}, {\em Arnoldi-{F}aber method for large
  non-{H}ermitian eigenvalue problems}, Electron. Trans. Numer. Anal., 5
  (1997), pp.~62--76.

\bibitem{Jenkins1958}
{\sc J.~A. Jenkins}, {\em On a canonical conformal mapping of {J}. {L}.
  {W}alsh}, Trans. Amer. Math. Soc., 88 (1958), pp.~207--213.

\bibitem{KamoBorodin1994}
{\sc S.~{Kamo} and P.~{Borodin}}, {\em {Chebyshev polynomials for Julia
  sets.}}, {Mosc. Univ. Math. Bull.}, 49 (1994), pp.~44--45.

\bibitem{KochLiesen2000}
{\sc T.~Koch and J.~Liesen}, {\em The conformal ``bratwurst'' maps and
  associated {F}aber polynomials}, Numer. Math., 86 (2000), pp.~173--191.

\bibitem{Landau1961}
{\sc H.~J. Landau}, {\em On canonical conformal maps of multiply connected
  domains}, Trans. Amer. Math. Soc., 99 (1961), pp.~1--20.

\bibitem{Liesen2000}
{\sc J.~Liesen}, {\em On the location of the zeros of {F}aber polynomials},
  Analysis (Munich), 20 (2000), pp.~157--162.

\bibitem{Liesen2001}
\leavevmode\vrule height 2pt depth -1.6pt width 23pt, {\em Faber polynomials
  corresponding to rational exterior mapping functions}, Constr. Approx., 17
  (2001), pp.~267--274.

\bibitem{Markushevich1965}
{\sc A.~I. Markushevich}, {\em Theory of functions of a complex variable.
  {V}ol. {I}}, Prentice-Hall Inc., Englewood Cliffs, N.J., 1965.

\bibitem{Markushevich1967}
\leavevmode\vrule height 2pt depth -1.6pt width 23pt, {\em Theory of functions
  of a complex variable. {V}ol. {III}}, Prentice-Hall, Inc., Englewood Cliffs,
  N.J., 1967.

\bibitem{MoretNovati2001}
{\sc I.~Moret and P.~Novati}, {\em The computation of functions of matrices by
  truncated {F}aber series}, Numer. Funct. Anal. Optim., 22 (2001),
  pp.~697--719.

\bibitem{MoretNovati2001a}
\leavevmode\vrule height 2pt depth -1.6pt width 23pt, {\em An interpolatory
  approximation of the matrix exponential based on {F}aber polynomials}, J.
  Comput. Appl. Math., 131 (2001), pp.~361--380.

\bibitem{NLS2015_numconf}
{\sc M.~M.~S. Nasser, J.~Liesen, and O.~S{\`e}te}, {\em {Numerical computation
  of the conformal map onto lemniscatic domains}}, Comput. Methods Funct.
  Theory,  (2016), pp.~1--27.

\bibitem{Peherstorfer2003}
{\sc F.~Peherstorfer}, {\em Inverse images of polynomial mappings and
  polynomials orthogonal on them}, in Proceedings of the {S}ixth
  {I}nternational {S}ymposium on {O}rthogonal {P}olynomials, {S}pecial
  {F}unctions and their {A}pplications ({R}ome, 2001), vol.~153, 2003,
  pp.~371--385.

\bibitem{PeherstorferSchiefermayr1999}
{\sc F.~Peherstorfer and K.~Schiefermayr}, {\em Description of extremal
  polynomials on several intervals and their computation. {I}, {II}}, Acta
  Math. Hungar., 83 (1999), pp.~27--58, 59--83.

\bibitem{PeherstorferSteinbauer2001}
{\sc F.~Peherstorfer and R.~Steinbauer}, {\em Orthogonal and {$L\sb
  q$}-extremal polynomials on inverse images of polynomial mappings}, J.
  Comput. Appl. Math., 127 (2001), pp.~297--315.

\bibitem{Ran95}
{\sc T.~Ransford}, {\em Potential theory in the complex plane}, vol.~28 of
  London Mathematical Society Student Texts, Cambridge University Press,
  Cambridge, 1995.

\bibitem{Schiefermayr2008_cheb}
{\sc K.~Schiefermayr}, {\em A lower bound for the minimum deviation of the
  {C}hebyshev polynomial on a compact real set}, East J. Approx., 14 (2008),
  pp.~223--233.

\bibitem{Schiefermayr2011}
\leavevmode\vrule height 2pt depth -1.6pt width 23pt, {\em Estimates for the
  asymptotic convergence factor of two intervals}, J. Comput. Appl. Math., 236
  (2011), pp.~28--38.

\bibitem{Schiefermayr2011_minres}
\leavevmode\vrule height 2pt depth -1.6pt width 23pt, {\em A lower bound for
  the norm of the minimal residual polynomial}, Constr. Approx., 33 (2011),
  pp.~425--432.

\bibitem{Sete2013}
{\sc O.~{S{\`e}te}}, {\em {Some properties of Faber--Walsh polynomials}}, ArXiv
  e-prints: 1306.1347v1,  (2013).

\bibitem{SeteLiesen2015}
{\sc O.~S{\`e}te and J.~Liesen}, {\em On conformal maps from multiply connected
  domains onto lemniscatic domains}, Electron. Trans. Numer. Anal., 45 (2016),
  pp.~1--15.

\bibitem{SmirnovLebedev}
{\sc V.~I. Smirnov and N.~A. Lebedev}, {\em Functions of a complex variable:
  {C}onstructive theory}, Translated from the Russian by Scripta Technica Ltd,
  The M.I.T. Press, Cambridge, Mass., 1968.

\bibitem{StarkeVarga1993}
{\sc G.~Starke and R.~S. Varga}, {\em A hybrid {A}rnoldi-{F}aber iterative
  method for nonsymmetric systems of linear equations}, Numer. Math., 64
  (1993), pp.~213--240.

\bibitem{Suetin1998}
{\sc P.~K. Suetin}, {\em Series of {F}aber polynomials}, vol.~1 of Analytical
  Methods and Special Functions, Gordon and Breach Science Publishers,
  Amsterdam, 1998.

\bibitem{Walsh1956}
{\sc J.~L. Walsh}, {\em On the conformal mapping of multiply connected
  regions}, Trans. Amer. Math. Soc., 82 (1956), pp.~128--146.

\bibitem{Walsh1958}
\leavevmode\vrule height 2pt depth -1.6pt width 23pt, {\em A generalization of
  {F}aber's polynomials}, Math. Ann., 136 (1958), pp.~23--33.

\bibitem{Walsh1969}
\leavevmode\vrule height 2pt depth -1.6pt width 23pt, {\em Interpolation and
  approximation by rational functions in the complex domain}, Fifth edition,
  American Mathematical Society, Providence, R.I., 1969.

\bibitem{Wegert2012}
{\sc E.~Wegert}, {\em {V}isual {C}omplex {F}unctions}, Birkh\"auser/Springer
  Basel AG, Basel, 2012.

\bibitem{WegertSemmler2011}
{\sc E.~{Wegert} and G.~{Semmler}}, {\em {Phase plots of complex functions: a
  journey in illustration}}, {Notices Amer. Math. Soc.}, 58 (2011),
  pp.~768--780.

\end{thebibliography}

\end{document}